\documentclass[preprint, bj]{imsart}

\startlocaldefs
\theoremstyle{plain}

\newtheorem{theorem}{Theorem}[section]
\newtheorem{proposition}{Proposition}[section] 
\newtheorem{lemma}[theorem]{Lemma}
\newtheorem{corollary}{Corollary}[section]

\theoremstyle{remark}
\newtheorem{definition}[theorem]{Definition}

\newtheorem{assumption}{Assumption}
\newtheorem{remark}{Remark}[section]
\endlocaldefs

\usepackage{url}
\usepackage{natbib}
\usepackage{amsmath}
\usepackage{amsthm}
\usepackage{amssymb}
\usepackage{amsfonts}
\usepackage{graphicx}
\usepackage{url} 
\usepackage{multirow}
\usepackage{longtable}
\usepackage{csquotes} 
\usepackage{dsfont}
\usepackage{pdflscape}
\usepackage{morefloats}
\usepackage{nccfoots} 
\usepackage{subcaption}
\RequirePackage[colorlinks,citecolor=blue,urlcolor=blue]{hyperref}

\newcommand{\E}{{\operatorname E}}
\newcommand{\Cov}{{\operatorname{Cov}}}
\newcommand{\Var}{{\operatorname{Var}}}
\newcommand{\dcov}{{\operatorname{dcov}}}
\newcommand{\cov}{{\operatorname{cov}}}
\newcommand{\R}{{\mathbb R}}
\newcommand{\N}{{\mathbb N}}

\newcommand{\e}{{\operatorname{e}}}

\newcommand{\Ret}{{\operatorname{Re}}}
\newcommand{\Imt}{{\operatorname{Im}}}

\makeatletter
\newcommand*{\defeq}{\mathrel{\rlap{%
                     \raisebox{0.3ex}{$\m@th\cdot$}}%
                     \raisebox{-0.3ex}{$\m@th\cdot$}}%
                     =}
\makeatother

\makeatletter
\newcommand*{\eqdef}{=\mathrel{\rlap{%
                     \raisebox{0.3ex}{$\m@th\cdot$}}%
                     \raisebox{-0.3ex}{$\m@th\cdot$}}%
                     }
\makeatother
\newcommand{\ignore}[1]{}

\newcommand{\pr}[1]{\left(#1\right)}

\begin{document}

\begin{frontmatter}
\title{Test for independence of long-range dependent time series using distance covariance }
\runtitle{Distance covariance-based test for independence of LRD time series}

\begin{aug}
\author[A]{\fnms{Annika} \snm{Betken}\ead[label=e1]{a.betken@utwente.nl}}
\and
\author[B]{\fnms{Herold} \snm{Dehling}\ead[label=e2]{herold.dehling@rub.de}}
\runauthor{Betken,  Dehling}

\address[A]{Faculty of Electrical Engineering, Mathematics and Computer Science (EEMCS), 
University of Twente,
 \printead[presep={\ }]{e1}}

\address[B]{Faculty of Mathematics,
Ruhr-University Bochum,
 \printead[presep={\ }]{e2}}
\end{aug}

\begin{abstract}
We apply the concept of distance covariance for testing independence of two long-range dependent time series. As test statistic we propose a linear combination of empirical distance cross-covariances. We derive the asymptotic distribution of the test statistic, and we show consistency against a very general class of alternatives. The asymptotic theory developed in this paper is based on a novel non-central limit theorem for stochastic processes with values in an $L^2$-Hilbert space.  This limit theorem is of  general theoretical interest which goes beyond the context of this article. Subject to the dependence in the data, the standardization and the limit distributions of the proposed test statistic vary. Since the limit distributions are unknown, we propose a subsampling procedure to determine the critical values for the proposed test, and we provide a proof for  the validity of subsampling. 
 In a simulation study, we investigate the finite-sample behavior of our test, and we compare its performance to tests based on the empirical cross-covariances. As an application of our results we analyze the cross-de\-pen\-den\-cies between  mean monthly discharges of three rivers.  
\end{abstract}



\end{frontmatter}

\tableofcontents


\section{Introduction}
Given observations $(X_1,Y_1),\ldots, (X_n,Y_n)$ that form the initial segment of a bivariate stationary process $(X_i,Y_i)_{i\geq 1}$, our goal is to test the hypothesis that the processes $(X_i)_{i\geq 1}$ and $(Y_i)_{i\geq 1}$ are independent. We will present a test that is based on a novel measure  of  dependence between time series using empirical distance cross-covariances. Unlike tests based on the empirical covariance, the proposed test is consistent against a very general class of deviations from the hypothesis of independence. 
 We analyze the large-sample behavior of the test for long-range dependent data, and  propose a subsampling procedure to determine critical values for the testing procedure and prove its validity for all measurable  statistics that apply to two long-range dependent time series. 

Classical tests for independence are based on  the empirical covariance as a measure of the degree of dependence between two random variables $X,Y$. Given data $(X_1,Y_1),\ldots, (X_n,Y_n)$, where each pair $(X_i,Y_i)$ has the same joint distribution as $(X,Y)$, the empirical correlation coefficient is defined as 
\[
  r_{X,Y}=\frac{ \sum_{i=1}^n (X_i-\bar{X})(Y_i -\bar{Y}) }{\sqrt{ \sum_{i=1}^n (X_i-\bar{X})^2 \sum_{i=1}^n (Y_i-\bar{Y})^2 } },
\]
where $\bar{X}=\frac{1}{n}\sum_{i=1}^n X_i$ and $\bar{Y}=\frac{1}{n}\sum_{j=1}^n Y_j$.
This statistic can be used to test for marginal independence, i.e. independence of $X_1$ and $Y_1$ in a stationary time series.  Portmanteau-type tests that are able to 
detect  dependence between the processes $(X_i)_{i\geq 1}$ and $(Y_i)_{i\geq 1}$ at arbitrary lags have been developed, e.g., by \cite{haugh:1976} and \cite{shao:2009}. 

It is well-known that the empirical covariance measures only the degree of linear dependence between random variables, while it is  insensitive to nonlinear dependence. As a result, the random variables might be highly dependent although they are uncorrelated. 
In a series of papers, \cite{szekely:2007} and
\cite{szekely:2009, szekely:2012, szekely:2013, szekely:2014} introduced distance covariance and distance correlation as alternative measures of the degree of dependence. Distance covariance of the random variables $X$ and $Y$ is defined as 
\begin{align*}
  \dcov(X,Y) = \iint \big| \varphi_{X,Y}(s,t)-\varphi_X(s)\varphi_Y(t)   \big|^2  w(s,t) \, ds\, dt, 
\end{align*}
 where $\varphi_{X,Y}(s,t)$, $\varphi_X(s)$, $\varphi_Y(t)$ denote the joint and marginal characteristic functions of $X,Y$. The function $w(s,t)$ is a positive weight function; a common choice is $w(s,t) =|s|^{-2}\, |t|^{-2}$.  
Throughout this article we stick to this choice. 
It is easy to see that the random variables $X$ and $Y$ are independent if and only if $\dcov(X,Y)=0$. The empirical distance covariance is defined as 
\begin{align*}
   \dcov_n(X,Y) = \iint \big| \varphi_{X,Y}^{(n)}(s,t)-\varphi_X^{(n)}(s) \, \varphi_Y^{(n)}(t)   \big|^2  w(s,t) \, ds\, dt, 
\end{align*}
where $\varphi_{X,Y}^{(n)}(s,t)$, $\varphi_X^{(n)}(s)$ and $\varphi_Y^{(n)}(t)$ denote the empirical characteristic functions of 
$(X_1,Y_1),\ldots,$ $(X_n,Y_n)$, and the marginal empirical characteristic functions of 
$X_1,\ldots, X_n$ and $Y_1,\ldots,Y_n$. \cite{szekely:2007} derive the large-sample distribution  of $\dcov_n$ for independent pairs $(X_i,Y_i)_{i\geq 1}$. 
\cite{dehling:et.al.:2020} apply the distance covariance to the components of
i.i.d. sequences of pairs of discretized stochastic processes  and
show that the
empirical distance covariance converges to  zero  if and only if the component
processes are independent.
 Under the assumption of absolutely regular processes  \cite{kroll:2020} 
derives the asymptotic distribution of the empirical distance covariance, while 
   \cite{betken:2021} 
 develop a test for independence of two absolutely regular processes and, for this,   
   prove the validity of a
    block bootstrap procedure for the empirical distance covariance.

\cite{zhou:2012} extends the concept of distance correlation to auto-distance correlation of  time series as a tool to explore nonlinear dependence within a time series.  \cite{davis:matsui:mikosch:wan:2018} apply the auto-distance correlation function to stationary multivariate time series in order to measure lagged auto- and cross-dependencies in a time series. Under mixing assumptions, these authors establish asymptotic theory for the empirical auto- and cross-de\-pen\-den\-cies. 
Within machine learning communities testing for independence of observations is often based 
on the Hilbert–Schmidt independence criterion (HSIC); see 
\cite{gretton:2005}, \cite{gretton:2007}, \cite{smola:2007}, \cite{zhang:2008}.
The associated test statistic corresponds to the  maximum mean discrepancy (MMD) of probability distributions, i.e. the difference between embeddings 
of  probability distributions into reproducing kernel Hilbert spaces.
Most notably, \cite{sejdinovic:2013}
show that for a specific choice of kernel function for the HSIC,  distance covariance and MMD  coincide. 
Although most applications of HSIC based testing  are  limited to independent and identically distributed data,  extensions to interdependent observations exist:
 \cite{zhang:2008}
estimate the MMD by
 a fourth-order $U$-statistic and establish asymptotic normality
 of this statistic for stationary mixing sequences.
The article focuses an applications of this result to time series clustering and segmentation and refers to
 \cite{borovkova:2001} for a formal  proof.
 For   random processes satisfying $\phi$- and $\beta$-mixing conditions, 
\cite{chwialkowski:2014} derive the asymptotic distribution of the  HSIC statistic from established theory on $U$-statistics.
\cite{wang:2021} apply the HSIC  to test for independence of innovations of two multiviariate time series  and, for this,  derive its asymptotic distribution under  $\beta$-mixing assumptions on the individual time series.

In the present paper, we initiate the study of distance covariance for long-range dependent processes. Such processes, also known as long memory processes, are commonly used as models for random phenomena that exhibit dependence at all scales,  slow decay of correlations and non-standard scaling behavior. Such phenomena occur, e.g., in hydrological and financial data, and are not captured by common time series models such as ARMA processes; see e.g. \cite{mandelbrot:1982, mandelbrot:1997}.  \cite{pipiras:taqqu:2017} present stochastic models and probabilistic theory for long-range dependent processes. Statistical methods for long-range dependent processes are presented in \cite{beran:feng:ghosh:kulik:2013} and in \cite{surgailis:2012}. Our mathematical analysis is based on novel theory for Hilbert space-valued long-range dependent processes which we apply to the empirical characteristic functions. Our results show that the large-sample behavior of the empirical distance covariance of long-range dependent data differs markedly from independent and short-range dependent data.

In order to detect possible cross-dependencies between  time series $X=(X_i)_{i\geq 1}$ and $Y = (Y_i)_{i\geq 1}$, we study the empirical distance cross-covariance function
\begin{align*}
  \dcov(X,Y;h)&=  \dcov(X_1,Y_{1+h}) \\
  &=\iint \big| \varphi_{X,Y; h}(s,t)-\varphi_X(s)\varphi_Y(t)   \big|^2  w(s,t) \, ds\, dt,
\end{align*}
where $ \varphi_{X,Y; h}(s,t)= \E\left( e^{i (sX_1 + t Y_{1+h}) }\right)$,
and its empirical analogue
\[
  \dcov_n(X,Y;h)=  \iint \big| \varphi_{X,Y;h}^{(n)}(s,t)-\varphi_X^{(n)}(s)\, \varphi_Y^{(n)}(t)   \big|^2  w(s,t) \, ds\, dt,
\]
where 
$\varphi_{X,Y; h}^{(n)} (s,t) =\frac{1}{n} \sum_{j=1}^{n-h} e^{i (sX_j + t Y_{j+h}) }$  is the joint empirical characteristic function of the pairs $(X_1,Y_{1+h}), \ldots, (X_{n-h},Y_n)$ .
 We  determine the joint large sample distribution of the empirical distance cross-covariances at various lags $h$.  Given a summable weight sequence $(a_k)_{k\geq 0}$, we propose 
a test for independence of  the long-range dependent processes $(X_i)_{i\geq 1}$ and $(Y_i)_{i\geq 1}$ using the
linear combination of empirical distance cross-covariances 
$  \sum_{h=0}^\infty a_h \, \dcov_n(X,Y;h) $ as test statistic.
  We  study the asymptotic behavior of this test and show that it is
consistent against a very general class of alternatives, namely against all alternatives where $X_1$ and $Y_h$ are dependent for some lag $h$.

Section \ref{sec:main} contains the main theoretical results of our work:
Section \ref{subsec:notation} introduces the (empirical)
distance cross-covariance as an element of an $L^2$-Hilbert space.
Section \ref{subsec:main} establishes the testing procedure for deciding on whether two time series are independent based on the empirical distance cross-covariances and an approximation of the test statistics distribution by  subsampling. 
Along the way, the
 asymptotic distribution of the empirical distance cross-covariance function, and of the proposed test statistic are derived under the test's hypothesis of independent data-generating processes.
 Under the weaker assumption of a stationary, ergodic bivariate data-generating process, consistency of the empirical distance cross-covariance and of the proposed test  is established. Moreover, we prove the validity of the corresponding subsampling procedure
as approximation of the distribution of any measurable function of two independent, stationary LRD time series, i.e.
  in a setting that applies to the considered situation, but may be of interest in other contexts, as well. 
 As basis for deriving the asymptotic distribution of the  distance cross-covariances of subordinated Gaussian processes, Section \ref{subsec:CLT} 
establishes a non-central limit theorem for processes with values in the corresponding $L^2$-Hilbert space. 
The limit theorem is  not specially geared to applications of distance cross-covariance functions. It is thus of particular and independent interest, and, therefore, considered separately.
We assess
the finite sample performance of a hypothesis test based on the
distance covariance through simulations in Section \ref{subsec:simulations}. In particular, we compare its finite sample
performance to that of a test based on the empirical covariance.
For this purpose, we also establish convergence results for this
dependence measure. Different dependencies between time series
are considered for a comparison between the two testing procedures. It turns out that only linear dependence is better
detected by a test based on  the empirical covariance, while all other dependencies are
better detected by a test based on the empirical distance covariance. An analysis with regard to
cross-dependencies between the mean monthly discharges of three
different rivers in Section \ref{sec:data} provides an application of the theoretical results
established in this article.

\section{Main results}\label{sec:main}

Before stating the main theoretical results of our work in Section \ref{subsec:main},
the following subsection (Section \ref{subsec:notation}) establishes the (empirical)
distance cross-covariance as an element of an $L^2$-Hilbert space. Moreover, it motivates 
the consideration of a non-central limit theorem for processes with values in the corresponding  space (see Section \ref{subsec:CLT}) in this context.

\subsection{Basic notations and outline of approach}\label{subsec:notation}

In this paper, we assume that time series are realizations of  stationary subordinated Gaussian processes $(X_i)_{i\geq 1}$ ,  i.e. we assume that there exists a Gaussian process $(\xi_i)_{i\geq 1}$ and a measurable function $G:\R\rightarrow \R$ such that
$X_i=G(\xi_i)$ for all $i\geq 1$. This class of processes has been widely studied in the literature; see, e.g. \cite{beran:feng:ghosh:kulik:2013}. For any particular distribution function $F$, one can find a transformation $G$ such that $X_i$ has the distribution $F$. Moreover, there exist algorithms for generating Gaussian processes that, after suitable transformation, yield subordinated Gaussian processes with marginal distribution $F$ and a predefined covariance structure; see \cite{pipiras:taqqu:2017}. 

We investigate the distance covariance for long-range dependent (LRD) processes $(X_i)_{i\geq 1}$, which are characterized by a slow decay of autocorrelations. We specifically assume that the autocorrelation function satisfies
\[
  \rho(k)=\Cov(X_1, X_{1+k}) \sim k^{-D} L(k), \mbox{ as } k \rightarrow \infty,
\]
with $D\in (0,1)$ for some slowly  varying function $L$. We refer to $D$ as the long-range dependence (LRD) parameter. 
Under certain assumptions, subordinated Gaussian processes exhibit long-range  dependence, if the underlying Gaussian process $(\xi_k)_{k\in \N}$ is long-range dependent. Specifically, assume that $\Cov(\xi_1,\xi_{1+k})\sim k^{-D} L(k)$, as $k\rightarrow \infty$, for some constant $D\in (0,1)$ and some slowly varying function $L$. Let $\varphi$ denote the density of a standard normal distribution, and assume that $G\in L^2(\R,\varphi(x)\, dx)$ is a function with Hermite rank 
\[
  r:=\min\{ k\geq 1: J_k(G)=0\},
\]
where $J_k(G)=E(G(X)H_r(X))$, and where $H_r(x)$ denotes the $r$-th order Hermite polynomial. Then 
\[
  \Cov(G(\xi_1),G(\xi_{1+k})) \sim J_r^2(G) \, r!\, k^{-Dr} L^r(k), \mbox{ as } k \rightarrow \infty.
\]
Hence, the subordinated Gaussian time series $(G(\xi_k))_{k\in \N}$ is long-range dependent with LRD parameter $D_G:=D\, r$ and slowly varying function $L_G(k)=J_r^2(G) \, r!\, L^r(k)$ whenever $D\, r<1$. 
 
In the literature one finds two approaches to the study of the asymptotic behavior of the distance covariance.  The first approach is based on a representation of the empirical distance covariance as a $V$-statistic, that was established by 
\cite{szekely:2007}, and later extended by \cite{lyons:2013}  to general metric spaces. This approach makes it possible to use existing limit theorems for $V$-statistics, available both for i.i.d. as well as for short-range dependent data. Since much less is known about $V$-statistics for long-range dependent data, we use an alternative approach based on  the representation of the empirical distance covariance as the square norm of the difference between the joint empirical characteristic function and the product of the marginal empirical characteristic functions in the complex Hilbert space $L^2(\R^2,w(s,t)\, ds\, dt)$ equipped with the norm
\[
 \|f\|_2=\Big(\iint |f(s,t)|^2 w(s,t) \, ds\, dt \Big)^{1/2}. 
\] 
By definition, we obtain 
\begin{equation}
\dcov_n(X,Y; h) = \big\| \varphi_{X,Y; h}^{(n)}(s,t) -\varphi_X^{(n)}(s)\, \varphi_Y^{(n)}(t)\big\|^2,
\label{eq:basic}
\end{equation}
where $\varphi_{X,Y; h}^{(n)}(s,t)$, $\varphi_X^{(n)}(s)$, and $\varphi_Y^{(n)}(t)$  denote 
the joint and the marginal empirical characteristic functions of the data $(X_1,Y_1),\ldots, (X_n,Y_n)$ defined as
\begin{align*}
& \varphi_{X,Y; h}^{(n)} (s,t) = \frac{1}{n} \sum_{j=1}^{n-h} e^{i(s\,X_j+t\, Y_{j+h})}, \\
& \varphi_X^{(n)}(s) = \frac{1}{n} \sum_{j=1}^n e^{isX_j}, \
 \text{and} \ \
  \varphi_Y^{(n)}(s) = \frac{1}{n} \sum_{j=1}^n e^{isY_j}.
\end{align*}

By the representation \eqref{eq:basic}, the asymptotic distribution of the empirical distance cross-covariance can be obtained from the asymptotic distribution of the process
$ \varphi_{X,Y;h}^{(n)}(s,t) -\varphi_X^{(n)}(s)\, \varphi_Y^{(n)}(t)$. In order to analyze this process, we make use of the following decomposition:
\begin{align}\label{eq:decomposition}
  &\varphi_{X,Y;h}^{(n)}(s,t) -\varphi_X^{(n)}(s)\, \varphi_Y^{(n)}(t) \\
   =& -\big(\varphi_X^{(n)}(s) -\varphi_X(s)\big) \big( \varphi_Y^{(n)}(t) -\varphi_Y(t)  \big) \nonumber \\
   &  +\frac{1}{n}\sum_{j=1}^{n-h} \big(\exp(isX_j)-\varphi_X(s)\big)\big(\exp(itY_{j+h})-\varphi_Y(t)\big)\notag\\
   & -\varphi_X(s)\frac{1}{n}\sum_{j=1}^{h}\exp(itY_{j}) -\varphi_Y(t)\frac{1}{n}\sum_{j=n-h+1}^{n}\exp(itX_{j})\notag\\
      &-\frac{h}{n}\varphi_X(s)\varphi_Y(t).   \notag
\end{align}
Note that the last three summands on the right-hand side of the above identity are $o(n^{-\gamma})$ for any $\gamma<1$. As a result, the asymptotic distribution of the empirical distance cross-covariance function is determined by the first two summands.
In the following sections, 
these two summands will be considered separately. For an analysis of the first summand, we make use of limit theorems for Hilbert space-valued random variables that we develop in this paper. For this, we consider $\varphi_X^{(n)}(s)-\varphi_X(s)$ and  $\varphi_Y^{(n)}(t)-\varphi_Y(t)$ as elements of  $L^2(\R,w(s)\, ds)$, where $w(s)=1/s^2$, and $(\varphi_X^{(n)}(s)-\varphi_X(s))(\varphi_Y^{(n)}(t)-\varphi_Y(t))$ 
as an element of $L^2(\R^2,w(s,t)\, ds\, dt)$. The following lemma provides theoretical justification for these considerations. 

\begin{lemma}\label{lemma:existence_1}
Let $(X_i)_{i\geq 1}$ and $(Y_i)_{i\geq 1}$ with $\E \left|X_1\right|<\infty$ and $\E \left|Y_1\right|<\infty$ be  stationary processes.
Then, it holds
that 
\begin{align*}
\int_{\mathbb{R}}\left|\varphi_X^{(n)}(s)-\varphi_X(s)\right|^2w(s)ds<\infty, \ \int_{\mathbb{R}}\left|\varphi_Y^{(n)}(t)-\varphi_Y(t)\right|^2w(t)dt<\infty,
\intertext{and}
\int_{\mathbb{R}}\int_{\mathbb{R}}\left|\left(\varphi_Y^{(n)}(t)-\varphi_Y(t)\right)\left(\varphi_X^{(n)}(s)-\varphi_X(s)\right)\right|^2w(s, t)ds dt<\infty,
\end{align*}
where  $w(s,t) =c|s|^{-2}\, c|t|^{-2}$ for some constant $c$.
\end{lemma}

The proof of Lemma \ref{lemma:existence_1} is based on arguments that have been established in \cite{szekely:2007}. It can be found in the supplement.

\subsection{Main theorems }\label{subsec:main}

Based on observations $X_1, \ldots, X_n$ and $Y_1, \ldots, Y_n$ stemming from real-valued time series $(X_i)_{i\geq 1}$ and $(Y_i)_{i\geq 1}$,
our goal
is to decide on the testing problem
\begin{align*}
&\text{$H_0$:  $(X_i)_{i\geq 1}$ and  $(Y_i)_{i\geq 1}$ are independent,} \\
&\text{$H_1$:  $(X_i)_{i\geq 1}$ and  $(Y_i)_{i\geq 1}$ are dependent.} 
\end{align*}
As test statistic
 we propose a linear combination of empirical distance cross-covariances, i.e. 
 \begin{align*}
  \sum_{h=0}^\infty a_h \, \dcov_n(X,Y;h),
 \end{align*}
 where $(a_h)_{h\geq 0}$ is a summable, real-valued sequence of weights.

This section establishes a testing procedure  based on this statistic and  provides theoretical verification for its validity.
More precisely, we derive the test statistic's
 asymptotic distribution  under $H_0$ and, since the corresponding limit distribution is unknown, establish a subsampling procedure to determine critical values for a test decision. 
Moreover, 
 we show  that the test statistic diverges to $\infty$ under $H_1$ thereby establishing consistency  of the proposed test. 
In both cases, under $H_0$ and under $H_1$, our results 
are based on corresponding limit theorems for the empirical distance cross-covariance. Accordingly, the former are preceded by the latter.
 
Since we can represent the empirical distance cross-covariance as
\begin{align*}
\dcov_n(X, Y; h)
=\int_{\mathbb{R}}\int_{\mathbb{R}}\left|\varphi_{X, Y; h}^{(n)}(s, t)-\varphi_X^{(n)}(s)\varphi_Y^{(n)}(t)\right|^2w(s,t)dsdt, 
\end{align*}
we derive its asymptotic distribution from  a limit theorem
for
$\varphi_{X, Y; h}^{(n)}(s, t)-\varphi_X^{(n)}(s)\varphi_Y^{(n)}(t)$
as a random object taking values in  $L^2(\mathbb{R}^2, w(s,t)dsdt)$.

\begin{theorem}\label{thm:main_result}
Let $X_i=G_1(\xi_i)$, $i\geq 1$, and $Y_i=G_2(\eta_i)$, $i\geq 1$, where $(\xi_i)_{i\geq 1}$ and $(\eta_i)_{i\geq 1}$ are two independent,  stationary, long-range dependent   Gaussian 
processes with $\E\pr{\xi_1}=\E\pr{\eta_1}=0$, $\Var\pr{\xi_1}=\Var \pr{\eta_1}=1$, 
$\rho_{\xi}(k)=\Cov(\xi_1, \xi_{1+k})=k^{-D_{\xi}}L_{\xi}(k) \text{ and } \rho_{\eta}(k)=\Cov(\eta_1, \eta_{1+k})=k^{-D_{\eta}}L_{\eta}(k)$
for 
 $D_{\xi}, D_{\eta} \-\in (0, 1)$ and 
 slowly varying functions $L_{\xi}$ and $L_{\eta}$. Assume that $\E|X_1|<\infty$ and  $\E|Y_1|<\infty$.
\begin{itemize}
\item[(i)] If  $D_{\xi}, D_{\eta}
\in (\frac{1}{2},1)$, it holds that
 \begin{align*}
\sqrt{n}\pr{\varphi_{X, Y; h}^{(n)}(s, t)-\varphi_X^{(n)}(s)\varphi_Y^{(n)}(t)}
\overset{\mathcal{D}}{\longrightarrow}Z(s, t),
\end{align*}
where $\left(Z(s, t)\right)_{s, t\in \mathbb{R}}$ is a complex-valued Gaussian process  with location parameter $\mu=0$, covariance matrix
\begin{align}\label{eq:complex_normal_parameters}
\Gamma_{s,t, s', t'}&=\Cov\pr{Z(s, t), Z(s', t')}=\E\pr{Z(s, t) \overline{Z(s', t')}}\notag\\
&=\sum_{k=-\infty}^{\infty}
\E \pr{f_{s, t}(X_1, Y_1)\overline{f_{s', t'}(X_{k+1}, Y_{k+1})}},
\intertext{and relation matrix}
C_{s,t,  s', t'}&=\Cov\pr{Z(s, t), \overline{Z(s', t')}}=\E\pr{Z(s, t) Z(s', t')}\notag\\
&=\sum_{k=-\infty}^{\infty}
\E \pr{f_{s, t}(X_1, Y_1)f_{s', t'}(X_{k+1}, Y_{k+1})},
\end{align}
where
\begin{align*}
f_{s, t}(X_j, Y_j)\defeq \left(\exp(isX_j)-\varphi_X(s)\right)\left(\exp(itY_j)-\varphi_Y(t)\right). \notag
\end{align*}
\item[(ii)]  If $G_1=G_2=\text{id}$ and $D_{\xi}+D_{\eta}<1$, it holds
 that
 \begin{multline*}
n^{\frac{D_{\xi}+D_{\eta}}{2}}L_{\xi}^{-\frac{1}{2}}(n)L_{\eta}^{-\frac{1}{2}}(n)\pr{\varphi_{X, Y; h}^{(n)}(s, t)-\varphi_X^{(n)}(s)\varphi_Y^{(n)}(t)}
\overset{\mathcal{D}}{\longrightarrow}\\
\int_{\left[-\pi, \pi\right)^2}\left[\left(\frac{\e^{ix}-1}{ix}\right)\left(\frac{\e^{iy}-1}{iy}\right)-\frac{\e^{i(x+y)}-1}{i(x+y)}\right]
Z_{X}(dx)Z_{Y}(dy)\\
\times st\exp\left(-\frac{s^2+t^2}{2}\right),
 \end{multline*}
where $Z_{X}$ and $Z_{Y}$ are  random spectral measures defined  subsequently by \eqref{eq:random_spectral_measures}.
\end{itemize}
\end{theorem}

\begin{remark}
\begin{enumerate}
\item[(i)] Note that the results stated in Theorem \ref{thm:main_result} differ mar\-ked\-ly, depending on whether $D_\xi\geq 1/2$ and $D_\eta \geq 1/2$, or $D_\xi+D_\eta<1$. Both, the normalization as well as the limit distribution, are completely different for the two cases. In addition, in the case when $D_\xi+D_\eta<1$, our results only cover Gaussian processes, while in the other case, we can also treat subordinated Gaussian processes. Moreover,  note that Theorem~\ref{thm:main_result} does not cover the case  $D_\xi+D_\eta >1$, but requires either $D_\xi<1/2$ or $D_\eta<1/2$. 
\item[(ii)] In real-life data, one  typically encounters LRD coefficients that are larger than $1/2$, which corresponds to  Hurst coefficients smaller than $0.75$. Such data is covered by part (i) of Theorem \ref{thm:main_result}. Notably, this part of Theorem \ref{thm:main_result} not only applies to Gaussian time series, but to time series allowing for a representation as a subordinated Gaussian process in general.
\end{enumerate}
\end{remark}

An outline of a step-by-step proof 
of Theorem \ref{thm:main_result} through a number of auxiliary results is given in Section \ref{sec:outline_of_proofs}.

As an immediate consequence  of Theorem \ref{thm:main_result}, an application of the continuous mapping theorem establishes the limit distribution of the distance cross-covariance:
\begin{corollary}\label{cor:asymp_distr_dcov}
 Let $X_i=G_1(\xi_i)$, $i\geq 1$, and $Y_i=G_2(\eta_i)$, $i\geq 1$, where $(\xi_i)_{i\geq 1}$ and $(\eta_i)_{i\geq 1}$  are two independent,  stationary, long-range dependent   Gaussian 
processes with $\E\pr{\xi_1}=\E\pr{\eta_1}=0$, $\Var\pr{\xi_1}=\Var \pr{\eta_1}=1$, 
$\rho_{\xi}(k)=\Cov(\xi_1, \xi_{1+k})=k^{-D_{\xi}}L_{\xi}(k) \text{ and } \rho_{\eta}(k)=\Cov(\eta_1, \eta_{1+k})=k^{-D_{\eta}}L_{\eta}(k)$
for 
 $D_{\xi}, D_{\eta} \\\in (0, 1)$ and 
 slowly varying functions $L_{\xi}$ and $L_{\eta}$. 
  Assume that $\E|X_1|<\infty$ and  $\E|Y_1|<\infty$.
\begin{itemize}
\item[(i)] 
 If  $D_{\xi}, D_{\eta}\in (\frac{1}{2},1)$, it holds that
 \begin{align*}
n\dcov_n(X, Y; h)
\overset{\mathcal{D}}{\longrightarrow}
\int_{\mathbb{R}}\int_{\mathbb{R}}\left|Z(s, t)\right|^2w(s,t)dsdt,
\end{align*}
where $\left(Z(s, t)\right)_{s, t\in \mathbb{R}}$ is the complex-valued Gaussian process defined in Theorem \ref{thm:main_result}.
\item[(ii)] If $G=\text{id}$ and $D_{\xi}+D_{\eta}<1$, it holds
 that
 \begin{multline*}
n^{D_{\xi}+D_{\eta}}L_{\xi}^{-1}(n)L_{\eta}^{-1}(n)
\dcov_n(X, Y; h)
\overset{\mathcal{D}}{\longrightarrow}\\
\left|\int_{\left[-\pi, \pi\right)^2}\left[\left(\frac{\e^{ix}-1}{ix}\right)\left(\frac{\e^{iy}-1}{iy}\right)-\frac{\e^{i(x+y)}-1}{i(x+y)}\right]
Z_{X}(dx)Z_{Y}(dy)\right|^2\\
\times\int_{\mathbb{R}}\int_{\mathbb{R}}s^2t^2\exp\left(-s^2+t^2\right)w(s,t)dsdt,
 \end{multline*}
where $Z_{X}$ and $Z_{Y}$ are  random spectral measures defined  subsequently by \eqref{eq:random_spectral_measures}.
\end{itemize}
\end{corollary}

Since the proposed  test statistic $\sum_{h=0}^\infty a_h \dcov_n(X,Y;h)$ is a linear combination of the distance cross-covariances at all lags $h$,  its asymptotic distribution  cannot be derived 
from Corollary \ref{cor:asymp_distr_dcov}.
Instead, for this, 
 joint convergence of    distance cross-covariances at different lags
 is required.
A corresponding result is established by Proposition \ref{prop:convergence_SRD} in Section \ref{sec:outline_of_proofs}.
As a corollary 
of Proposition \ref{prop:convergence_SRD}, 
we obtain the following limit theorem for the test statistic.

\begin{theorem}\label{thm:test_statistic}
Let $X_i=G_1(\xi_i)$, $i\geq 1$, and $Y_i=G_2(\eta_i)$, $i\geq 1$, where $(\xi_i)_{i\geq 1}$ and $(\eta_i)_{i\geq 1}$  are two independent, stationary, long-range dependent   Gaussian 
processes with $\E\pr{\xi_1}=\E\pr{\eta_1}=0$, $\Var\pr{\xi_1}=\Var \pr{\eta_1}=1$, 
$\rho_{\xi}(k)=\Cov(\xi_1, \xi_{1+k})=k^{-D_{\xi}}L_{\xi}(k) \text{ and } \rho_{\eta}(k)=\Cov(\eta_1, \eta_{1+k})=k^{-D_{\eta}}L_{\eta}(k)$
for 
 $D_{\xi}, D_{\eta} \- \in (\frac{1}{2}, 1)$ and 
 slowly varying functions $L_{\xi}$ and $L_{\eta}$.  Assume that $\E|X_1|<\infty$ and  $\E|Y_1|<\infty$.
Given  a real-valued sequence $(a_h)_{h\geq 0}$ with $\sum_{h=0}^{\infty}|a_h|<\infty$ it holds that
 \begin{align*}
n\sum\limits_{h=0}^{\infty}a_h\dcov_n(X, Y; h)\overset{\mathcal{D}}{\longrightarrow}  \sum\limits_{h=0}^{\infty}a_hZ_h,
\end{align*}
where
\begin{align*}
Z_h\defeq \int_{\R}  \int_{\R} \big|Z(s, t; h)\big|^2 w(s,t) ds \,  dt,
\end{align*}
and $\left(Z(s, t; 0), Z(s, t; 1), \ldots, Z(s, t; H)\right)_{s, t\in \mathbb{R}}$ is the complex-valued Gaussian process defined in Theorem~\ref{thm:main_result}.
\end{theorem}

Theorem \ref{thm:test_statistic}
shows that the considered test statistic converges in distribution to a non-degenerate limit.
As the limit distribution is unknown, we base test decisions on a subsampling procedure. 
For the theoretical results on subsampling, we 
do not particularly focus on the proposed test, but we presuppose the following more general situation: Given observation $X_1, \ldots, X_n$
and $Y_1, \ldots, Y_n$
stemming from two independent time series $(X_i)_{i\geq 1}$ and $(Y_i)_{i\geq 1}$, 
 our goal is
to decide on the testing problem $(H_0, H_1)$.
For this purpose, we consider a test statistic  $T_n\defeq T_n(X_1, \ldots, X_n, Y_1, \ldots, Y_n)$, such that we intend to 
 approximate the distribution $F_{T_n}$ of $T_n$.

Therefore, 
the subsampling procedure has to be designed in such a way that it mimics the behavior of the test statistic for two independent time series regardless of whether the  data has been generated according to the model assumptions under the  hypothesis or under the alternative.
To not destroy the dependence structure of the individual time series, it seems reasonable to consider blocks of observations. 
For this, we define blocks
\begin{align*}
B_{k, l_n}\defeq (X_k, \ldots, X_{k+l_n-1}), \ C_{k, l_n}\defeq (Y_k, \ldots, Y_{k+l_n-1})
\end{align*}
based on a predefined block length $l_n$, where  $l_n\leq n$.
To mimic the behavior of two independent time series, it seems reasonable to compute the distance cross-covariance of blocks that are far apart. 
For this reason, we compute the test statistic on blocks that are separated by a lag $d_n$, i.e.  we compute
\begin{align*}
T_{l_n, k}\defeq T_{l_n}\left(X_k, \ldots, X_{k+l_n-1}, Y_{k+d_n}, \ldots, Y_{k+d_n+l_n-1}\right),  \ k=1, \ldots, m_n,
\end{align*}
 where $m_n\defeq n-l_n-d_n$ with $d_n\leq n-l_n$. As a result, we obtain multiple (though dependent) realizations of the test statistic $T_{l_n}$.
Due to the fact that 
 consecutive observations are chosen, the subsamples
retain the dependence structure of the original sample, so that 
the empirical distribution function of  $T_{l_n, 1}, \ldots, T_{l_n, m_n}$, defined by
\begin{align}\label{eq:subsampling_estimator}
\widehat{F}_{m_n,  l_n}(t)\defeq\frac{1}{m_n}\sum\limits_{k=1}^{m_n}1_{\left\{T_{l_n, k}\leq t\right\}},
\end{align}
can be considered as an appropriate estimator for $F_{T_n}$.
 
In order  to establish the validity of the subsampling procedure, i.e. in order to show that the empirical distribution function of  $T_{l_n, 1}, \ldots, T_{l_n, m_n}$ can be considered as a suitable approximation of $F_{T_n}$, we aim at proving that the distance between $\widehat{F}_{m_n,  l_n}$ and  $F_{T_n}$
vanishes as the number of observations tends to $\infty$.
For this, we have to make the following technical assumptions:

\begin{assumption}\label{ass:spectral_density_2}
Let $(\xi_k)_{k\geq 1}$  denote a stationary, long-range dependent Gaussian process  with   $\E(\xi_1)=0$,  $\Var(\xi_1)=1$, LRD parameter $D$ and spectral density  $f(\lambda)=~|\lambda|^{D-1}L_f(\lambda)$ for a slowly varying function $L_f$ which is bounded away from $0$ on $\left[0, \pi\right]$. Moreover, assume that $\lim_{\lambda\rightarrow0}L_f(\lambda)\in(0,\infty]$ exists.
\end{assumption}

\begin{assumption}\label{ass:covariance_function}
Let $(\xi_k)_{k\geq 1}$  denote a stationary, long-range dependent Gaussian process  with   $\E(\xi_1)=0$,  $\Var(\xi_1)=1$,  and   covariance function
\begin{align*}
\rho(k)\defeq~\Cov(\xi_1,\xi_{k+1})=~k^{-D}L_{\rho}(k)
\end{align*}
for some parameter $D\in \left(0, 1\right)$ and some slowly varying function $L_{\rho}$. Assume that there exists a constant $K\in (0, \infty)$, such that for all $n\in \mathbb{N}$
\begin{align*} 
\max\limits_{n+1\leq j\leq n+2l-2}\left|L_{\rho}(n)-L_{\rho}(j)\right|\leq K\frac{l}{n}\min\left\{L_\rho(n),1\right\}
\end{align*}
for  $l\in\{l_n,\ldots,n\}$, where $l_n$ denotes the block length.
\end{assumption}

Given Assumptions \ref{ass:spectral_density_2} and \ref{ass:covariance_function}, consistency of the subsampling procedure is established by the following theorem:

\begin{theorem}\label{thm:subsampling}
Given two independent, stationary, subordinated Gaussian LRD   time series $(X_i)_{i\geq 1}$ and $(Y_i)_{i\geq 1}$  satisfying Assumptions \ref{ass:spectral_density_2} and \ref{ass:covariance_function} with LRD parameters $D_X$ and $D_Y$ and  (measurable) statistics $T_n=T_n(X_1,\ldots,X_n, Y_1, \ldots, Y_n)$  that converge in distribution to a (non-degenerate) random variable $T$. Let $F_T$ and $F_{T_n}$ denote the distribution functions of $T$ and $T_n$.
Moreover, let $(l_n)_{n\geq 1}$ be an increasing, divergent series of integers. If \textcolor{black}{$l_n=\mathcal{O}\left(n^{(1+\min(D_X, D_Y))/2 -\varepsilon}\right)$}  for some $\varepsilon>0$, 
then
\begin{align*}
\left|\widehat{F}_{m_n,  l_n}(t)-F_{T_n}(t)\right|\overset{\mathcal{P}}{\longrightarrow}0,  \ \ \text{as $n\rightarrow\infty$,}
\end{align*}
 for all points of continuity $t$ of $F_T$, i.e. the sampling-window method  is consistent.
\end{theorem}

\begin{remark}
If, additionally, $F_T$ is continuous,
the usual Glivenko-Cantelli argument for  uniform convergence of empirical distribution functions implies that
\begin{align*}
\sup\limits_{t\in \mathbb{R}}\left|\widehat{F}_{m_n,  l_n}(t)-F_{T_n}(t)\right|
\overset{\mathcal{P}}{\longrightarrow}0, \   \ \text{as $n\rightarrow\infty$.}
\end{align*}
\end{remark}

The proof of Theorem \ref{thm:subsampling} is based on arguments that have been established in \cite{betken:wendler:2016}. It can be found in the supplement.

The previous results analyze the asymptotic behaviour of the empirical distance cross-covariance computed on the basis of two independent time series.
The following results characterize the  statistic under more general assumptions. More precisely, we establish consistency of the empirical distance cross-covariance as an approximation of the  distance cross-covariance for stationary, ergodic bivariate processes and derive consistency of the proposed testing procedure for these processes under the additional assumption that for some lag $h$ the random variables $X_j$ and $Y_{j+h}$ are dependent.

\begin{theorem}\label{thm:consistency_dcov_n}
Assume that $(X_j,Y_j)_{j\geq 1} $ is a stationary ergodic process with   $\E|X_1|<\infty$ and  $\E|Y_1|<\infty$. Then, as $n\rightarrow \infty$,
\begin{align*}
 \dcov_n(X,Y;h) \longrightarrow \dcov(X,Y; h)
\end{align*}
almost surely. 
\end{theorem}

Theorem \ref{thm:consistency_dcov_n} can be directly derived from an ergodic theorem for Hilbert space-valued random variables; see the supplement to this manuscript.
An alternative proof based on $U$-statistic theory can be found in \cite{kroll:2020}. 
As an immediate  consequence of Theorem \ref{thm:consistency_dcov_n}, we obtain consistency of the proposed hypothesis test for a broad class of alternatives:

\begin{theorem}\label{thm:distr_test_statistic}
Let $(X_j,Y_j)_{j\geq 1}$ be a stationary ergodic process  $\E|X_1|<\infty$ and  $\E|Y_1|<\infty$, and such that, for some lag $h$, the  random variables $X_j$ and $Y_{j+h}$ are dependent.   Given a real-valued sequence $(a_h)_{h\geq 0}$ with $a_h\neq 0$, as $n\rightarrow \infty$,
\[
 n\, \sum_{h=0}^\infty a_h \dcov_n(X,Y;h) \rightarrow \infty. 
\]
\end{theorem}

\begin{remark} Preliminary calculations show that our test is able to detect deviations from independence of the order $n^{-\frac{1}{2}}$ when $D_\xi > \frac{1}{2}, D_\eta>\frac{1}{2}$, and of order $n^{1-\frac{D_\xi+D_\eta}{2}}$ when $D_\xi+D_\eta<1$. The same rates apply for tests based on empirical cross-covariances as considered in Theorem \ref{thm:sample_covariance}.
\end{remark}

\subsection{A non-central limit theorem for Hilbert space-valued LRD processes}
\label{subsec:CLT}

We aim
at basing our theoretical results on 
limit theorems for Hilbert space-valued random elements. To this end, recall that
for a measure space $\pr{S, \mathcal{S}, \mu}$ the set
\[\mathcal{L}^2\pr{S, \mu}\defeq\left\{f:S\longrightarrow\mathbb{C} \ \text{measurable}\left|\right. \|f\|_{2}<\infty\right\},\] 
equipped with $\|\cdot\|_{2}$, where
  $\|f\|_{2}\defeq\left(\int_{S}\left|f\right|^2d\mu\right)^{\frac{1}{2}}$,
is a semi-normed vector space.
The quotient space 
$L^{2}(S,  \mu)\defeq \mathcal{L}^{2}(S,  \mu)/{\mathcal {N}} \
\text{with} \
\mathcal{N}\defeq \text{ker}(\|\cdot\|_{2})$,
equipped with $\|\cdot\|_{2}$, is then a normed vector space and, in particular, a Hilbert space.

For the analysis of the distance cross-covariance of two random vectors $X$ and $Y$, we note that 
$\dcov_n(X, Y; h)
=\|\varphi_{X, Y; h}^{(n)}-\varphi_X^{(n)}\varphi_Y^{(n)}\|_{2}^2$, where
\begin{align*}
\|f\|_{2}^2=\int_{\mathbb{R}}\int_{\mathbb{R}}\left|f(s, t)\right|^2w(s,t)dsdt
\end{align*}
for $f\in L^{2}(\mathbb{R}^2,  w(s, t)dsdt)$.

In order to derive the asymptotic distribution of  $\dcov_n(X, Y; h)$,  recall that
\begin{align}\label{eq:decomposition_3}
&\varphi_{X, Y; h}^{(n)}(s, t)-\varphi_X^{(n)}(s)\varphi_Y^{(n)}(t)\\
=&-\left(\varphi_X^{(n)}(s)-\varphi_X(s)\right)\left(\varphi_Y^{(n)}(t)-\varphi_Y(t)\right)\notag\\
&+\frac{1}{n}\sum\limits_{j=1}^{n-h}\left(\exp(isX_j)-\varphi_X(s)\right)\left(\exp(itY_{j+h})-\varphi_Y(t)\right) +o(n^{-\gamma})\notag
\end{align}
for any $\gamma<1$.
 According to Lemma \ref{lemma:existence_1} we can  consider $\varphi_X^{(n)}(s)-\varphi_X(s)$ and $\varphi_Y^{(n)}(t)-\varphi_Y(t)$ 
as elements of  $L^{2}(\mathbb{R},  w(s)ds)$, where $w(s)\defeq (cs^2)^{-1}$,
and $\left(\varphi_Y^{(n)}(t)-\varphi_Y(t)\right)\left(\varphi_X^{(n)}(s)-\varphi_X(s)\right)$ as an element of $L^{2}(\mathbb{R}^2,  w(s, t)dsdt)$.
For an analysis of these terms,  we thus 
establish a non-central limit theorem for processes with values in an $L^2$-Hilbert space. 
Against the background of analyzing the  distance cross-covariance of time series, this limit theorem
provides a basis for deriving the asymptotic distribution of the  distance cross-covariance of subordinated Gaussian processes. 
Yet, the limit theorem is  not specially geared to this problem and can thus be considered of  independent interest.

\begin{theorem}\label{thm:1_parameter}
 Let  $(X_i)_{i\geq 1}$ be a  stationary, long-range dependent Gaussian 
process with $\E X_1=0$, $\Var \pr{X_1}=1$, and auto-covariance function
$\rho(k)\sim k^{-D}L(k)$
for the LRD parameter $D\in (0,1)$ and a slowly varying function $L$.
 Given a positive weight function $w:~\mathbb{R}\longrightarrow\mathbb{R}_+$, consider the Hilbert space $S\defeq L^2\left(\mathbb{R}, w(t)dt\right)$. Let 
  $f:\mathbb{R}\longrightarrow S$ map $x\in \mathbb{R}$ to the function $t\mapsto f_t(x)$ with $(t, x)\mapsto f_t(x)$ measurable and $f_t\in L^2(\mathbb{R}, \varphi(x)dx)$ for all $t\in \mathbb{R}$.
Moreover, assume that  the  LRD parameter $D$ meets the condition $0<D<~\frac{1}{r}$, where $r$ denotes the Hermite rank of the  class of functions $\left\{f_t(X_1)-\E f_t(X_1), \ t\in \mathbb{R}\right\}$, i.e. 
\begin{align*}
r\defeq~\min \left\{q\geq 1: J_q(t)\neq 0 \ \text{for  some} \ t\in \mathbb{R}\right\}, \ J_{q}(t)=~\E \left(f_t(X_1)H_q(X_1)\right). 
\end{align*} 
If $\E\left(\|f_{X}\|_{2}^2\right)<~\infty$, where $f_X(t)\defeq f_t(X_1)$, then
\begin{align*}
\left\| n^{\frac{rD}{2}-1}L^{-\frac{r}{2}}(n)\sum\limits_{j=1}^n\left[\left(f_{t}(X_j)-\E f_{t}(X_j) \right)-\frac{1}{r!}J_r(t)H_r(X_j)\right]\right\|_{2}=o_P\left(1\right).
\end{align*}
Moreover, it follows that
\begin{align*}
n^{\frac{rD}{2}-1}L^{-\frac{r}{2}}(n)\sum\limits_{j=1}^n\left(f_t(X_j)-\E f_t(X_j) \right)\overset{\mathcal{D}}{\longrightarrow} \frac{1}{r!}J_r(t)Z_{r, H} (1), \ t \in \mathbb{R},
\end{align*}
where
 $Z_{r, H}$ denotes an $r$-th order Hermite process with  $H=1-\frac{rD}{2}$, and   $\overset{\mathcal{D}}{\longrightarrow}$ denotes  convergence in distribution in $L^2\left(\mathbb{R}, w(t)dt\right)$.
\end{theorem}

 Theorem \ref{thm:1_parameter}
allows to characterize the asymptotic behavior of the empirical characteristic function through the following corollary:

\begin{corollary}\label{cor:convergence_of_ecf}
Let $X_i=G(\xi_i)$, $i\geq 1$, where $(\xi_i)_{i\geq 1}$ is a  stationary, long-range dependent   Gaussian 
process with $\E\pr{\xi_1}=0$, $\Var\pr{\xi_1}=1$,  and
$\rho(k)=\Cov(\xi_1, \xi_{1+k})=k^{-D}L(k)$
for 
 $D \in (0, 1)$ and  a
 slowly varying function $L$.
 Given a positive weight function $w:~\mathbb{R}\longrightarrow\mathbb{R}_+$, consider the Hilbert space $S\defeq L^2\left(\mathbb{R}, w(t)dt\right)$.
Moreover, let $f_{s, 1}(x)=\cos(sG(x))$, $f_{s, 2}(x)=\sin(sG(x))$, and let $r_1$ and $r_2$ denote the corresponding Hermite ranks, i.e. $r_i:=\min\{q\geq 1:J_{q, i}(s)\neq 0 \ \text{for some $s\in \mathbb{R}$}\}$, where $J_{q, i}(s)=\E\left(f_{s, i}(\xi_1)H_{q}(\xi_1)\right), \ i=1, 2$. Assume that $0< D<\frac{1}{r_i}$, $i=1, 2$.
Then, we have
\begin{align*}
&n^{\frac{r_1 D}{2}}L^{-\frac{r_1}{2}}(n)\left\|\Ret\left(\varphi_X^{(n)}(s)-\varphi_X(s)\right)-\frac{1}{r_1!}J_{r_1, 1}(s)\frac{1}{n}\sum\limits_{i=1}^nH_{r_1}(\xi_i)\right\|_{2}=o_P\left(1\right),\\
&n^{\frac{r_2 D}{2}}L^{-\frac{r_2}{2}}(n)\left\|\Imt\left(\varphi_X^{(n)}(s)-\varphi_X(s)\right)-\frac{1}{r_2!}J_{r_2, 2}(s)\frac{1}{n}\sum\limits_{i=1}^nH_{r_2}(\xi_i)\right\|_{2}=o_P\left(1\right).
\end{align*}
Moreover, when $G=\text{id}$, it follows that
\begin{align*}
&n^{\frac{D}{2}}L^{-\frac{1}{2}}(n)\Ret\left(\varphi_X^{(n)}(s)-\varphi_X(s)\right)\overset{\mathcal{D}}{\longrightarrow}0, \; s\in \mathbb{R}, \\
\intertext{while}
&n^{\frac{D}{2}}L^{-\frac{1}{2}}(n)\Imt\left(\varphi_X^{(n)}(s)-\varphi_X(s)\right)\overset{\mathcal{D}}{\longrightarrow}\text{$c_D$}s\exp\left(-\frac{s^2}{2}\right)Z, \; s\in \mathbb{R}, 
\end{align*}
where  $Z$ is a standard normally distributed random variable and $c_D=\sqrt{\frac{2}{(1-D)(2-D)}}$. 
\end{corollary}

\begin{remark}
For $G=\text{id}$ the assertion of
Corollary \ref{cor:convergence_of_ecf} follows from
\begin{align*}
\E\left(\cos(s\xi_1)\xi_1\right)=0 \ \text{and} \
\E\left(\sin(s \xi_1)\xi_1\right)
=\exp\left(-\frac{s^2}{2}\right)s.
\end{align*}
Note that, since $w(s)=(cs^2)^{-1}$,
\begin{align*}
\int\left|\exp\left(-\frac{s^2}{2}\right)s\right|^2w(s)ds<\infty,
\end{align*}
such that the limit in the above corollary takes values in $L^2\left(\mathbb{R}, w(s)ds\right)$.
\end{remark}

\section{Outline of proofs and auxiliary results}\label{sec:outline_of_proofs}

In the following, we outline a step-by-step proof 
of Theorem \ref{thm:main_result}.
For this purpose, recall that
\begin{align}\label{eq:decomposition_2}
&\varphi_{X, Y; h}^{(n)}(s, t)-\varphi_X^{(n)}(s)\varphi_Y^{(n)}(t)\\
=&-\left(\varphi_X^{(n)}(s)-\varphi_X(s)\right)\left(\varphi_Y^{(n)}(t)-\varphi_Y(t)\right)\notag\\
&+\frac{1}{n}\sum\limits_{j=1}^{n-h}\left(\exp(isX_j)-\varphi_X(s)\right)\left(\exp(itY_{j+h})-\varphi_Y(t)\right) +o(n^{-\gamma})\notag
\end{align}
for any $\gamma<1$.
Under the assumption of independence or short-range dependence within the sequences  $(X_i)_{i\geq 1}$ and $(Y_i)_{i\geq 1}$ 
$$
\sqrt{n}\left(\varphi_X^{(n)}(s)-\varphi_X(s)\right)\left(\varphi_Y^{(n)}(t)-\varphi_Y(t)\right)=o_P(1),$$
i.e. with a corresponding normalization the first summand on the right-hand side of the above equation is asymptotically negligible, while the second summand determines the asymptotic distribution of the left-hand side.
For long-range dependent time series  $(X_i)_{i\geq 1}$ and  $(Y_i)_{i\geq 1}$ the asymptotic behavior of the second summand depends on 
an interplay of dependence within the time series, such that both summands may contribute to the limit distribution.
Accordingly, 
 we take account of both summands for our analysis.
For this, we consider the two summands separately and state
corresponding intermediate results.
Detailed proofs of these
are left to the supplement.\\

For the first summand, we prove the following result:

\begin{proposition}\label{prop:reduction}
Let $(X_i)_{i\geq 1}$ and $(Y_i)_{i\geq 1}$ be two independent,  stationary, long-range dependent Gaussian 
processes with $\E\pr{X_1}=\E\pr{Y_1}=0$, $\Var\pr{X_1}=\Var\pr{Y_1}=1$, 
$\rho_X(k)=\Cov(X_1, X_{1+k})=k^{-D_X}L_X(k)$ and $\rho_Y(k)=\Cov(Y_1, Y_{1+k}) = k^{-D_Y}L_Y(k)
$ for $D_X, D_Y \in (0, 1)$ and slowly varying functions $L_X$ and $L_Y$. 
Then, it holds
 that
 \begin{multline}\label{eq:reduction_prod}
 \left\|\left(\varphi_X^{(n)}(s)-\varphi_X(s)\right)\left(\varphi_Y^{(n)}(t)-\varphi_Y(t)\right)-
 J_1(s)\frac{1}{n}\sum\limits_{j=1}^nX_j J_1(t)\frac{1}{n}\sum\limits_{j=1}^nY_j 
 \right\|_{2}\\
 =o_P\left(n^{-\frac{D_X+D_Y}{2}}L_X^{\frac{1}{2}}(n)L_Y^{\frac{1}{2}}(n)\right),
 \end{multline}
 where $J_1(s)=i\exp\left(-\frac{s^2}{2}\right)s$.
 Moreover, it follows that 
 \begin{align*}
n^{\frac{D_X+D_Y}{2}}L_X^{-\frac{1}{2}}(n)L_Y^{-\frac{1}{2}}(n)\left(\varphi_X^{(n)}(s)-\varphi_X(s)\right)\left(\varphi_Y^{(n)}(t)-\varphi_Y(t)\right)
\overset{\mathcal{D}}{\longrightarrow}st\exp\left(-\frac{s^2+t^2}{2}\right)Z_XZ_Y,
 \end{align*}
where $Z_X$, $Z_Y$ are independent standard normally distributed random variables.  
 \end{proposition}

 The asymptotic behavior of the second summand in the decomposition in formula
\eqref{eq:decomposition_2} depends on the values of the long-range dependence parameters $D_X$ and $D_Y$.
For this reason, the following propositions treat different values of these separately.
Initially, we consider small values of $D_X$ and $D_Y$. Following this, we focus on bigger values of $D_X$ and $D_Y$.

\begin{proposition}\label{prop:reduction_LRD_2}
Let  $(X_i)_{i\geq 1}$ and $(Y_i)_{i\geq 1}$ be two independent,  stationary, long-range dependent Gaussian 
processes with $\E\pr{X_1}=\E\pr{Y_1}=0$, $\Var \pr{X_1}=\Var\pr{Y_1}=1$,  $\rho_X(k)=k^{-D_X}L_X(k)$, and $\rho_Y(k)=k^{-D_Y}L_Y(k)$ for $D_X+D_Y \in (0, 1)$ and  slowly varying functions $L_X$, $L_Y$.
Then, it holds that
\begin{multline*}
\Biggl\|\frac{1}{n}\sum\limits_{j=1}^{n-h}\left(\exp(isX_j)-\varphi_X(s)\right)\left(\exp(itY_{j+h})-\varphi_Y(t)\right)
\\
+\exp\left(-\frac{s^2}{2}\right)s\exp\left(-\frac{t^2}{2}\right)t\frac{1}{n}\sum\limits_{j=1}^{n-h}X_jY_{j+h}\Biggr\|_{2}=o_P\left(n^{-\frac{D_X+D_Y}{2}}L_X^{\frac{1}{2}}(n)
L_Y^{\frac{1}{2}}(n)\right).
\end{multline*}
\end{proposition}

Taking Proposition \ref{prop:reduction}
and the decomposition in \eqref{eq:decomposition_2}
 into consideration
 and noting that
\begin{align*}
\int_{\mathbb{R}}\int_{\mathbb{R}}
\left|\exp\left(-\frac{s^2}{2}\right)s\exp\left(-\frac{t^2}{2}\right)t\right|^2w(s, t)dsdt=\frac{\pi}{c^2},
\end{align*} it follows that the limit
of 
\begin{align*}
&n^{D_X+D_Y}L_X^{-1}(n)L_Y^{-1}(n)\dcov_n(X, Y; h)\\
=&\int_{\mathbb{R}}\int_{\mathbb{R}}\left|n^{\frac{D_X+D_Y}{2}}L_X^{-\frac{1}{2}}(n)L_Y^{-\frac{1}{2}}(n)\left(\varphi_{X, Y; h}^{(n)}(s, t)-\varphi_X^{(n)}(s)\varphi_Y^{(n)}(t)\right)\right|^2w(s,t)dsdt
\end{align*}
equals  the limit of
\begin{align*}
\frac{\pi}{c^2}\biggl(n^{\frac{D_X+D_Y}{2}-2}L_X^{-\frac{1}{2}}(n)L_Y^{-\frac{1}{2}}(n)\sum\limits_{i=1}^n\sum\limits_{j=1}^{n}X_iY_{j}-n^{\frac{D_X+D_Y}{2}-1}L_X^{-\frac{1}{2}}(n)L_Y^{-\frac{1}{2}}(n)\sum\limits_{j=1}^{n-h}X_jY_{j+h}\biggr)^2.
\end{align*}

In order to derive the limit distribution of the above expression, we make use of the theory on spectral distributions established in \cite{major:2020}. 
For this, we consider the following representation of Gaussian random variables:
\begin{align}\label{eq:random_spectral_measures}
X_j=\int_{\left[-\pi, \pi\right)}\e^{ijx}Z_{ X}(dx), \
Y_j=\int_{\left[-\pi, \pi\right)}\e^{ijy}Z_{ Y}(dy),
\end{align}
where $Z_{ X}$ and $Z_{ Y}$ are
corresponding random spectral measures determined by the positive semidefinite matrix-valued, even  spectral measure $(G_{j, j'})$, $1\leq j, j'\leq 2$, on the torus $\left[-\pi, \pi\right)$ with coordinates $G_{j, j'}$ satisfying
\begin{flalign*}
&\E(X_j X_{j+k})=\int_{\left[-\pi, \pi\right)}\e^{ikx}G_{1, 1}(dx), \ 
\E(Y_j Y_{j+k})=\int_{\left[-\pi, \pi\right)}\e^{ikx}G_{2, 2}(dx), \\
&\E(X_j Y_{j+k})=\int_{\left[-\pi, \pi\right)}\e^{ikx}G_{1, 2}(dx)=0, \ 
\E(Y_j X_{j+k})=\int_{\left[-\pi, \pi\right)}\e^{ikx}G_{2, 1}(dx)=0.
\end{flalign*}

According to 
\cite{major:2020}
the random spectral measures
\begin{align*}
&Z_{X}^{(n)}(A)=\sqrt{n^{D_X}L_X^{-1}(n)}Z_{X}\left(\frac{A}{n}\right)=n^{\frac{D_X}{2}}L_X^{-\frac{1}{2}}(n)Z_{X}\left(\frac{A}{n}\right)
\intertext{and}
&Z_{G, Y}^{(n)}(A)=\sqrt{n^{D_Y}L_Y^{-1}(n)}Z_{G, Y}\left(\frac{A}{n}\right)=n^{\frac{D_Y}{2}}L_Y^{-\frac{1}{2}}(n)Z_{G, Y}\left(\frac{A}{n}\right)
\end{align*}
converge to limits $Z_{X, 0}$ und $Z_{Y, 0}$.

\begin{proposition}\label{prop:conv_LRD}
Let  $(X_i)_{i\geq 1}$ and $(Y_i)_{i\geq 1}$ be two independent,  stationary, long-range dependent Gaussian 
processes with $\E\pr{X_1}=\E\pr{Y_1}=0$, $\Var\pr{X_1}=\Var\pr{Y_1}=1$, 
$\rho_X(k)=\Cov(X_1, X_{1+k})=k^{-D_X}L_X(k)$  and  $\rho_Y(k)=\Cov(Y_1, Y_{1+k})=k^{-D_Y}L_Y(k)$
for $D_X +D_Y \in (0, 1)$ and slowly varying functions $L_X$ and $L_Y$.
Then, it holds that
 \begin{multline*}
n^{\frac{D_X+D_Y}{2}-2}L_X^{-\frac{1}{2}}(n)L_Y^{-\frac{1}{2}}(n)\sum\limits_{i=1}^n\sum\limits_{j=1}^nX_iY_j-n^{\frac{D_X+D_Y}{2}-1}L_X^{-\frac{1}{2}}(n)L_Y^{-\frac{1}{2}}(n)\sum\limits_{j=1}^nX_jY_j\\
\overset{\mathcal{D}}{\longrightarrow}
\int_{\left[-\pi, \pi\right)^2}\left[\left(\frac{\e^{ix}-1}{ix}\right)\left(\frac{\e^{iy}-1}{iy}\right)-\frac{\e^{i(x+y)}-1}{i(x+y)}\right]
Z_{X, 0}(dx)Z_{Y, 0}(dy).
 \end{multline*}
\end{proposition}

For $D_X, D_Y\in \left(\frac{1}{2}, 1\right)$, we derive the asymptotic distribution  of the second summand in the decomposition \eqref{eq:decomposition_2} under the general assumption of subordinated Gaussian processes $X_i=G_1(\xi_i)$, $i\geq 1$, and $Y_i=G_2(\eta_i)$, $i\geq 1$.
Most notably, the following Proposition does not only establish convergence of the summand for a single lag $h$, but also  joint convergence of  summands for different lags.
The latter is needed for  deriving the asymptotic distribution of the test statistic $ n\, \sum_{h=0}^\infty a_h \dcov_n(X,Y;h)$, i.e. for a proof of Theorem \ref{thm:distr_test_statistic}.

\begin{proposition}\label{prop:convergence_SRD}
Let $X_i=G_1(\xi_i)$, $i\geq 1$, and $Y_i=G_2(\eta_i)$, $i\geq 1$, where $(\xi_i)_{i\geq~1}$ and $(\eta_i)_{i\geq 1}$ are two independent, stationary, long-range dependent   Gaussian 
processes with $\E\pr{\xi_1}=\E\pr{\eta_1}=0$, $\Var\pr{\xi_1}=\Var \pr{\eta_1}=1$, 
$\rho_{\xi}(k)=\Cov(\xi_1, \xi_{1+k})=k^{-D_{\xi}}L_{\xi}(k)$ and  $\rho_{\eta}(k)=\Cov(\eta_1, \eta_{1+k})=k^{-D_{\eta}}L_{\eta}(k)$
for 
 $D_{\xi}, D_{\eta} \in (\frac{1}{2}, 1)$ and slowly varying functions $L_{\xi}$ and $L_{\eta}$.
Assume that $\E|X_1|<\infty$   and $\E|Y_1|<\infty$ and define
 \begin{align*}
   Z_n(s, t; h)&=\frac{1}{\sqrt{n}} \sum\limits_{j=1}^{n-h}f_{s, t}(X_j, Y_{j+h}), \
\intertext{where}
\
f_{s, t}(X_j, Y_{j+h})&=\left(\exp(isX_j)-\varphi_X(s)\right)\left(\exp(itY_{j+h})-\varphi_Y(t)\right).
\end{align*}
Then, it holds that
 \begin{align*}
\left(Z_n(s, t; 0), Z_n(s, t; 1), \ldots, Z_n(s, t; H) \right)\overset{\mathcal{D}}{\longrightarrow}
 \left(Z(s, t; 0), Z(s, t; 1), \ldots, Z(s, t; H)\right),
\end{align*}
where  $\overset{\mathcal{D}}{\longrightarrow}$ denotes convergence in $\mathcal{L}^{2}(\mathbb{R}^2,  w(s, t)dsdt)\otimes \cdots \otimes \mathcal{L}^{2}(\mathbb{R}^2,  w(s, t)dsdt)$\\
 and $\left(Z(s, t; 0), Z(s, t; 1), \ldots, Z(s, t; H)\right)_{s, t\in \mathbb{R}}$ is a complex-valued Gaussian process  with covariance structure 
 \begin{align*}
\Gamma_{s,t, s', t'}(i, j)&=\Cov\pr{Z(s, t; i), Z(s', t'; j)}=\E\pr{Z(s, t; i) \overline{Z(s', t'; j)}}\\
&=\sum\limits_{k=-\infty}^{\infty}
\E \pr{f_{s, t}(X_1, Y_{1+i})\overline{f_{s', t'}(X_{k+1}, Y_{k+1+j})}},\\
C_{s,t,  s', t'}(i, j)&=\Cov\pr{Z(s, t; i), \overline{Z(s', t'; j)}}=\E\pr{Z(s, t; i) Z(s', t'; j)}\\
&=\sum\limits_{k=-\infty}^{\infty}
\E \pr{f_{s, t}(X_1, Y_{1+i})f_{s', t'}(X_{k+1}, Y_{k+1+j})}.
\end{align*}
\end{proposition}

\section{Finite sample performance}\label{sec:simulations_and_data}

So far, we focused on analyzing the asymptotic behavior of the distance cross-covariances with respect to data
$\pr{X_i, Y_i}$, $i=1, \ldots, n$, stemming from long-range dependent time series  $(X_i)_{i\geq 1}$ and $(Y_i)_{i\geq 1}$.

In Section \ref{subsec:simulations}, we will assess the finite sample performance of the corresponding hypothesis test. In particular, we will compare its finite sample performance to that of a hypothesis test based on the empirical cross-covariance
\begin{align*}
\cov_n(X, Y; h)\defeq\frac{1}{n}\sum\limits_{i=1}^{n-h}(X_i-\bar{X})(Y_{i+h}-\bar{Y}).
\end{align*}

In Section \ref{sec:data},
we apply both hypothesis tests for an analysis of
 the mean monthly discharges of three different rivers with regard to cross-dependence between the corresponding data-generating processes.

\subsection{Simulations}\label{subsec:simulations}

Prior to a comparison of the finite sample performance of the two dependence measures, we 
derive a limit theorem for 
the  empirical cross-covariance 
complementing our main theoretical results stated in Theorem \ref{thm:main_result}.

\begin{theorem}\label{thm:sample_covariance}
 Let $X_i=G_1(\xi_i)$, $i\geq 1$, and $Y_i=G_2(\eta_i)$, $i\geq 1$, where $(\xi_i)_{i\geq 1}$ and $(\eta_i)_{i\geq 1}$  are two independent,  stationary, long-range dependent   Gaussian 
processes with $\E\pr{\xi_1}=\E\pr{\eta_1}=0$, $\Var\pr{\xi_1}=\Var \pr{\eta_1}=1$, 
$\rho_{\xi}(k)=\Cov(\xi_1, \xi_{1+k})=k^{-D_{\xi}}L_{\xi}(k) \text{ and } \rho_{\eta}(k)=\Cov(\eta_1, \eta_{1+k})=k^{-D_{\eta}}L_{\eta}(k)$
for 
 $D_{\xi}, D_{\eta} \in (0, 1)$ and 
 slowly varying functions $L_{\xi}$ and $L_{\eta}$. 
 Assume that $\E(X_1^2)<\infty$   and  $\E(Y_1^2)<\infty$.
\begin{itemize}
\item[(i)] If  $D_{\xi}r_1+ D_{\eta} r_2> 1$, where  $r_1$ and $r_2$ denote the Hermite ranks of $G_1$ and $G_2$,  it holds that
 \begin{align*}
\frac{1}{\sqrt{n}}\sum_{i=1}^{n-h}(X_i-\bar{X})(Y_{i+h}-\bar{Y})
\overset{\mathcal{D}}{\longrightarrow}\mathcal{N}(0, \sigma^2),
\end{align*}
where  
$\sigma^2=
\sum_{k=-\infty}^{\infty}\rho_X(k)\rho_Y(k)$.
\item[(ii)] 
If $G=\text{id}$ and $D_{\xi} + D_{\eta}<1$, it holds that
 \begin{align*}
&n^{\frac{D_X+D_Y}{2}}L_X^{-\frac{1}{2}}(n)L_Y^{-\frac{1}{2}}(n)\frac{1}{n}\sum_{i=1}^{n-h}(X_i-\bar{X})(Y_{i+h}-\bar{Y})\\
&\overset{\mathcal{D}}{\longrightarrow}
\int_{\left[-\pi, \pi\right)^2}\left[\left(\frac{\e^{ix}-1}{ix}\right)\left(\frac{\e^{iy}-1}{iy}\right)-\frac{\e^{i(x+y)}-1}{i(x+y)}\right]
Z_{X, 0}(dx)Z_{Y, 0}(dy),
 \end{align*}
where $Z_{X, 0}$ and $Z_{Y, 0}$ correspond to the limit measures in Proposition \ref{prop:conv_LRD}.
\end{itemize}
\end{theorem}

Analogous to  Theorem \ref{thm:main_result}  for the empirical distance cross-covariance, Theorem \ref{thm:sample_covariance} focuses on a characterization of the limit distribution for the empirical cross-covariances in the case of relatively large values of $D_{\xi}$ and $D_{\eta}$.
According to the corresponding restrictions of the two theorems, the simulation results presented in this section are all based on long-range dependent time series  satisfying these restrictions.
In particular, simulation result are based on LRD time series characterized by LRD parameters  $D_{\xi}$, $D_{\eta}\in \left(\frac{1}{2}, 1\right)$.
Moreover, we restrict our considerations to tests based on the empirical distance cross-covariance and the empirical cross-covariance at lag $h=0$, i.e. we choose the empirical distance covariance $\dcov_n(X, Y; 0)$ and the empirical covariance $\cov_n(X, Y; 0)$ as test statistics.
 
In order to compare the performance of hypothesis tests based on empirical distance covariances to that  based on the  empirical covariances, we consider four different scenarios:  \\
1. \enquote{linearly} correlated data, i.e. we simulate
$k=5000$ repetitions of $(X_1, \ldots, X_n,\\
 Y_1\ldots, Y_n)\defeq (2\Phi(Z_1)-1, \ldots, 2\Phi(Z_{2n})-1)$, where $(Z_1, \ldots, Z_{2n})$ is multivariate normally distributed with mean $0$ and
covariance matrix
\begin{align*}
\sigma_{i, j} \defeq \begin{cases}
\rho(|i-j|) \ &\text{for} \ 1\leq i,j\leq n\\
\sigma_{i-n, j-n} \ &\text{for} \ n+1\leq i,j\leq 2n\\
r\sigma_{i-n, j} \ &\text{for} \ n+1\leq i\leq 2n, 1\leq j\leq n\\
r\sigma_{i, j-n} \ &\text{for} \ 1\leq i\leq n, n+1\leq j\leq 2n\\
\end{cases},
\end{align*}
where
$\rho(k)=\frac{1}{2}\left(\left|k+1\right|^{2H}-2\left|k\right|^{2H}+\left|k-1\right|^{2H}\right)$;
see Figure \ref{fig:linear} for an illustration of different parameter combinations.
\begin{figure}[htbp]
\begin{center} 
\includegraphics[width=.75\textwidth]{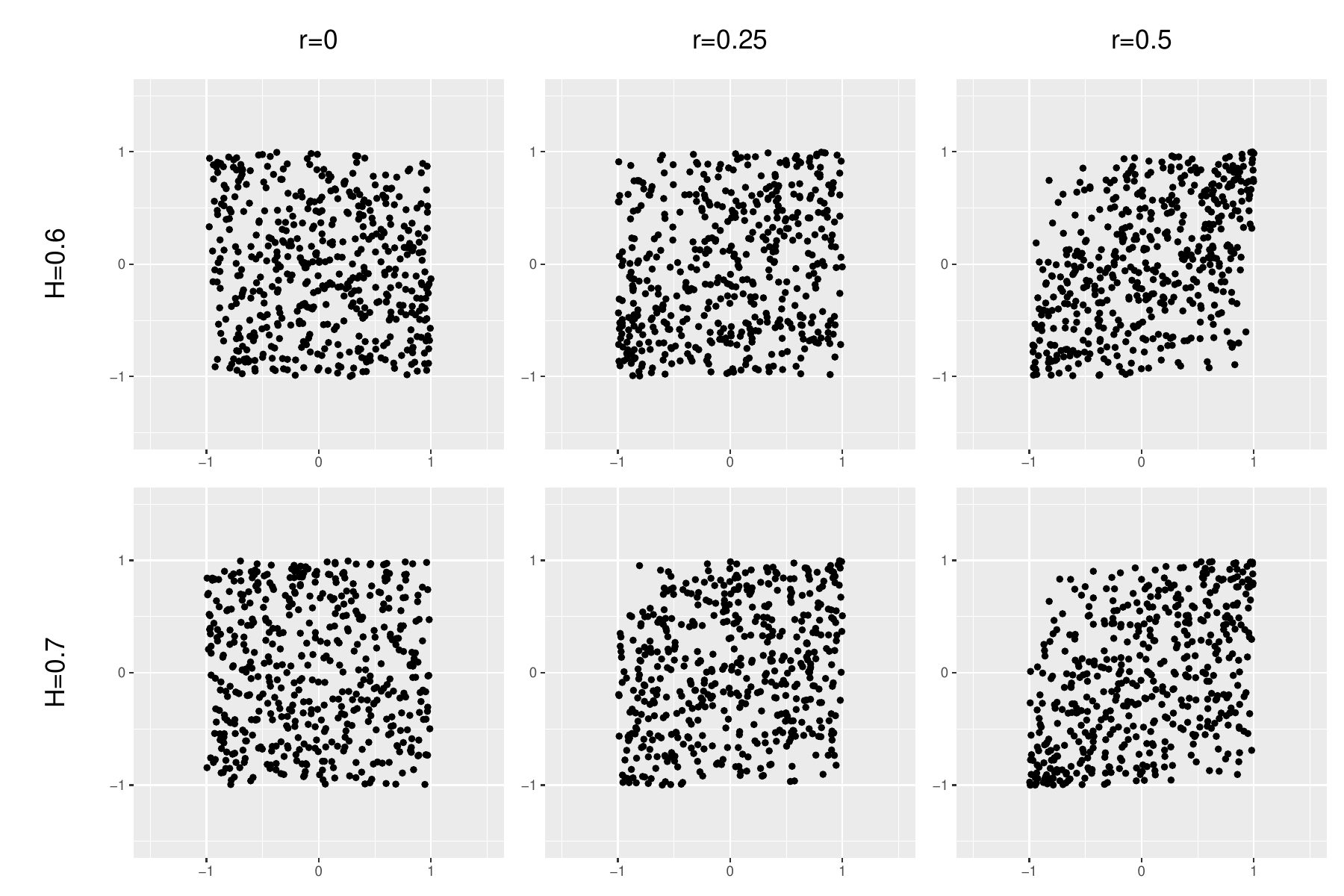}
\caption{ \enquote{Linearly} correlated data $(X_i,Y_i)$,   $i=1, \ldots, 500$,  with parameters $H$ and $r$.}
\label{fig:linear}
\end{center}
\end{figure}

2. \enquote{parabolically} correlated data, i.e.  we simulate
$k=5000$ repetitions of $(X_1, \ldots, X_n, Y_1\ldots, Y_n)$, where 
$(X_1, \ldots, X_n)\defeq (2\Phi(Z_1)-1, \ldots, 2\Phi(Z_{n})-1)$, for fractional Gaussian noise $(Z_i)_{i\geq 1}$ with parameter $H$, and 
\begin{align}\label{eq:squares_transform}
Y_i\defeq v\pr{X_i^2-\frac{1}{3}}+w\xi_i,  \ w=\sqrt{1-\frac{4}{15}v^2},
\end{align}
where  $(\xi_i)_{i\geq 1}$  are independent uniformly on $[-1,1]$ distributed random variables; see Figure \ref{fig:squares}  for an illustration of different parameter combinations.
The choice of the parameter $w$ guarantees $\E Y_1=\E X_1=0$ and $\Var\pr{Y_1}=\Var\pr{X_1}=\frac{1}{3}$.
\begin{figure}[htbp]
\begin{center} 
\includegraphics[width=.75\textwidth]{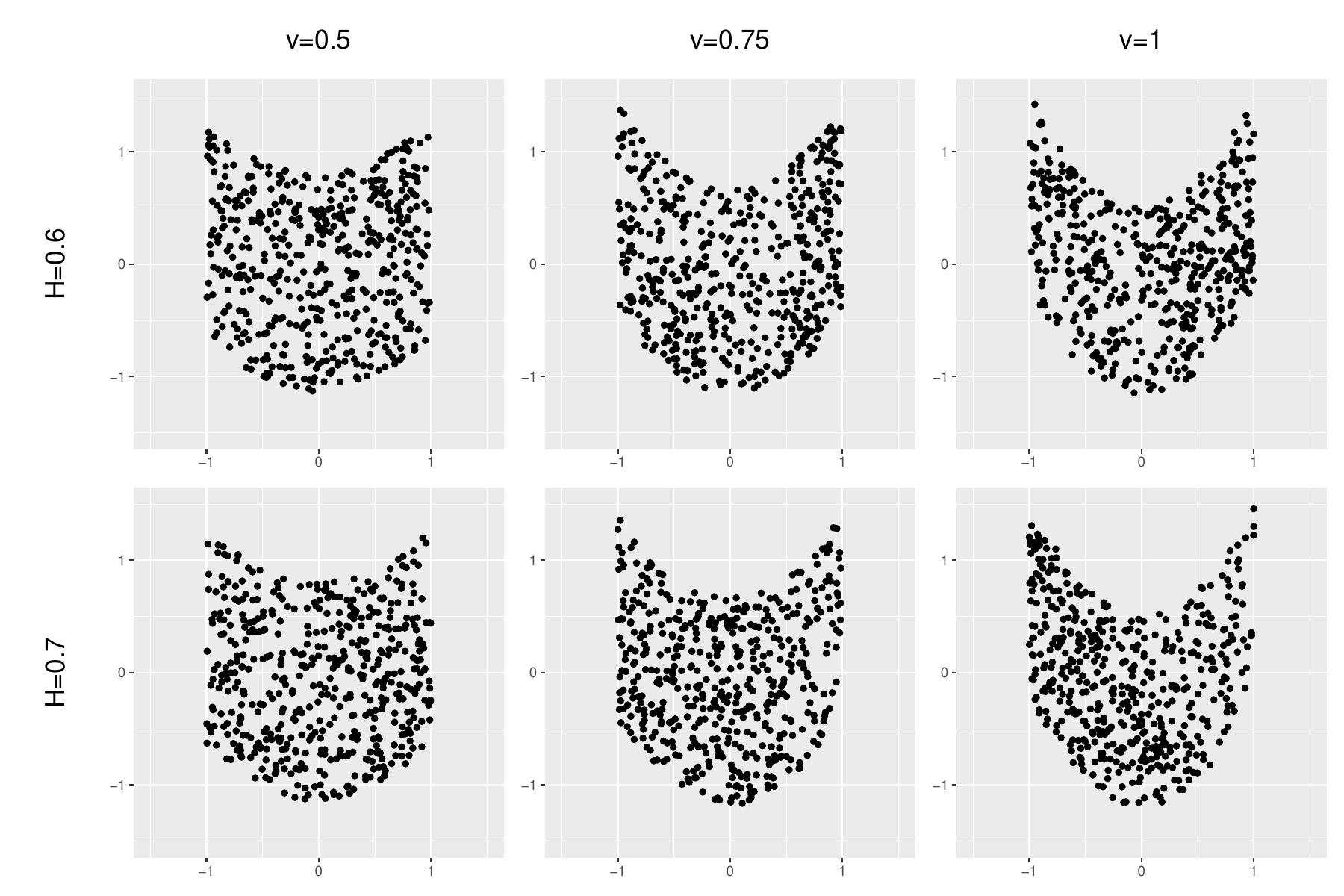}
\caption{ \enquote{Parabolically} correlated data  $(X_i,Y_i)$,   $i=1, \ldots, 500$,  with parameters $H$ and $v$.}
\label{fig:squares}
\end{center}
\end{figure}

3. \enquote{wavily} correlated data, i.e.  we simulate
$k=5000$ repetitions of $(X_1, \ldots, X_n, \\
Y_1\ldots, Y_n)$, where 
$(X_1, \ldots, X_n)\defeq (2\Phi(Z_1)-1, \ldots, 2\Phi(Z_{n})-1)$, for fractional Gaussian noise  $(Z_i)_{i\geq 1}$  with parameter $H$, and 
\begin{align}\label{eq:wave_transform}
Y_j\defeq v\pr{\pr{X_j^2-\frac{1}{3}}^2-3/45}+w\xi_j,  \ w=\sqrt{1-242/4725v^2},
\end{align}
where  $(\xi_i)_{i\geq 1}$  are independent $\mathcal{U}[-1,1]$ distributed random variables; see Figure \ref{fig:wave}  for an illustration of different parameter combinations.
The choice of the parameter $w$ guarantees $\E Y_1=\E X_1=0$ and $\Var\pr{Y_1}=\Var\pr{X_1}=\frac{1}{3}$.
\begin{figure}[htbp]
\begin{center} 
\includegraphics[width=.75\textwidth]{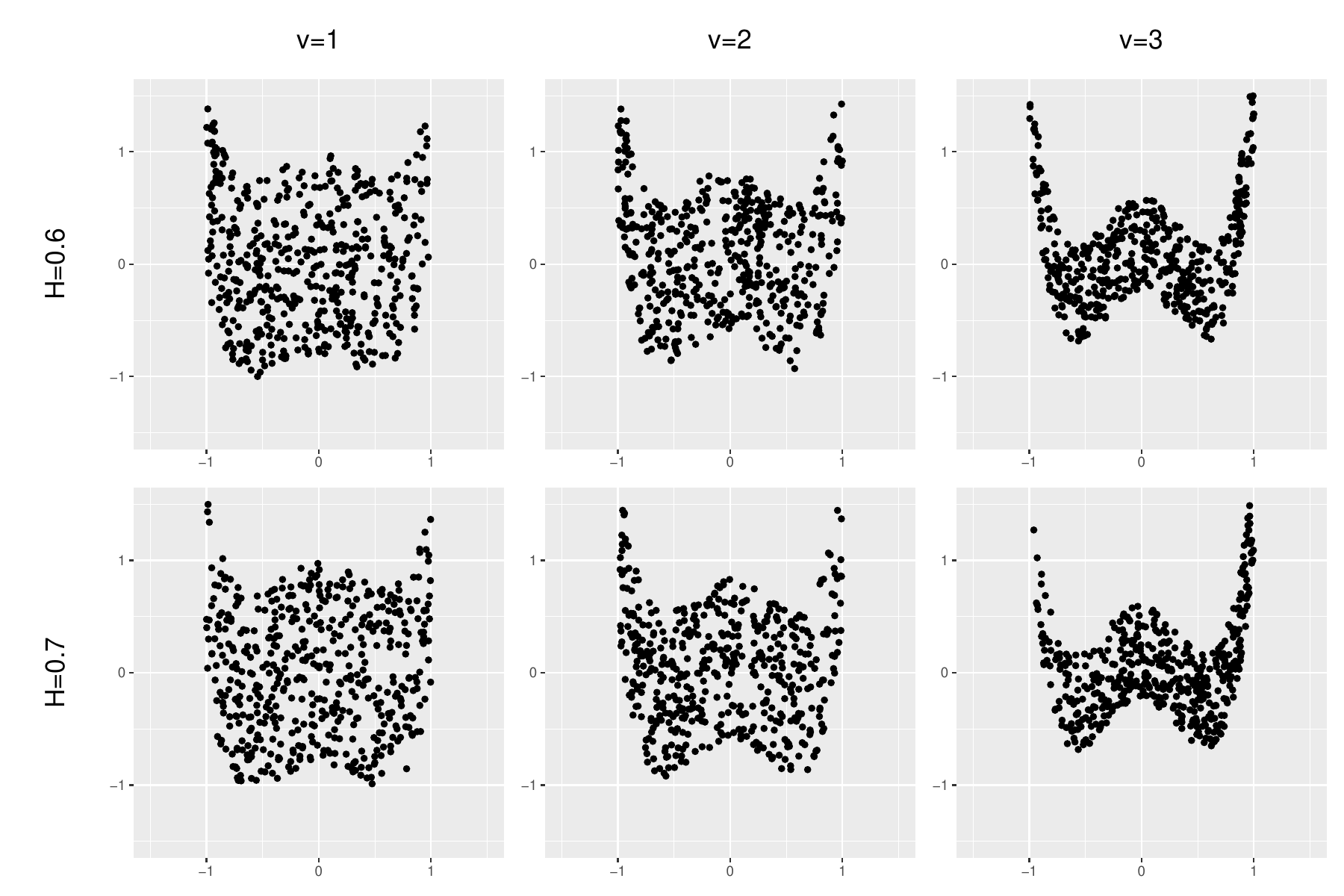}
\caption{\enquote{Wavily} correlated data $(X_i,Y_i)$,   $i=1, \ldots, 500$,  with parameters $H$ and $v$.}
\label{fig:wave}
\end{center}
\end{figure}

4. \enquote{rectangularly} correlated data, i.e.
  we simulate
$k=5000$ repetitions of $(X_1, \ldots, X_n, Y_1\ldots, Y_n)$, where 
\begin{align}\label{eq:rec_transform}
\begin{pmatrix}X_1 & Y_1\\ X_2 & Y_2\\ \vdots & \vdots\\
X_n & Y_n \end{pmatrix}=\begin{pmatrix}\xi_1 & \eta_1\\ \xi_2 & \eta_2\\ \vdots & \vdots\\
\xi_n & \eta_n \end{pmatrix}\begin{pmatrix}\cos(\frac{\pi}{12}v) & -\sin(\frac{\pi}{12}v) \\ \sin(\frac{\pi}{12}v) & \cos(\frac{\pi}{12}v) \end{pmatrix}
\end{align}
and for  two independent fractional Gaussian noise sequences  $(Z_i)_{i\geq 1}$  and  $(\tilde{Z}_i)_{i\geq 1}$  each with parameter $H$,
$(X_1, \ldots, X_n)\defeq (2\Phi(Z_1)-1, \ldots, 2\Phi(Z_{n})-1)$ and
$(Y_1, \ldots, Y_n)\defeq~(2\Phi(\tilde{Z}_1)-~1, \ldots, 2\Phi(\tilde{Z}_{n})-1)$; see Figure \ref{fig:rec}  for an illustration of different parameter combinations.
 \begin{figure}[htbp]
\begin{center}
\label{fig:rec} 
\includegraphics[width=.75\textwidth]{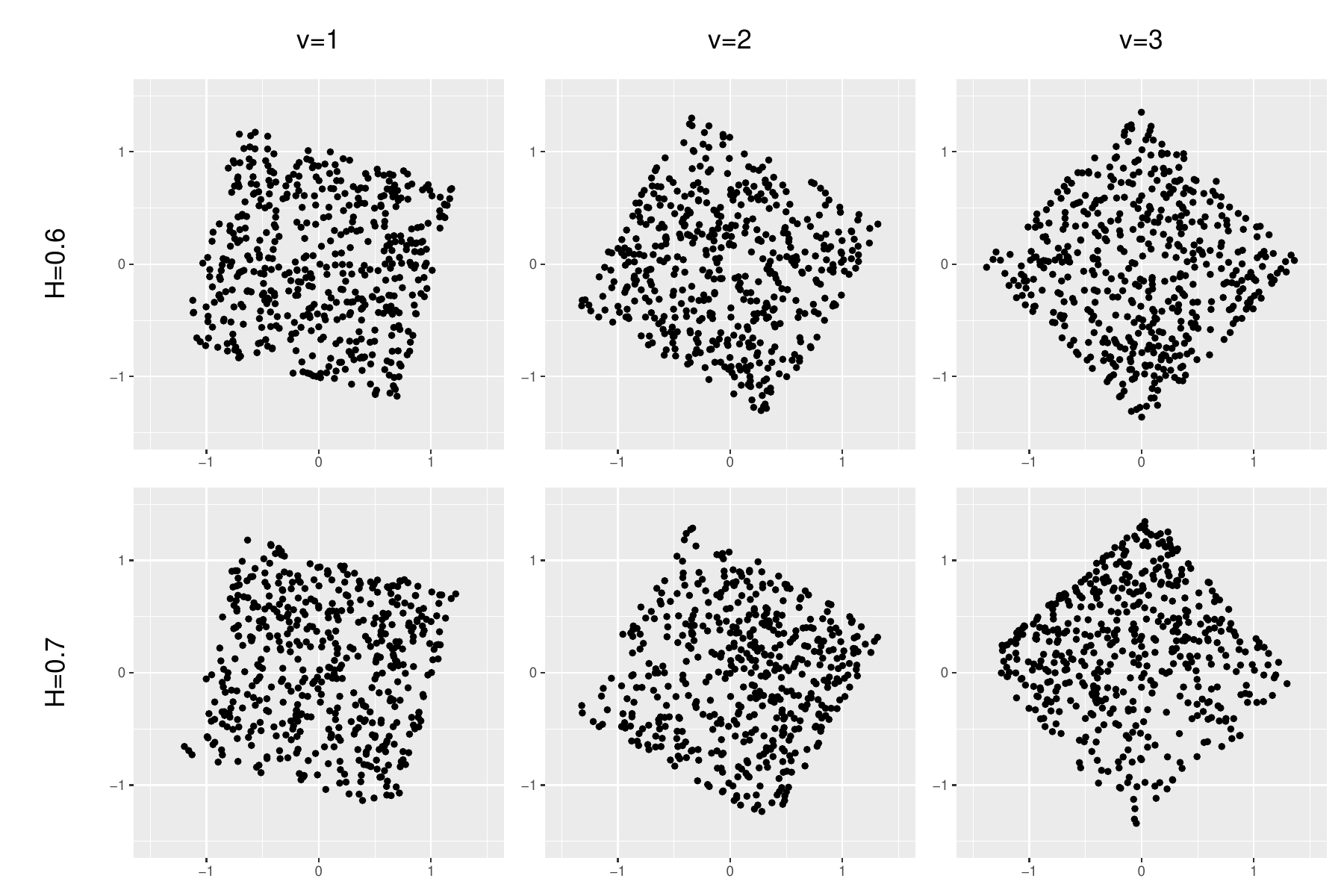}
\caption{\enquote{Rectangularly} correlated data $(X_i,Y_i)$,   $i=1, \ldots, 500$,  with parameters $H$ and $v$.}
\label{fig:rec}
\end{center}
\end{figure}

All calculations are based on $5,000$ realizations of simulated time series and test decisions are based on an application of the sampling-window method for a significance level of $5\%$, meaning that the values of the test statistics are compared to the   95\%-quantile of the empirical distribution function $\widehat{F}_{m_n, l_n}$ defined by  \eqref{eq:subsampling_estimator}. 
The  fractional Gaussian noise sequences are generated by the function \verb$simFGN0$ from the \verb$longmemo$ package in \verb$R$.
Detailed simulation results can be found in  Tables 
\ref{table:linear} -- \ref{table:rectangular} 
in the supplement. These display results for sample sizes $n=100, 300, 500, 1000$,   block lengths $l_n~=~\lfloor n^{\gamma}\rfloor$ with  $\gamma\in \left\{0.4, 0.5, 0.6\right\}$, Hurst parameters $H=0.6, 0.7$ and different values of the parameters $r$ and $v$.

 \begin{figure}[htbp]
\begin{center}
\label{fig:lin_dep_rej_rates} 
\includegraphics[width=.75\textwidth]{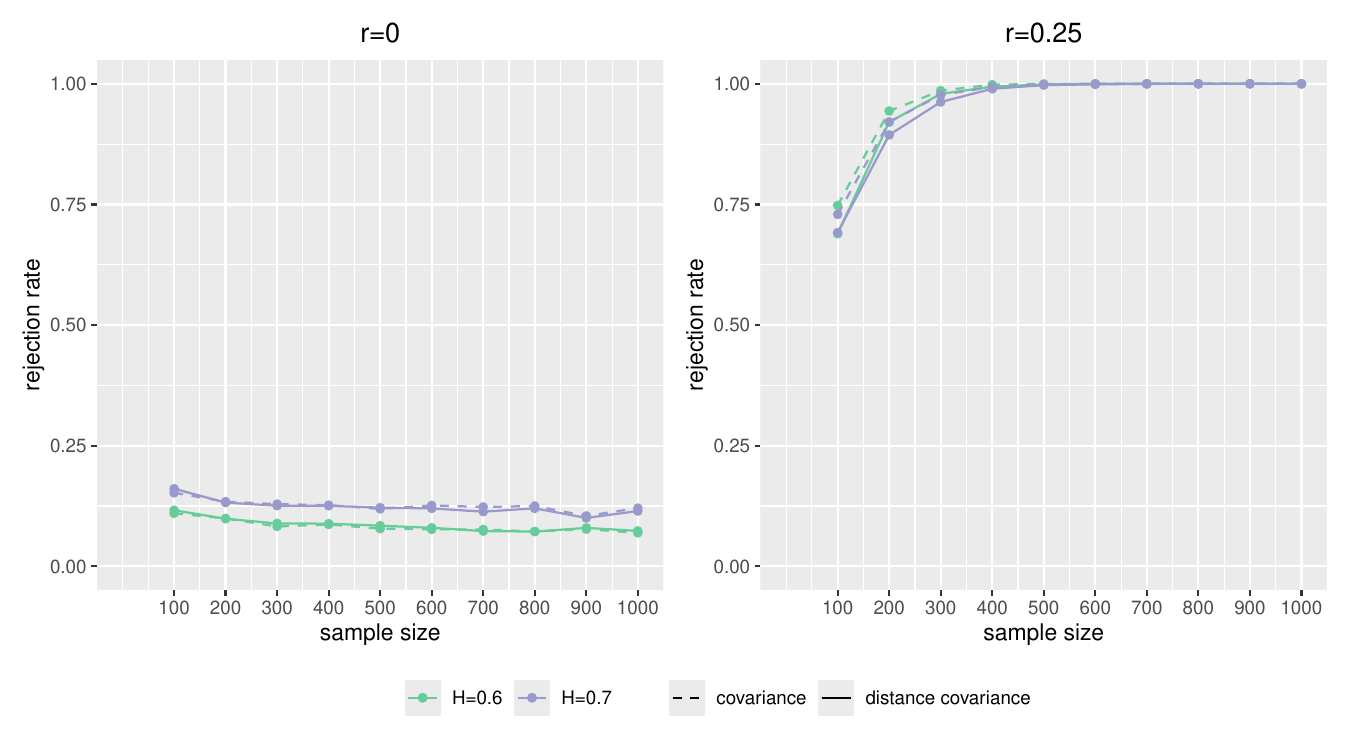}
\caption{Rejection rates of the hypothesis tests resulting from the empirical covariance  and the empirical distance covariance
obtained by  subsampling   based on \enquote{linearly} correlated fractional Gaussian noise time series $X_j$, $j=1, \ldots, n$, $Y_j$, $j=1, \ldots, n$ with block length $l_n~=~\lfloor \sqrt{n}\rfloor$, $d=0.1n$, Hurst parameters $H=0.6, 0.7$,  and cross-covariance $\Cov(X_i, Y_{j})=r \Cov(X_i, X_{j})$,  $1\leq i, j\leq n$, with $r=0$ and $r=0.25$.
The level of significance equals 5\%.}
\end{center}
\end{figure}
 \begin{figure}[htbp]
\begin{center}
\label{fig:squares_rej_rates} 
\includegraphics[width=.75\textwidth]{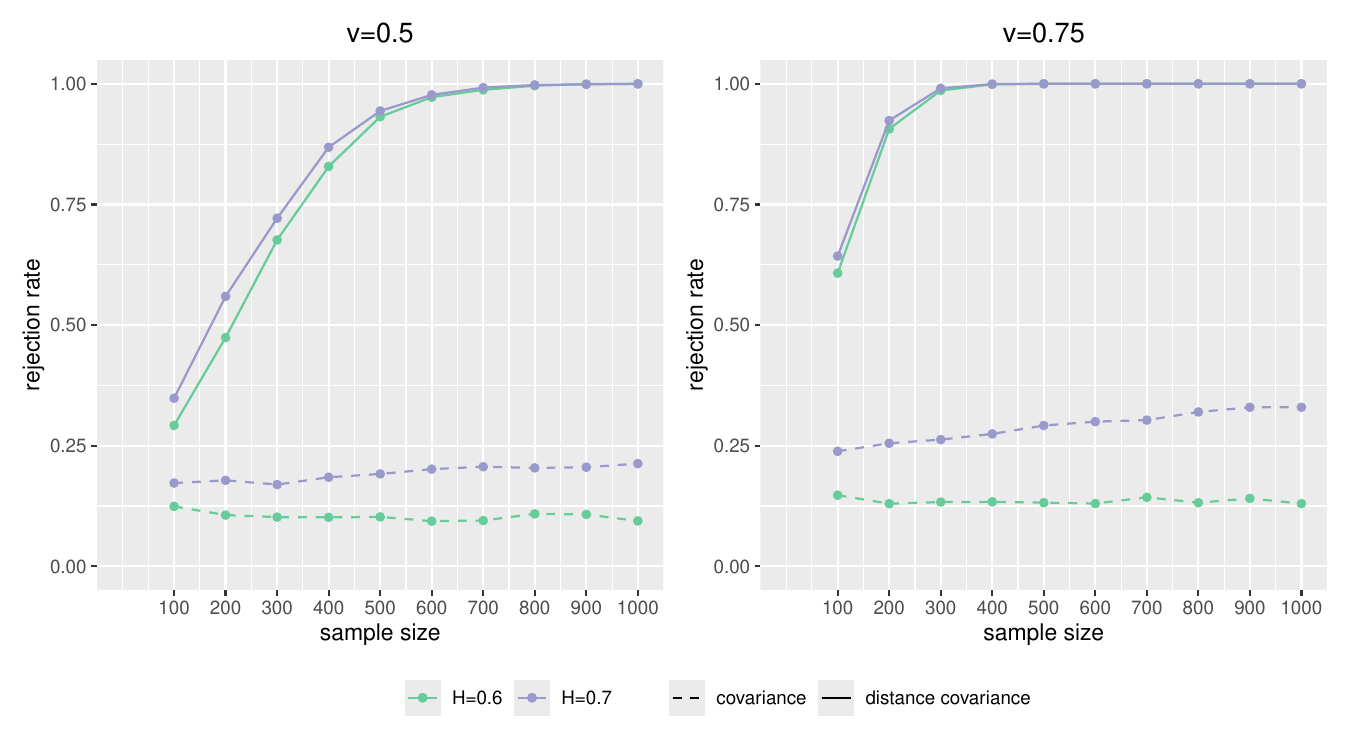}
\caption{Rejection rates of the hypothesis tests resulting from the empirical covariance  and the empirical distance covariance
obtained by  subsampling   based on  \enquote{parabolically} correlated  time series $X_j$, $j=~1, \ldots, n$, $Y_j$, $j=1, \ldots, n$ according to \eqref{eq:squares_transform} with block length $l_n~=~\lfloor \sqrt{n}\rfloor$,  $d=0.1n$, Hurst parameters $H=0.6, 0.7$,  and  with $v=0.5$ and $v=0.75$.
}
\end{center}
\end{figure}
 \begin{figure}[htbp]
\begin{center}
\label{fig:wave_rej_rates} 
\includegraphics[width=.75\textwidth]{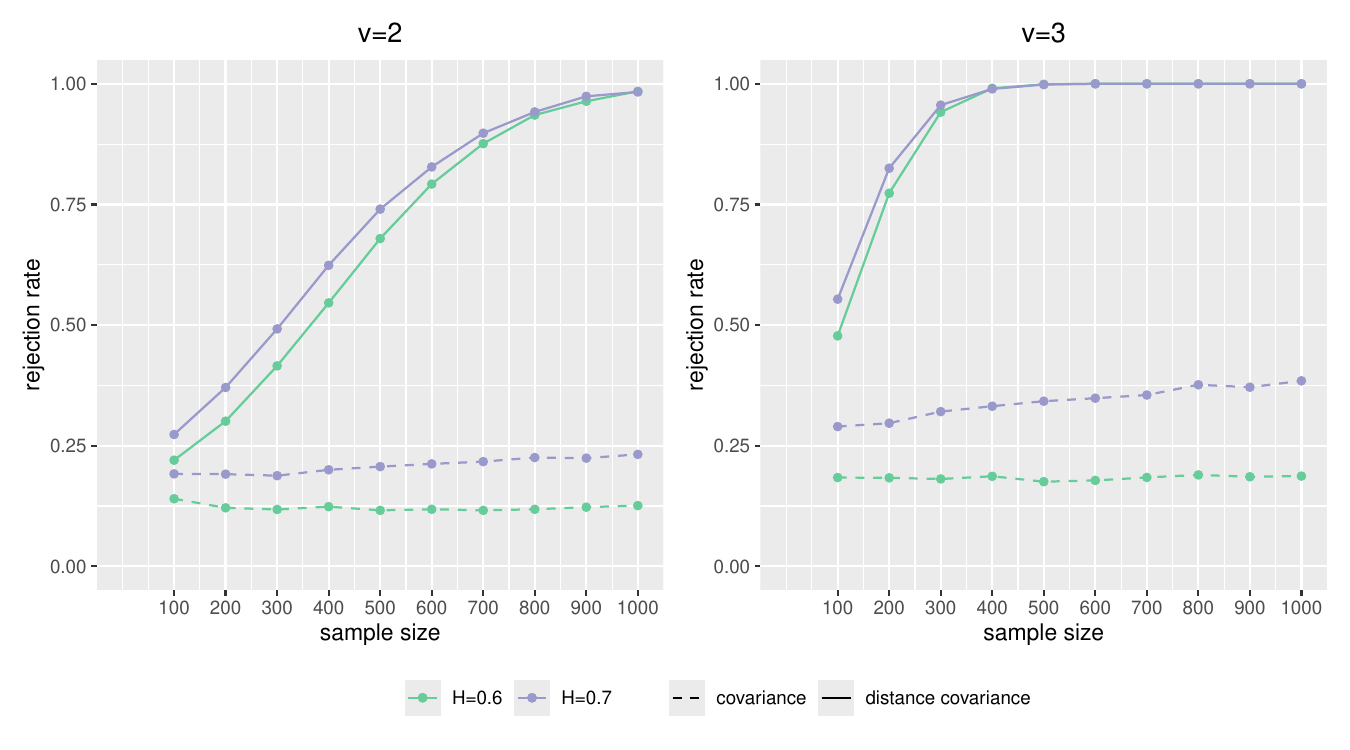}
\caption{Rejection rates of the hypothesis tests resulting from the empirical covariance  and the empirical distance covariance
obtained by  subsampling   based on  \enquote{wavily} correlated  time series $X_j$, $j=1, \ldots, n$, $Y_j$, $j=1, \ldots, n$ according to \eqref{eq:wave_transform} with block length $l_n~=~\lfloor \sqrt{n}\rfloor$,  $d=0.1n$,  Hurst parameters $H=0.6, 0.7$,  and with $v=2$ and $v=3$.
The level of significance equals 5\%.}
\end{center}
\end{figure}
 \begin{figure}[htbp]
\begin{center}
\label{fig:rec_rej_rates} 
\includegraphics[width=.75\textwidth]{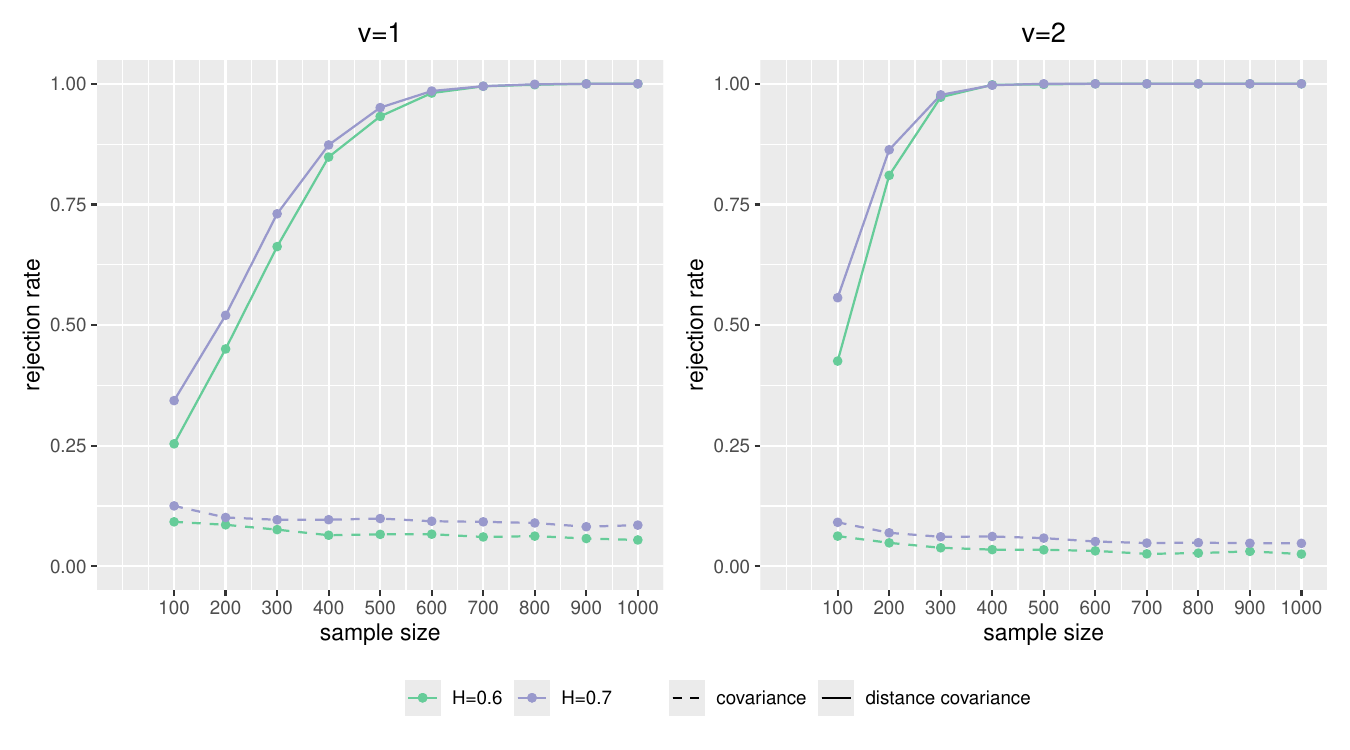}
\caption{Rejection rates of the hypothesis tests resulting from the empirical covariance  and the empirical distance covariance
obtained by  subsampling   based on  \enquote{rectangularly} correlated  time series $X_j$, $j=~1, \ldots, n$, $Y_j$, $j=1, \ldots, n$ according to \eqref{eq:rec_transform} with block length $l_n~=~\lfloor \sqrt{n}\rfloor$,  $d=0.1n$, Hurst parameters $H=0.6, 0.7$,  and with $v=1$ and $v=2$.
The level of significance equals 5\%.}
\end{center}
\end{figure}

As a whole, the simulation results concur with the expected behaviour of hypothesis tests for independence of time series: For both testing procedures, an increasing sample size goes along with an improvement of the finite sample performance of the test, i.e. the
empirical size (that can  be found in the columns of Table \ref{table:linear}   
superscribed by $r=0$) approaches the level of significance and the empirical power increases; stronger deviations from the hypothesis, i.e. an increase of the parameters $r$ and $v$ leads to an increase of the rejection rates.
Moreover, the testing procedures seem to be sensitive to a dependence within the individual time series, as an increase  of the Hurst parameter $H$ results in significantly higher or lower  rejection rates.

Both testing procedures tend to be oversized for small sample sizes. Table \ref{table:linear}  
shows that linear correlation as well as independence of two time series are slightly better detected by a test based on the empirical covariance than by a test based on the empirical distance covariance.
For linearly correlated data, an increase of dependence within the time series, i.e. an increase of the parameter $H$, results in a decrease of the empirical power of both testing procedures.
For \enquote{parabolically} and \enquote{wavily} correlated data, this observation can only be made with respect to the finite sample performance of the test based on  the empirical distance covariance. Most notably, in these two cases,  the test based on  the empirical distance covariance clearly outperforms the test based on the empirical covariance in that it yields decisively higher empirical power.
In addition to this, it seems remarkable that the test based on the empirical distance covariance 
interprets a rotation  of  data points  generated by independent stochastic processes as dependence between the coordinates, while the test that is based on the empirical covariance tends to classify these as being generated by independent processes.

\subsection{Data example}\label{sec:data}

In the following, the mean monthly discharges of three different rivers are analyzed with regard to 
cross-dependence between the corresponding data-generating processes by an application of the test statistics considered in the previous sections. 

The data was provided by the Global Runoff Data
Centre (GRDC) in  Koblenz, Germany; see \cite{grdc}. The GRDC is an international archive  currently comprising river discharge data of more than 9,900 stations from 159 countries.

The time series  we are considering consist of $n=96$ measurements of the mean monthly discharge from January 2000 to December 2007, i.e. a time period of 8 years, for the 
Amazon River, monitored at a station in S\~{a}o Paulo de Oliven\c{c}a, Brazil (corresponding to GRDC-No. 3623100), the Rhine,  monitored at a station in Cologne, Germany (corresponding to GRDC-No. 6335060), and the Juta\'{i} River,  a tributary of the Amazon River, monitored at a station in Coloca\c{c}\~{a}o Caxias
(corresponding to GRDC-No. 3624201).
(We chose the Juta\'{i} River because its discharge volume compares to that of the Rhine.)

As the discharge volume of rivers is affected by seasonalities and trends, we eliminated these effects from the original data sets by the Small Trend Method, see \cite{brockwell:davis:1991}, Chapter 1.4, p. 21, before our analysis. Figures  \ref{fig:rhine}, \ref{fig:amazonas}, and \ref{fig:jutai} depict the values of the detrended and deseasonalized  time series.

\begin{figure}[htbp]
\begin{center} 
\includegraphics[width=0.7\textwidth]{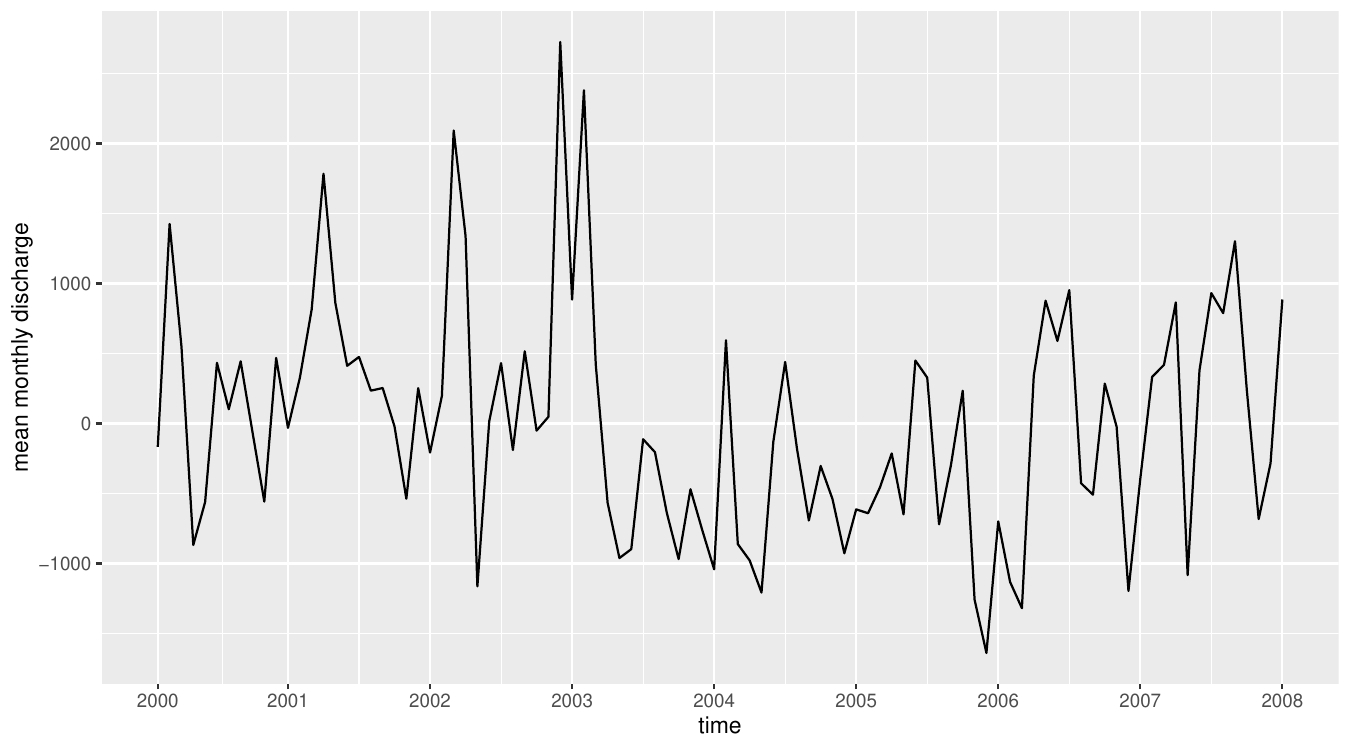}
\caption{Detrended and deseasonalized mean  monthly discharges of the  Rhine River.}\label{fig:rhine}
\end{center}
\end{figure}

\begin{figure}[htbp]
\begin{center} 
\includegraphics[width=0.7\textwidth]{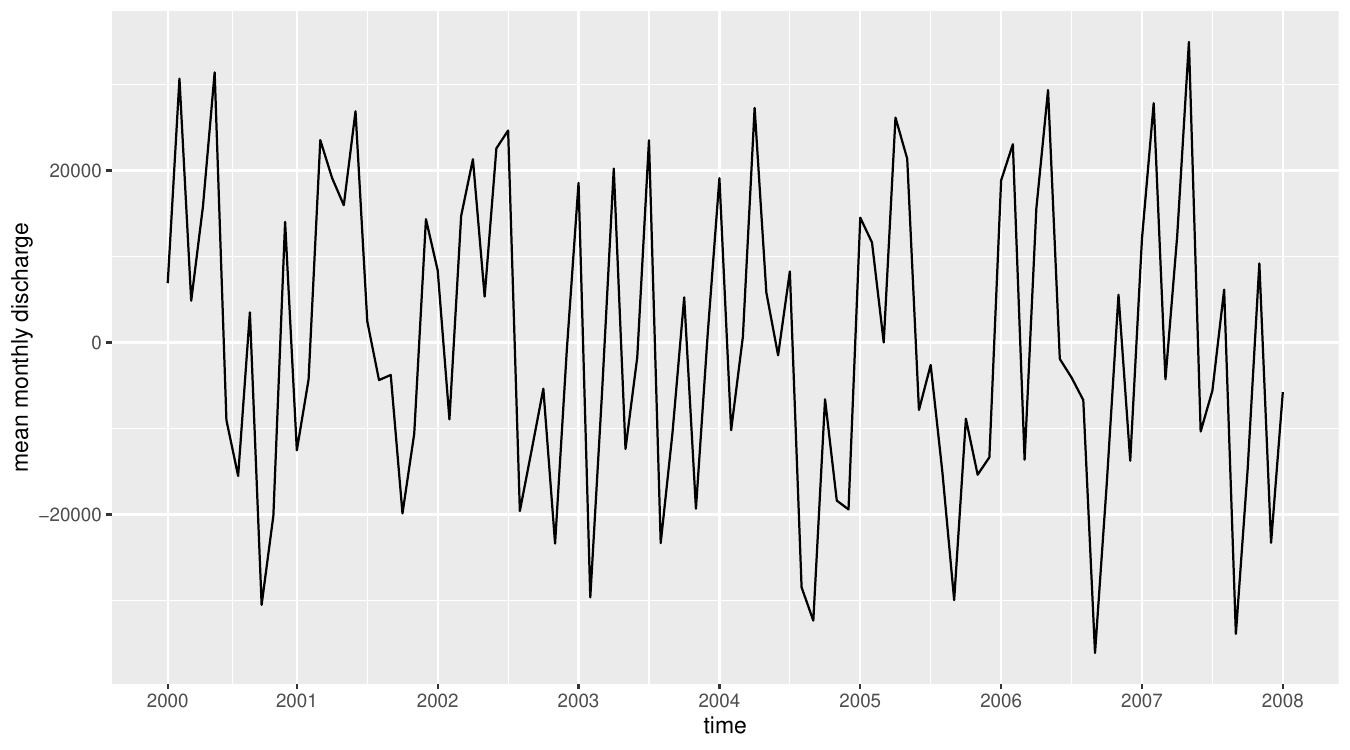}
\caption{Detrended and deseasonalized mean  monthly discharges of the   Amazon River.}\label{fig:amazonas}
\end{center}
\end{figure}

\begin{figure}[htbp]
\begin{center} 
\includegraphics[width=0.7\textwidth]{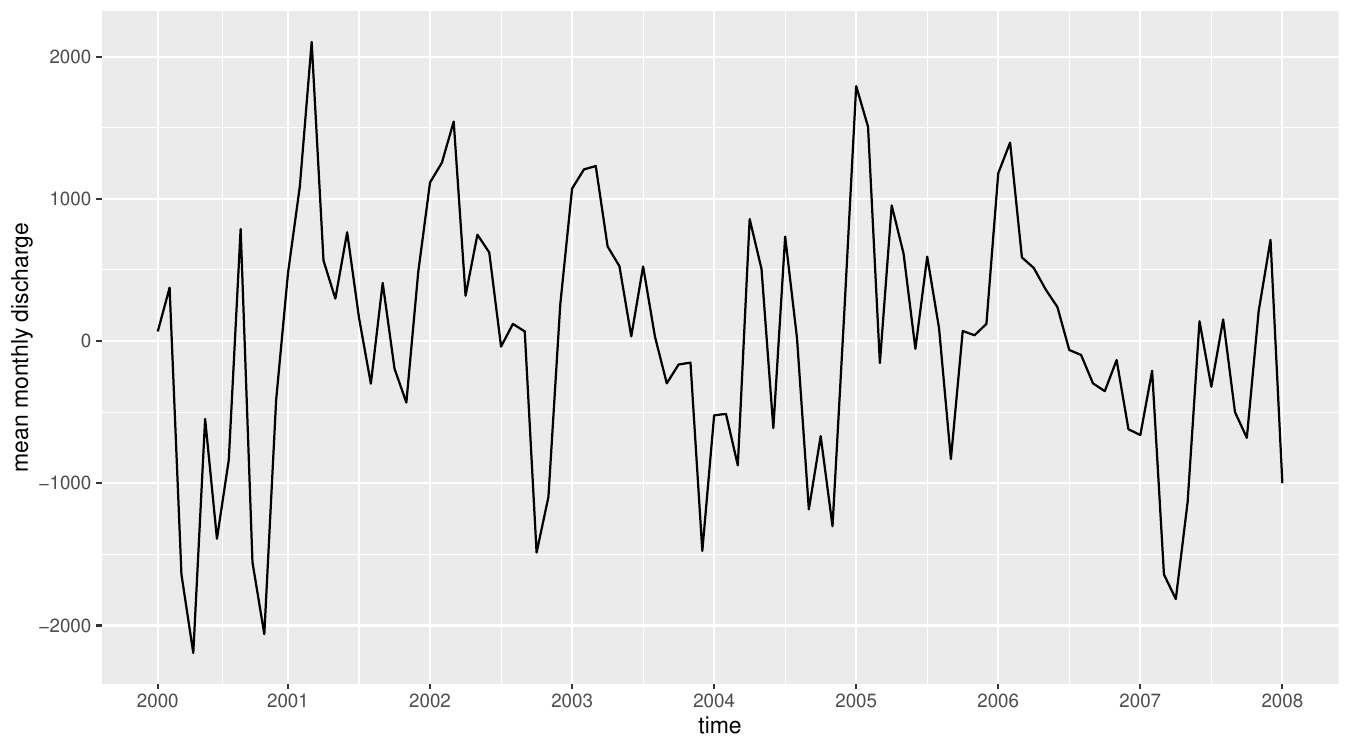}
\caption{Detrended and deseasonalized mean  monthly discharges of the Juta\'{i} River.}\label{fig:jutai}
\end{center}
\end{figure}

The mean monthly discharges of rivers typically display long-range dependence
characterized by a Hurst parameter $H$ that is close to $0.7$, meaning the long-range dependence parameter $D$ of the data-generating time series may be assumed  to be close to $0.6$. 
Under a corresponding assumption, 
a test decision based on the distance cross-covariance rejects the hypothesis  for large values of 
\begin{align*}
\sqrt{n}\int_{\mathbb{R}^{p}}\int_{\mathbb{R}^{q}}\left|\varphi_{X, Y}^{(n)}(s, t)-\varphi_X^{(n)}(s)\varphi_Y^{(n)}(t)\right|^2 s^{-2}t^{-2}ds dt,
\end{align*}
while a test decision
based on the  empirical cross-covariance rejects the hypothesis  for large values of 
\begin{align*}
\frac{1}{\sqrt{n}}\left|\sum\limits_{i=1}^n(X_i-\bar{X})(Y_i-\bar{Y})\right|.
\end{align*}

In our analysis, we apply both tests to the data.
We base test decisions on an approximation of the distribution of the test statistics by the sampling-window method with block size $l=~\lfloor\sqrt{n}\rfloor=9$. As  significance level we choose $\alpha=5\%$.

As the Rhine is geographically separated from the other two rivers, we expect the tests to decide in favor of the hypothesis of independence, when applied to the discharges of the Rhine and one of the Brazilian rivers.
Due to the fact that the  Juta\'{i} River is a tributary of the Amazon River, and due to
the spatial proximity of the two measuring stations in Brazil, which are  approximately 200 kilometers apart, we expect a test for independence of the discharge volumes of these two rivers to reject the hypothesis.
In fact, both tests do not reject the hypothesis of two independent time series when applied to the Rhine's discharge and the Amazon River's or Juta\'{i} River's discharge, respectively, and reject the hypothesis when applied to the discharges of the two Brazilian rivers.

\section*{Acknowledgments}

The authors would like to thank the Global Runoff Data Centre for providing the considered data.
Both authors were supported by 
  Collaborative Research Center SFB 823 {\em Statistical modelling of non-linear dynamic processes}.

\bibliographystyle{imsart-nameyear}

\bibliography{arXiv_distance_correlation}

\newpage

\begin{appendix}

\section{Gaussian subordination and long-range dependence}

This article  focuses on the consideration of  subordinated Gaussian time series, i.e. on random observations generated by  transformations of  Gaussian processes.

\begin{definition}\label{def_sub_gauss}
Let $(\xi_t)_{t\in T}$ be a Gaussian process with index set $T$.
A process $(Y_t)_{t\in~T}$ satisfying $Y_t=G(\xi_t)$ for some measurable function $G:\mathbb{R}\longrightarrow \mathbb{R}$ is called  subordinated Gaussian process.
\end{definition}

\begin{remark}
For any particular distribution function $F$, 
an appropriate choice of the transformation $G$ in Definition \ref{def_sub_gauss} 
yields subordinated Gaussian processes  with marginal distribution $F$.
Moreover, there exist algorithms for generating Gaussian processes that,
after suitable  transformation, yield subordinated Gaussian processes with marginal distribution $F$ and a predefined covariance structure;  see \cite{pipiras:taqqu:2017}.
\end{remark}

\subsection*{Univariate Hermite expansion}

The subordinated random variables $Y_t=G(\xi_t)$, $t\in~T$, can be considered as elements of the Hilbert space $L^2(\mathbb{R}, \varphi(x)dx)=~\mathcal{L}^2(\mathbb{R}, \varphi(x)dx)/\mathcal{N}$, where $\mathcal{L}^2(\mathbb{R}, \varphi(x)dx)$ denotes  the space of all measurable, real-valued  functions which are square-integrable with respect to the measure $\varphi(x)dx$ associated with the standard normal density function $\varphi$ and $\mathcal{N}\defeq~\text{ker}(\|\cdot\|_{L^2})$.
For two functions  $G_1, G_2 \in L^2(\mathbb{R}, \varphi(x)dx)$ the corresponding inner product is defined by
\begin{align}\label{eq:inner_product}
\langle G_1, G_2 \rangle_{L^2}\defeq\displaystyle\int_{-\infty}^{\infty}G_1(x)G_2(x)\varphi(x)dx=\E G_1(X)G_2(X)
\end{align}
 with $X$ denoting a standard normally distributed random variable.

A collection of orthogonal elements in $L^2(\mathbb{R}, \varphi(x)dx)$ is given by the sequence of Hermite polynomials; see \cite{pipiras:taqqu:2017}.

\begin{definition}
For $n\geq 0$, the Hermite polynomial of order $n$ is defined by
\begin{align*}
H_n(x)=(-1)^{n}\e^{\frac{1}{2}x^2}\frac{d^n}{d x^n}\e^{-\frac{1}{2}x^2}, \ x\in \mathbb{R}.
\end{align*}
\end{definition}

Orthogonality of the  sequence $(H_n)_{n\geq 0}$  in $L^2(\mathbb{R}, \varphi(x)dx)$ follows from
\begin{align*}
\langle H_n, H_m \rangle_{L^2}
=\begin{cases}
n! \ \ &\text{if $n=m$,}\\
0 \  \ &\text{if $n\neq m$.}
\end{cases} 
\end{align*}
 Moreover, it can be shown that
the Hermite polynomials form an orthogonal basis of 
$L^2(\mathbb{R}, \varphi(x)dx)$.  As a result,  every $G\in L^2(\mathbb{R}, \varphi(x)dx)$
has an expansion in Hermite polynomials, i.e. for $G\in L^2(\mathbb{R}, \varphi(x)dx)$ and $\xi$ standard normally distributed, we have
\begin{align}\label{eq:Hermite_expansion}
G(\xi)=\sum\limits_{r=0}^{\infty}\frac{J_r(G)}{r!}H_r(\xi),
\end{align}
where the so-called {\em Hermite coefficient} $J_r(G)$ is given by
\begin{align*}
J_r(G): =\langle G, H_r\rangle_{L^2}=\E G(X)H_r(X).
\end{align*}

Equation \eqref{eq:Hermite_expansion} 
holds in an $L^2$-sense, meaning
\begin{align*}
\lim\limits_{n\rightarrow \infty}\left\|G(\xi)-\sum\limits_{r=0}^{n}\frac{J_r(G)}{r!}H_r(\xi)\right \|_{L^2}
= 0, 
\end{align*}
where $\|\cdot\|_{L^2}$ denotes the norm induced by the inner product $\langle \cdot, \cdot\rangle_{L^2}$.

Given the Hermite expansion
\eqref{eq:Hermite_expansion}, it is possible to characterize the dependence structure of subordinated Gaussian time series  $G(\xi_n)$, $n\in \mathbb{N}$.
In fact, it holds that
\begin{align}\label{eq:cov_sub_Gaussian}
\Cov(G(\xi_1), G(\xi_{k+1}))=\sum\limits_{r=1}^{\infty}\frac{J^{2}_r(G)}{r!}\left(\rho(k)\right)^{r},
\end{align}
where $\rho$ denotes the auto-covariance function of $(\xi_n)_{n\geq 1}$; see \cite{pipiras:taqqu:2017}.
 Under the assumption that, as $k$ tends to $\infty$,  $\rho(k)$ converges to $0$ with a certain rate, the asymptotically dominating term in the series \eqref{eq:cov_sub_Gaussian} is the summand corresponding to the smallest integer $r$ for which the Hermite coefficient $J_r(G)$ is non-zero. This index, which decisively depends on $G$, is called {\em Hermite rank}.

\begin{definition}
Let $G\in L^2(\mathbb{R}, \varphi(x)dx)$ with $\E G(X)=0$ for standard normally distributed $X$ and let $J_r(G)$, $r\geq 0$, be the Hermite coefficients in the Hermite expansion of $G$. The smallest index $k\geq 1$ for which $J_k(G)\neq 0$ is called the  Hermite rank of $G$, i.e.
\begin{align*}
r\defeq \min\left\{k\geq 1: J_k(G)\neq 0\right\}.
\end{align*}
\end{definition}

\subsection*{Multivariate Hermite expansion}

Let $(\xi_t)_{t\in T}$, be a  multivariate Gaussian process with index set $T$. More precisely, assume that $\xi_t\defeq (\xi^{(1)}_t, \xi^{(2)}_t, \ldots, \xi^{(d)}_t)$ are Gaussian random vectors  with mean $\mathbf{0}\defeq (0, \ldots, 0)^{\top}$ and covariance matrix $\Sigma$.
We write  $\varphi_{\Sigma}$ for the corresponding density
and denote by $I_d$ the $d\times d$ identity matrix. 
Set $\mathbf{q}=(q_1, \ldots, q_d)^{\top}$, $\mathbf{q}=q_1!\cdots q_d!$, $|\mathbf{q}|=q_1+\ldots +q_d$, $\mathbf{x}=(x_1, \ldots, x_d)^{\top}$, $\mathbf{x}^{\mathbf{q}}=x_1^{q_1}\cdots x_d^{q_d}$, $\partial\mathbf{x}^{\mathbf{q}}=\partial x_1^{q_1}\cdots \partial x_d^{q_d}$ and
\begin{align*}
\left(\frac{d}{d\mathbf{x}}\right)^{\mathbf{q}}=
\frac{\partial^{|\mathbf{q}|}}{\partial \mathbf{x}^{\mathbf{q}}}=
\frac{\partial^{q_1+\ldots +q_d}}{\partial x_1^{q_1}\cdots \partial x_d^{q_d}}.
\end{align*}

Given a measurable function $G:\mathbb{R}^d\longrightarrow \mathbb{R}$, subordinated random variables $Y_t=G(\xi_t)$, $t\in T$, can be considered as elements of the Hilbert space $L^2(\Omega, \varphi_{\Sigma})=\mathcal{L}^2(\Omega, \varphi_{\Sigma})/\mathcal{N}$, where $\mathcal{L}^2(\Omega, \varphi_{\Sigma})$ denotes  the space of all measurable, real-valued  functions which are square-integrable with respect to the measure associated with the  density function $\varphi_{\Sigma}$ and $\mathcal{N}\defeq~\text{ker}(\|\cdot\|_{L^2})$.
For two functions  $G_1, G_2 \in L^2(\mathbb{R}^d, \varphi_{I_2})$ the corresponding inner product is defined by
\begin{align}\label{eq:inner_product_bivariate}
\langle G_1, G_2 \rangle_{L^2}\defeq\displaystyle\int_{-\infty}^{\infty}G_1(x)G_2(x)\varphi_{I_d}(x)dx=\E G_1(\mathbf{X})G_2(\mathbf{X})
\end{align}
 with $\mathbf{X}$ denoting a standard normally distributed, $d$-variate  random vector.

\begin{definition}
For $\mathbf{x}=(x_1, \ldots, x_d)^{\top}\in \mathbb{R}^d$ and $\mathbf{q}=(q_1, \ldots, q_k)^{\top}\in \mathbb{N}^d$ we call $H_{\mathbf{q}}$,  defined by
\begin{align*}
H_{\mathbf{q}}(\mathbf{x}; \Sigma)\defeq \frac{(-1)^{|\mathbf{q}|}}{\varphi_{\Sigma}(\mathbf{x})}\left(\frac{d}{d\mathbf{x}}\right)^{q}\varphi_{\Sigma}(\mathbf{x}),
\end{align*}
a multivariate  Hermite polynomial of degree $k=|\mathbf{q}|$.
\end{definition}

If $\Sigma=I_d$ then $H_{\mathbf{q}}(\mathbf{x}; \Sigma)=\prod_{j=1}^dH_{q_j}(x_j)$, i.e. if the components of the vector $\mathbf{X}$ are independent, then a multivariate Hermite polynomial is a product of univariate ones. 
In the following it is shown that, in fact, it is sufficient to consider Gaussian random vectors with independent, identically distributed components, i.e.
\begin{align*}
\tilde{\mathbf{X}}=(\tilde{X}_1, \ldots, \tilde{X}_d)^{\top}\sim \mathcal{N}(0, I_d).
\end{align*}
Let $G\in L^2(\mathbb{R}^d, \varphi_{I_d})$ and define
\begin{align*}
J(G, \tilde{\mathbf{X}}, \mathbf{q})=J(G, I_d, \mathbf{q})=\langle G, H^*_{\mathbf{q}}\rangle =\E\left(G(\tilde{\mathbf{X}})H^*_{\mathbf{q}}(\tilde{\mathbf{X}})\right),
\intertext{where}
H^*_{\mathbf{q}}(\mathbf{x})=H^*_{q_1, \ldots, q_k}(x_1, \ldots, x_d)=\prod\limits_{j=1}^dH_{q_j}(x_j).
\end{align*}
The Hermite rank $r(G, \tilde{\mathbf{X}})=r(G, I_d)$ of $G$
with respect to $\tilde{\mathbf{X}}$, i.e. with respect to the distribution $\mathcal{N}(0, I_d)$, is the largest integer $r$
such that $J(G, I_d, \mathbf{q})=0$ for all $0<|\mathbf{q}|<r$, where $|\mathbf{q}|=q_1+\ldots +q_d$. Note that this is the same as the largest integer $r$ such that 
\begin{align*}
\langle G(\tilde{\mathbf{X}}),\tilde{\mathbf{X}}^{\mathbf{q}} \rangle =\E\left(G(\tilde{\mathbf{X}})\prod\limits_{j=1}^{d}\tilde{X}_j^{q_j}\right)=0 \ \text{for all} \ 0<|\mathbf{q}|<r.
\end{align*}
As in the univariate case, we have the orthogonal expansion 
\begin{align*}
G(\tilde{X}_1, \ldots, \tilde{X}_d)=\E(G(\tilde{\mathbf{X}}))+\sum\limits_{|\mathbf{q}|\geq r(G, I_d)}\frac{J(G, I_d, \mathbf{q})}{q_1!\cdots q_d!}\prod\limits_{j=1}^d H_{q_j}(\tilde{X}_j).
\end{align*}

Since $\mathbf{X}\sim \mathcal{N}(0, \Sigma)$ is equal in distribution to $U(\tilde{\mathbf{X}})=\Sigma^{\frac{1}{2}}\tilde{\mathbf{X}}$,
we  have the expansion
\begin{align*}
G(\mathbf{X})=\tilde{G}(\tilde{\mathbf{X}})=
\E(G(\mathbf{X})+\sum\limits_{|\mathbf{q}|\geq r(\tilde{G}, I_d)}\frac{J(G\circ U, I_d, \mathbf{q})}{q_1!\cdots q_d!}\prod\limits_{j=1}^d H_{q_j}(\tilde{X}_j).
\end{align*}

\begin{definition}
Let $\mathbf{X}\sim \mathcal{N}(0, \Sigma)$ and $G\in L^2(\mathbb{R}^d, \varphi_{I_d})$.
We define the Hermite coefficients of $G$ with respect to $\mathbf{X}$ by
\begin{align*}
J(G, \mathbf{X}, \mathbf{q})=J(G, \Sigma, \mathbf{q})=\E\left(G(\mathbf{X})H_{\mathbf{q}}^{*}(\mathbf{X})\right).
\end{align*}
The Hermite rank $r(G, \mathbf{X})=r(G, \Sigma)$  is defined as the largest integer $r$ 
such that 
\begin{align*}
J(G, \Sigma, \mathbf{q})=0 \ \text{for all} \ 0<|\mathbf{q}|<r.
\end{align*}
\end{definition}

\begin{remark}
It follows that 
\begin{align*}
r(G, \Sigma)=r(G\circ U, I_d).
\end{align*}
\end{remark}

\subsection*{Long-range dependence}

In this article, we study the asymptotic behaviour of distance covariance for long-range dependent time series. 
The rate of decay  of the auto-covariance function is crucial to the definition of long-range dependent time series.
A relatively slow decay of the auto-covariances characterizes long-range dependent time series, while a relatively fast decay characterizes short-range dependent processes; see \cite{pipiras:taqqu:2017}, p. 17.

\begin{definition}
A  (second-order) stationary, real-valued time series $(X_k)_{k\geq 1}$, is called  long-range dependent if its auto-covariance function $\rho$ satisfies
\begin{align*}
\rho(k)\defeq\Cov(X_1, X_{k+1})\sim k^{-D}L(k), \ \ \text{as $k\rightarrow \infty$,}
\end{align*}
with $D\in \left(0, 1\right)$ for  some slowly varying function $L$.  We refer to $D$ as long-range dependence (LRD) parameter. 
\end{definition}

It follows from  \eqref{eq:cov_sub_Gaussian} that subordination of long-range dependent Gaussian time series  potentially generates time series whose auto-covariances  decay faster than the auto-covariances of the underlying Gaussian process. In some cases, the subordinated time series is long-range dependent as well, in other cases subordination may even yield short-range dependence.
Given that 
  $\Cov(\xi_1, \xi_{k+1})\sim k^{-D}L(k)$, as $k\rightarrow\infty$, for some slowly varying function $L$ and $D\in (0,1)$
and given that $G\in  L^2(\mathbb{R}, \varphi(x)dx)$ is a function with Hermite rank $r$, we have
\begin{align*}
\Cov(G(\xi_1), G(\xi_{k+1}))\sim J_r^2(G)r!k^{-Dr}L^r(k), \ \ \text{as $k\rightarrow \infty$.}
\end{align*}

It immediately follows  that subordinated Gaussian time series  $G(\xi_n)$, $n\geq 1$, are long-range dependent  with LRD parameter $D_G\defeq Dr$ and slowly-varying function $L_G(k)=~J_r^2(G)r!L^r(k)$ whenever $Dr<1$.

\section{Proofs}\label{app:proofs}

\begin{proof}[Proof of Lemma  \ref{lemma:existence_1}]
	Define $u_j(s)\defeq\exp(isX_j)-\varphi_X(s)$. Then, it follows that
	\begin{align*}
	\varphi_X^{(n)}(s)-\varphi_X(s)
	=\frac{1}{n}\sum\limits_{j=1}^nu_j(s).
	\end{align*}
	
	In order to show that $\varphi_X^{(n)}(s)-\varphi_X(s)$ is an element 
	of the vector space $L^{2}(\mathbb{R},  w(s)ds)$, it thus suffices to show that 
	$u_j\in L^{2}(\mathbb{R},  w(s)ds)$.
	For this, note that
	\begin{align*}
	|u_j(s)|^2
	&=\left(\exp(isX_j)-\E\left(\exp(isX_j)\right)\right)\left(\exp(-isX_j)-\E\left(\exp(-isX_j)\right)\right)\\
	&=1+\varphi_X(s)\bar{\varphi}_X(s)-\exp(isX_j)\bar{\varphi}_X(s)
	-\exp(-isX_j)\varphi_X(s).
	\end{align*}
	For $X'\overset{\mathcal{D}}{=}X$, independent of $X$, and $\E_X$ denoting the expected value taken with respect to $X$, we have
	\begin{align*}
	\varphi_X(s)\bar{\varphi}_X(s)
	&=\E\left[\left(\cos(sX)+i\sin(sX)\right)\left(\cos(sX')-i\sin(sX')\right)\right]\\
	&=\E\left[\cos(sX)\cos(sX')+\sin(sX)\sin(sX')\right]\\
	&=\E\left[\cos(s(X-X'))\right],
	\end{align*}
	\begin{align*}
	&\exp(isX_j)\bar{\varphi}_X(s)\\
	=&\left(\cos(sX_j)+i\sin(sX_j)\right)\E\left(\cos(sX)-i\sin(sX)\right)\\
	=&\E_X\left(\cos(sX_j)\cos(sX)+\sin(sX)\sin(sX_j)\right)\\
	&+i\sin(sX_j)\E\left(\cos(sX)\right)-i\cos(sX_j)\E\left(\sin(sX)\right)\\
	=&\E_X\left(\cos (s(X_j-X))\right)+i\sin(sX_j)\E\left(\cos(sX)\right)-i\cos(sX_j)\E\left(\sin(sX)\right), 
	\end{align*}
	and
	\begin{align*}
	&\exp(-isX_j)\varphi_X(s)\\
	=&\left(\cos(sX_j)-i\sin(sX_j)\right)\E\left(\cos(sX)+i\sin(sX)\right)\\
	=&\E_X\left(\cos(sX_j)\cos(sX)+\sin(sX)\sin(sX_j)\right)\\
	&-i\sin(sX_j)\E\left(\cos(sX)\right)+i\cos(sX_j)\E\left(\sin(sX)\right)\\
	=&\E_X\left(\cos (s(X_j-X))\right)-i\sin(sX_j)\E\left(\cos(sX)\right)+i\cos(sX_j)\E\left(\sin(sX)\right).
	\end{align*}
	
	It follows that
	\begin{align*}
	|u_j(s)|^2=2 \E_X\left[1-\cos(s(X_j-X))\right]-\E\left[1-\cos(s(X-X'))\right].
	\end{align*}
	
	As a result, and according to Lemma 1 in 
	\cite{szekely:2007}, we arrive at
	\begin{align*}
	&\int_{\mathbb{R}}\frac{|u_j(s)|^2}{|s|^2}ds\\
	=&\int_{\mathbb{R}}\frac{1}{|s|^2}\left(2 \E_X\left[1-\cos(s(X_j-X))\right]-\E\left[1-\cos(s(X-X'))\right]\right)ds\\
	=&\left(2 \E_X\left[\int_{\mathbb{R}}\frac{1}{|s|^2}\left(1-\cos(s(X_j-X))\right)ds\right]-\E\left[\int_{\mathbb{R}}\frac{1}{|s|^2}\left(1-\cos(s(X-X'))\right)ds\right]\right)\\
		=&2	\text{$\pi$}\E_X(|X_j-X|)-	\text{$\pi$}\E|X-X'|\leq 2	\text{$\pi$}\left(|X_j|+\E|X|\right).
	\end{align*}
	Consequently, $\varphi_X^{(n)}(s)-\varphi_X(s)$ takes values in $L^2\left(\mathbb{R}, w(s)ds\right)$.

	Furthermore, it holds that
	\begin{multline*}
	\int_{\mathbb{R}}\int_{\mathbb{R}}\left|\left(\varphi_Y^{(n)}(t)-\varphi_Y(t)\right)\left(\varphi_X^{(n)}(s)-\varphi_X(s)\right)\right|^2w(s, t)ds dt\\
	=
	\int_{\mathbb{R}}\left|\varphi_X^{(n)}(s)-\varphi_X(s)\right|^2(cs^2)^{-1}ds \int_{\mathbb{R}}\left|\varphi_Y^{(n)}(t)-\varphi_Y(t)
	\right|^2(ct^2)^{-1}dt 
	<\infty, 
	\end{multline*}
	i.e. $\left(\varphi_Y^{(n)}(t)-\varphi_Y(t)\right)\left(\varphi_X^{(n)}(s)-\varphi_X(s)\right)$ takes values in $L^2\left(\mathbb{R}^2, w(s, t)dsdt\right)$.
\end{proof}

\begin{proof}[Proof of Theorem \ref{thm:test_statistic}]
In order to show convergence of the test statistic we apply Theorem 4.2 in \cite{billingsley:1968}.
For this, we define 
\begin{align*}
&V_n\defeq\sum\limits_{h=0}^{\infty}a_{h}
\dcov_n(X,Y;h), \ \
V\defeq \sum\limits_{h=0}^{\infty}a_{h}Z_h, \\
& U_{m, n}\defeq  \sum\limits_{h=0}^{m}a_{h}\dcov_n(X,Y;h), \ \
U_m\defeq  \sum\limits_{h=0}^{m}a_{h}Z_h.
\end{align*}
Based on
Proposition \ref{prop:convergence_SRD}
 it follows from an application of the continuous mapping theorem that
 $
U_{m, n} \overset{\mathcal{D}}{\longrightarrow} U_m$,   $n\rightarrow \infty$.
According to the proof of Proposition \ref{prop:convergence_SRD}
\begin{align*}
&\E \left|Z_n(s, t; h)\right|^{2}\leq f(s, t) \ \text{for all $\pr{s, t}\in \mathbb{R}^2$, $\; n\in \mathbb{N}$,}
\end{align*}
where $f$ is a positive, $w(s, t)dsdt$-integrable function that is independent of $h$.
Since $\E|\dcov_n(X,Y;h)|=\E \left|Z_n(s, t; h)\right|^{2}+o(1)$, $\E \left|Z_n(s, t; h)\right|^{2}\longrightarrow \E \left|Z(s, t; h)\right|^{2}$, and $Z_h=\int_{\mathbb{R}}\int_{\mathbb{R}}\left|Z(s, t; h)\right|^{2}w(s, t)dsdt$ by definition, it follows that 
\begin{align*}
 \E\left|U_{m}-V\right|\leq \sum\limits_{h=m+1}^{\infty}\left|a_h\right|\E\left|Z_h\right|\leq C\sum\limits_{h=m+1}^{\infty}\left|a_h\right|.
\end{align*}
The right-hand side of the above inequality goes to $0$
since, by assumption,
 $\lim_{m\rightarrow\infty}\sum_{h=m+1}^{\infty}|a_h|=0$.
Consequently, $U_{m} \overset{\mathcal{D}}{\longrightarrow} V$,   $m\rightarrow \infty$.
For any $\epsilon>0$ it holds that  
\begin{align*}
P\left(\left|U_{m, n}-V_n\right|\geq \epsilon\right)
\leq \frac{1}{\epsilon}\E\left|U_{m, n}-V_n\right|\leq \sum\limits_{h=m+1}^{\infty}|a_h|\E\left|\dcov_n(X,Y;h)\right|.
\end{align*}
Since  $\E\left(\left|\int_{\R^p}  \int_{\R^q} \big|Z_n(s, t; h)\big|^2 w(s,t) ds \,  dt\right|\right)\leq C$ for some of $h$ independent constant $C$, and for all $n\in\mathbb{N}$,  and  since  $\lim_{m\rightarrow\infty}\sum_{h=m+1}^{\infty}|a_h|=0$ by assumption, it follows that
\begin{align*}
\lim\limits_{m\rightarrow\infty}\limsup\limits_{n\rightarrow\infty }P\left(\left|U_{m, n}-V_n\right|\geq \epsilon\right)=0.
\end{align*}
 Theorem 4.2 in \cite{billingsley:1968} therefore yields $V_n \overset{\mathcal{D}}{\longrightarrow} V$.
\end{proof}

\begin{proof}[Proof of Theorem \ref{thm:subsampling}]

In order  to establish the validity of the subsampling procedure, note that  the triangular inequality  yields
\begin{align}\label{eq:triangular_inequality}
|\widehat{F}_{m_n,  l_n}(t)-F_{T_n}(t)|\leq |\widehat{F}_{m_n,  l_n}(t)-F_T(t)|+|F_{T}(t)-F_{T_n}(t)|.
\end{align}
The second term on the right-hand side of the  inequality converges to $0$ for all points of continuity $t$ of $F_T$ if
the  statistics $T_n$, $n\in\mathbb{N}$, are measurable and converge in distribution to a (non-degenerate) random variable $T$   with distribution function $F_T$.

It remains to show that the first summand on the right-hand side of inequality \eqref{eq:triangular_inequality} converges to $0$ as well.
 As $L^2$-convergence implies convergence in probability,
it suffices to show that $\lim_{n\rightarrow\infty}\E|\widehat{F}_{m_n, l_n}(t)-F_T(t)|^2=0$.
For this purpose, we consider the following bias-variance decomposition:
 \begin{align*}
  \E\left(|\widehat{F}_{m_n, l_n}(t)-F_T(t)|^2\right)
    =\Var\pr{\widehat{F}_{m_n, l_n}(t)}+\left|\E \widehat{F}_{m_n, l_n}(t)-F_T(t)\right|^2.
 \end{align*}

Stationarity and independence of the processes $(X_n)_{n\geq 1}$ and $(Y_n)_{n\geq 1}$  imply
that
$\E \widehat{F}_{m_n, l_n}(t) =F_{T_{l_n}}(t)$, so that,
due to  the convergence of $T_{l_n}$ to $T$, the bias term of the above equation converges to $0$ as $l_n$ tends to $\infty$.

As a result, it remains  to show that  the variance term vanishes as $n$ tends to $\infty$. 
Initially, note that
 \begin{align*}
 &\Var\pr{\widehat{F}_{m_n, l_n}(t)}\\
              =&\frac{1}{m_n}\Var \pr{1_{\left\{T_{l_n, 1}\leq t\right\}}}+\frac{2}{m_n^2}\sum\limits_{k=2}^{m_n}(m_n-k+1)\Cov\left(1_{\left\{T_{l_n, 1}\leq t\right\}}, 1_{\left\{T_{l_n, k}\leq t\right\}}\right)\\
 \leq &\frac{2}{m_n}\sum\limits_{k=1}^{m_n}\left|\Cov\left(1_{\left\{T_{l_n, 1}\leq t\right\}}, 1_{\left\{T_{l_n, k}\leq t\right\}}\right)\right|
\end{align*}
due to stationarity.

Since $T_{l_n, k}=T_{l_n}\left(X_k, \ldots, X_{k+l_n-1}, Y_{k+d_n}, \ldots, Y_{k+d_n+l_n-1}\right)$,
\begin{align*}
&\left|\Cov\left(1_{\left\{T_{l_n, 1}\leq t\right\}}, 1_{\left\{T_{l_n, k}\leq t\right\}}\right)\right|\\
&\leq
\rho\left(\sigma(X_i,  Y_{d_n+i}, 1\leq i\leq l_n), \sigma(X_j, Y_{d_n+j},  k\leq j\leq k+l_n-1)\right),
\end{align*}
where $\sigma(X_i,  Y_{d_n+i}, 1\leq i\leq l_n)$ and  $\sigma(X_j, Y_{d_n+j},  k\leq j\leq k+l_n-1)$)  denote the $\sigma$-fields generated by the random variables $X_i,  Y_{d_n+i}, 1\leq i\leq l_n$, and $X_j, Y_{d_n+j},  k\leq j\leq k+l_n-1$), respectively.

For  $\beta
\in (0,1)$, 
we split the sum of covariances into two parts:
\begin{align*}
&\frac{1}{m_n}\sum\limits_{k=1}^{m_n}\left|\Cov\left(1_{\left\{T_{l_n, 1}\leq t\right\}}, 1_{\left\{T_{l_n, k}\leq t\right\}}\right)\right|\\
=&\frac{1}{m_n}\sum\limits_{k=1}^{\lfloor n^{\beta}\rfloor}\!\!\left|\Cov\left(1_{\left\{T_{l_n, 1}\leq t\right\}}, 1_{\left\{T_{l_n, k}\leq t\right\}}\right)\right|
+\frac{1}{m_n}\sum\limits_{k=\lfloor n^{\beta}\rfloor+1}^{m_n}\!\!\!\!\left|\Cov\left(1_{\left\{T_{l_n, 1}\leq t\right\}}, 1_{\left\{T_{l_n, k}\leq t\right\}}\right)\right|\\
\leq&\frac{\lfloor n^{\beta}\rfloor}{m_n}+\frac{1}{m_n}\sum\limits_{k=\lfloor n^{\beta}\rfloor+1}^{m_n}\rho\left(\sigma(X_i,  Y_{d_n+i}, 1\leq i\leq l_n), \sigma(X_j, Y_{d_n+j},  k\leq j\leq k+l_n-1)\right)\\
\leq&\frac{\lfloor n^{\beta}\rfloor}{m_n}+\frac{1}{m_n}\sum\limits_{k=\lfloor n^{\beta}\rfloor+1}^{m_n}\rho_{k, l_n, d_n},
\end{align*} 
where
\begin{align*}
\rho_{k, l_n, d_n}\defeq\rho\left(\sigma(X_i,  Y_{d_n+i}, 1\leq i\leq l_n), \sigma(X_j, Y_{d_n+j},  k\leq j\leq k+l_n-1)\right).
\end{align*}
The first summand on the right-hand side of the inequality converges to $0$ if $l_n=o(n)$ and $d_n=o(n)$.
In order to show that the second summand converges to $0$,   a sufficiently good approximation to the sum of maximal correlations is needed.
In particular, we have to show that
\begin{align*}
\sum\limits_{k=\lfloor n^{\beta}\rfloor+1}^{m_n}\rho_{k, l_n, d_n}=o(m_n).
\end{align*}

According to \cite{bradley:2005}, Theorem 5.1, 
\begin{align*}
\rho(\mathcal{A}_1\vee \mathcal{A}_2, \mathcal{B}_1\vee \mathcal{B}_2)\leq\max\left\{\rho(\mathcal{A}_1, \mathcal{B}_1), \rho(\mathcal{A}_2, \mathcal{B}_2)\right\}, 
\end{align*}
if $\mathcal{A}_1\vee \mathcal{B}_1$
 and  $\mathcal{A}_2\vee \mathcal{B}_2$ are independent. 
 
As a result (and because of stationarity), it holds that
\begin{align*}
\rho_{k, l_n, d_n}\leq \max\{\rho_{k, l_n, X}, \rho_{k, l_n, Y}\},
\end{align*} 
where
\begin{align*}
\rho_{k, l_n, X}\defeq &\rho\left(\sigma(X_i,  1\leq i\leq l_n), \sigma(X_j,   k\leq j\leq k+l_n-1)\right), \\
\rho_{k, l_n, Y}\defeq &\rho\left(\sigma(Y_{i}, 1\leq i\leq l_n), \sigma(Y_{j},  k\leq j\leq k+l_n-1)\right).
\end{align*}

For this reason, it suffices to show that 
\begin{align*}
\sum\limits_{k=\lfloor n^{\beta}\rfloor+1}^{m_n}\rho_{k, l_n, X}=\sum\limits_{k=\lfloor n^{\beta}\rfloor+1}^{m_n}\rho_{k, l_n, Y}=o(m_n).
\end{align*} 

\cite{betken:wendler:2016} establish the following result:
\begin{lemma}[\cite{betken:wendler:2016}]\label{LemC} 
Given a time series $\xi_k$, $k\geq 1$, satisfying    Assumptions \ref{ass:spectral_density_2} and \ref{ass:covariance_function}, 
there exist constants $C_1,C_2\in (0, \infty)$, such that
\begin{multline*}
\rho\left(\sigma(\xi_i, 1\leq i\leq l), \sigma(\xi_j, k+l\leq j\leq k+2l-1)\right)\\
\leq C_1l^{D}k^{-D}L_{\rho}(k)+C_2l^2k^{-D-1}\max\{L_{\rho}(k),1\}
\end{multline*}
for all $k\in\mathbb{N}$ and all $l\in\{l_k,\ldots,k\}$.
\end{lemma}

We consider $\sum_{k=\lfloor n^{\beta}\rfloor+1}^{m_n}\rho_{k, l_n, X}$ only, since the same argument yields $$\sum\limits_{k=\lfloor n^{\beta}\rfloor+1}^{m_n}\rho_{k, l_n, Y}=~o(m_n).$$

Let $\varepsilon>0$. By assumption,   $l_n\leq C_l n^{\alpha}$ for $\alpha\defeq\frac{1}{2}(1+D_X)-\varepsilon$ and some constant $C_l\in (0, \infty)$. 
As a consequence of Potter's Theorem, for every $\delta >0$, there exists a constant $C_{\delta}\in (0, \infty)$ such that $L_{\rho}(k)\leq C_{\delta}k^{\delta}$ for all $k\in \mathbb{N}$; see Theorem 1.5.6 in  \cite{bingham:1987}.

Moreover, we choose $\beta>\alpha$ and $n$ large enough such that $l_n< \frac{1}{2}\lfloor n^{\beta}\rfloor$.
According to this, Lemma \ref{LemC} yields
\begin{align*}
\frac{1}{m_n}\sum\limits_{k=\lfloor n^{\beta}\rfloor+1}^{m_n}\rho_{k, l_n, X}
=&\frac{1}{m_n}\sum\limits_{k=\lfloor n^{\beta}\rfloor-l_n+1}^{m_n-l_n}\rho_{k+l_n, l_n, X}\\
\leq &C_1l_n^{D_X}\frac{1}{m_n}\sum\limits_{k=\lfloor n^{\beta}\rfloor-l_n+1}^{m_n-l_n}k^{-D_X}L_{\rho}(k)\\
&+C_2
\frac{l_n^2}{m_n}\sum\limits_{k=\lfloor n^{\beta}\rfloor-l_n+1}^{m_n-l_n}k^{-D_X-1}\max\{L_{\rho}(k),1\}\\
\leq &C_{\delta}C_1\frac{l_n^{D_X}}{m_n}\sum\limits_{k=\lfloor n^{\beta}\rfloor/2}^{m_n-l_n}k^{-D_X+\delta}\\
&+C_{\delta}C_2\frac{l_n^2}{m_n}\sum\limits_{k=\lfloor n^{\beta}\rfloor/2}^{m_n-l_n}k^{-D_X-1+\delta}\\
\leq &C\left(n^{D_X\alpha-\beta D_X+\beta\delta} 
+ n^{2\alpha-\beta D_X-\beta +\delta\beta}\right)
\end{align*}
for some constant $C\in (0, \infty)$. 
By definition of $\alpha$ and for a suitable choice of $\beta$, the right-hand side of the above inequality converges to $0$.
 \end{proof}

\begin{proof}[Proof of Theorem \ref{thm:consistency_dcov_n}]
By the decomposition (2) in the main document,
 we obtain
\begin{align}
&\varphi_{X,Y;h}^{(n)}(s,t) -\varphi_X^{(n)}(s)\, \varphi_Y^{(n)}(t)\label{eq:basic-decomp}  \\
 =& -(\varphi_X^{(n)} (s)-\varphi_X(s))(\varphi_Y^{(n)}(t) -\varphi_Y(t))
\nonumber\\
& \quad + \frac{1}{n} \sum_{j=1}^{n-h} (\exp(isX_j)-\varphi_X(s)) (\exp(itY_{j+h})-\varphi_Y(t)) +o(1). \nonumber
\end{align}
We now study the terms on the right hand side separately. We first apply the ergodic theorem for Hilbert space-valued random variables to the $L^2(\R, w(s) \, ds)$-valued process $(\exp(i\, s \, X_j)-\varphi(s))_{j\geq 1}$, and obtain
\[
  \int \Big| \frac{1}{n} \sum_{j=1}^n e^{i\, s\, X_j} -\varphi_X(s)  \Big|^2 w(s)\, ds \longrightarrow 0,
\]
almost surely. In the same way, we have $\int \Big| \frac{1}{n} \sum_{j=1}^n e^{i\, t\, Y_j} -\varphi_Y(t)  \Big|^2 w(t)\, dt \longrightarrow 0 $, almost surely, and thus we finally get
\begin{align*}
   &\iint \Big|   \Big( \frac{1}{n} \sum_{j=1}^n e^{i\, s\, X_j} -\varphi_X(s)     \Big)  \Big(\frac{1}{n} \sum_{j=1}^n e^{i\, t\, Y_j} -\varphi_Y(t)   \Big) \Big|^2 w(s,t) \, ds\, dt \\
     =  &\Big(  \int \Big| \frac{1}{n} \sum_{j=1}^n e^{i\, s\, X_j} -\varphi_X(s)  \Big|^2 w(s)\, ds \Big)
  \Big( \int \Big| \frac{1}{n} \sum_{j=1}^n e^{i\, t\, Y_j} -\varphi_Y(t)  \Big|^2 w(t)\, dt\Big)  \longrightarrow 0,
\end{align*}
again almost surely, as $n\rightarrow \infty$. In order to analyze the second term in \eqref{eq:basic-decomp}, we apply the ergodic theorem for Hilbert space-valued random variables to the  $L^2(\R^2, w(s,t) \, ds)$-valued process $\big((\exp(i\, s \, X_j)-\varphi(s))(\exp(i\, t\, Y_{j+h})-\varphi_Y(t))\big)_{j\geq 1}$. Observe that 
\[
  \E\Big( (\exp(i\, s \, X_j)-\varphi(s))(\exp(i\, t\, Y_{j+h})-\varphi_Y(t)) \Big)=\varphi_{X,Y;h}(s,t)-\varphi_X(s)\varphi_Y(t),
\]
and hence we obtain  
\begin{multline*}
   \iint \Big| \frac{1}{n} \sum_{j=1}^{n-h} \big(\exp^{isX_j}-\varphi_X(s)\big) \big(e^{itY_j}-\varphi_Y(t)\big) \\
   -  \big(\varphi_{X,Y;h}(s,t)-\varphi_X(s)\varphi_Y(t)\big) \Big|^2 w(s,t)\, ds\, dt 
   \rightarrow 0, 
\end{multline*}
which implies almost sure convergence of $\iint \big|(\exp(i\, s \, X_j)-\varphi(s))(\exp(i\, t\, Y_{j+h})-\varphi_Y(t)) \big|^2 w(s,t)\, ds\,dt$
to $\iint \big|\varphi_{X,Y;h}(s,t)-\varphi_X(s)\varphi_Y(t)\big|^2 w(s,t) \, ds\, dt$.
\end{proof}

\begin{proof}[Proof of Theorem \ref{thm:1_parameter}]
Since, due to Fubini's theorem,   
\[
\int_{\mathbb{R}} \E \left|f_t(X_1)\right|^2w(t)dt=\E\left(\|f_{X}\|_{2}^2\right)<\infty
\] 
by assumption, we have  $\E \left|f_t(X_1)\right|^2<\infty$ for almost every $t$, so that it is possible to expand the function $f_t$ in Hermite polynomials, meaning that 
\begin{align*}
f_t(X_j)-\E f_t(X_j)\overset{L^2}{=}\sum\limits_{q=r}^{\infty}\frac{J_q(t)}{q!}H_q(X_j), 
\end{align*}
i.e.
\begin{align*}
\lim\limits_{n\rightarrow \infty}\left\|f_t(X_j)-\E f_t(X_j)-\sum\limits_{q=r}^{n}\frac{J_q(t)}{q!}H_q(X_j)\right \|_{L^2}
= 0, 
\end{align*}
where $\|\cdot\|_{L^2}$ denotes the norm induced by the inner product \eqref{eq:inner_product}.

We will see that the first summand in the Hermite expansion of the function $f_t$ determines the asymptotic behaviour of the  sum.

To this end, we  show  $L_2$-convergence of
\begin{align*}
 n^{\frac{rD}{2}-1}L^{-\frac{r}{2}}(n)\sum\limits_{j=1}^n\left(f_t(X_j)-\E  f_t(X_j) -\frac{1}{r!}J_r(t)H_r(X_j)\right).
\end{align*}
 Fubini's theorem  yields
\begin{align*}
&\E\left(\int_{\mathbb{R}} \left|n^{\frac{rD}{2}-1}L^{-\frac{r}{2}}(n)\sum\limits_{j=1}^n\left(f_t(X_j)-\E f_t(X_j)-\frac{1}{r!}J_r(t)H_r(X_j)\right)\right|^2w(t)dt\right)\\
&=\int_{\mathbb{R}} \E\left(\left|n^{\frac{rD}{2}-1}L^{-\frac{r}{2}}(n)\sum\limits_{j=1}^n\left(f_t(X_j)-\E f_t(X_j)-\frac{1}{r!}J_r(t)H_r(X_j)\right)\right|^2\right)w(t)dt.
\end{align*}
Since $\E\pr{H_q(X_i)H_{q'}(X_j)}=0$ for $q\neq q'$ and $\E\pr{H_q(X_i)H_{q}(X_j)}=q!\rho(i-j)^q$, we have
\begin{align*}
&\E\left(\left|\sum\limits_{j=1}^n\left(f_t(X_j)-\E f_t(X_j)\right)  -\frac{1}{r!}J_r(t)\sum\limits_{j=1}^nH_r(X_j)\right|^2\right)\\
=&\E\left(\left|\sum\limits_{j=1}^n\sum\limits_{q=r+1}^{\infty}\frac{1}{q!}J_q(t)H_q(X_j)\right|^2\right)\\
 =&\sum\limits_{q=r+1}^{\infty}\frac{1}{q!^2}\left|J_q(t)\right|^2
\sum\limits_{i=1}^n\sum\limits_{j=1}^n\E\left(H_q(X_i)H_q(X_j)\right)\\
\leq &\sum\limits_{q=r+1}^{\infty}\frac{1}{q!}\left|J_q(t)\right|^2
\sum\limits_{i=1}^n\sum\limits_{j=1}^n|\rho(i-j)|^{q}\\
\leq &
\sum\limits_{q=r+1}^{\infty}\frac{1}{q!}\left|J_q(t)\right|^2
\sum\limits_{i=1}^n\sum\limits_{j=1}^n|\rho(i-j)|^{r+1}.
\end{align*}
In general, i.e. for an auto-covariance function $
\rho(k)= k^{-D}L(k), \  \text{as } k\rightarrow \infty,$ 
where $0~<~D<~1$ and where $L$ is a slowly varying function, it holds that
\begin{align*}
\sum\limits_{i=1}^n\sum\limits_{j=1}^n|\rho
(i-j)|^{r+1}=\mathcal{O}\pr{n^{1\vee (2-(r+1)D)}L^{\prime}(n)},
\end{align*}
where $L^{\prime}$ is some slowly varying function; see p. 1777 in 
\cite{dehling:taqqu:1989}. 

As a result, the previous considerations establish
\begin{align*}
&\E\left(\int_{\mathbb{R}} \left|n^{\frac{rD}{2}-1}L^{-\frac{r}{2}}(n)\sum\limits_{j=1}^n\left(f_t(X_j)-\E f_t(X_j)-\frac{1}{r!}J_r(t)H_r(X_j)\right)\right|^2w(t)dt\right)\\
&=\mathcal{O}\pr{n^{rD-2}L^{-r}(n)n^{1\vee (2-(r+1)D)}L^{\prime}(n)\int_{\mathbb{R}} \sum\limits_{q=r+1}^{\infty}\frac{1}{q!}\left|J_q(t)\right|^2w(t)dt}.
\end{align*}
Since $\sum_{q=r+1}^{\infty}\frac{1}{q!}\left|J_q(t)\right|^2\leq\sum_{q=1}^{\infty}\frac{1}{q!}\left|J_q(t)\right|^2=\E \left|f_t(X_1)-\E f_t(X_1)\right|^2$,
 we  conclude that the right-hand side of the above equality is
\begin{align*}
\mathcal{O}\left(n^{-\min{(1-rD, D)}}\tilde{L}(n)\int_{\mathbb{R}} \E \left|f_t(X_1)\right|^2w(t)dt\right)
\end{align*}
for some slowly varying function $\tilde{L}$.
This expression is $o(n^{-\delta})$ for some $\delta>0$ as $D<\frac{1}{r}$ and $\int_{\mathbb{R}} \E \left|f_t(X_1)\right|^2w(t)dt=\E\left(\|f_{X}\|_{2}^2\right)<\infty$ by assumption. 
The assertion then follows from the fact that 
\begin{align*}
n^{\frac{rD}{2}-1}L^{-\frac{r}{2}}(n)\sum\limits_{j=1}^nH_r(X_j)\overset{\mathcal{D}}{\longrightarrow} Z_{r, H}(1);
\end{align*}
see \cite{taqqu:1979} and \cite{dobrushin:major:1979}.
\end{proof}

\begin{proof}[Proof of Proposition \ref{prop:reduction}]
 Note that 
\begin{align*}
 &\left(\varphi_X^{(n)}(s)-\varphi_X(s)\right)\left(\varphi_Y^{(n)}(t)-\varphi_Y(t)\right)-
 J_1(s)\frac{1}{n}\sum\limits_{i=1}^nX_i J_1(t)\frac{1}{n}\sum\limits_{j=1}^nY_j 
\\
=&
\left(\varphi_X^{(n)}(s)-\varphi_X(s)-J_1(s)\frac{1}{n}\sum\limits_{i=1}^nX_i\right)\left(\varphi_Y^{(n)}(t)-\varphi_Y(t)\right)\\
&+ J_1(s)\frac{1}{n}\sum\limits_{i=1}^nX_i\left(\varphi_Y^{(n)}(t)-\varphi_Y(t)-J_1(t)\frac{1}{n}\sum\limits_{j=1}^nY_j\right).
\end{align*}
For the first summand on the right-hand side of the above equation Corollary  2.2
implies
\begin{align*}
  &\left\|\varphi_X^{(n)}(s)-\varphi_X(s)- J_1(s)\frac{1}{n}\sum\limits_{i=1}^nX_i
 \right\|_{2}\left\|\varphi_Y^{(n)}(t)-\varphi_Y(t)\right\|_{2}\\
 &=o_P\left(n^{-\frac{D_X}{2}}L_X^{\frac{1}{2}}(n)\right)
 \mathcal{O}_P\left(n^{-\frac{D_Y}{2}}L_Y^{\frac{1}{2}}(n)\right)\\
 &=o_P\left(n^{-\frac{D_X}{2}}L_X^{\frac{1}{2}}(n)n^{-\frac{D_Y}{2}}L_Y^{\frac{1}{2}}(n)\right).
\end{align*}
The second summand is    $o_P\left(n^{-\frac{D_X}{2}}L_X^{\frac{1}{2}}(n)n^{-\frac{D_Y}{2}}L_Y^{\frac{1}{2}}(n)\right)$ by means of the same argument.
\end{proof}

\begin{proof}[Proof of Proposition \ref{prop:reduction_LRD_2}]

Define the function 
\[
  f_{s,t}(x,y) := (e^{i\, s\, x}-\varphi_X(s)) (e^{i\, t\, y}-\varphi_Y(t)).
\]
For fixed values of $(s,t)$, this is a bounded function of $(x,y)$, which allows for an expansion in bivariate Hermite polynomials
\[
  f_{s,t}(x,y) =\sum_{q=r}^\infty \sum_{k,l\geq 0: k+l=q} \frac{J_{k,l}}{k!\, l!} H_k(x) H_l(y),
\]
where $J_{k,l}(s,t)=\E[f_{s,t}(X_1,Y_{1+h})H_k(X_1)H_l(Y_{1+h})]$, and where $r=r(s,t) =\min\{k+l: J_{k,l}(s,t)=0\}$ Note that, by indenpendence of the processes $(X_j){j\geq 1}$ and $(Y_j)_{j\geq 1}$, the Hermite coefficient $J_{k,l}(s,t)$ does not depend on the lag $h$. In addition, note that $H_0(x) \equiv 1$, and thus $J_{0,l}(s,t)=J_{k,0}(s,t)\equiv 0$ for all  indices $k,l\geq 0$. Thus, we obtain
\[
    f_{s,t}(x,y) =\sum_{q=3}^\infty \sum_{k,l\geq 0: k+l=q} \frac{J_{k,l}}{k!\, l!} H_k(x) H_l(y).
\]
This is an expansion in the Hilbert space $L_2(\R^2,\varphi_I)$, and thus we obtain
\begin{equation}
\sum_{j=1}^{n-h} f_{s,t}(X_j,Y_{j+h})  = \sum_{q=2}^\infty \sum_{k,l\geq 1: k+l=q} \frac{J_{k,l}(s,t)}{k!\, l!} 
\sum_{j=1}^{n-h} H_k(X_j)\, H_l(Y_{j+h}).
\label{eq:hermite}
\end{equation}
In what follows, we will show that the sum on the right hand side is dominated by the lowest order term 
$J_{1,1}(s,t) \sum_{j=1}^{n-h}X_j\, Y_{j+h}$. First, we observe that 
\begin{align*}
J_{1,1}(s,t) 
& =  \E\big[ (e^{i\, s\, X_1} -\varphi_X(s)) \, X_1\, (e^{i\, t\, Y_1} -\varphi_Y(t))\, Y_1  \big] \\
&= \E[X_1 \, e^{i\,s\,X_1}]  \,  \E[Y_1 \, e^{i\,s\,Y_1}]  =(i\, s\, e^{-\frac{s^2}{2}}) (i\, t\, e^{-\frac{t^2}{2}}) = -s\, t\, e^{-\frac{s^2+t^2}{2}}.
\end{align*}
Using this identity, the fact that $J_{0,2}(s,t)=J_{2,0}(s,t)=0$, and equation \eqref{eq:hermite}, we finally obtain
\[
 \sum_{j=1}^{n-h} \Big\{ f_{s,t}(X_j,Y_{j+h}) +s\, t\, e^{-\frac{s^2+t^2}{2}} \, X_j\, Y_{j+h}  \Big\}
 = \sum_{q=3}^\infty \sum_{k,l\geq 1:k+l=q} \frac{J_{k,l}(s,t)}{k!\, l!} 
\sum_{j=1}^{n-h} H_k(X_j)\, H_l(Y_{j+h}).
\]
By orthogononality of the Hermite polynomials, and using independence of the processes $(X_j)_{j\geq 1}$ and $(Y_j)_{j\geq 1}$, we obtain 
\begin{align*}
 & \E\big[ H_{k_1}(X_i) H_{l_1}(Y_{i+h}) H_{k_2}(X_j)H_{l_2}(Y_{j+h})   \big]\\[1mm]
 &\qquad = \E\big[ H_{k_1}(X_i)  H_{k_2}(X_j) \big] \E\big[ H_{l_1}(Y_{i+h}) H_{l_2}(Y_{j+h})   \big]\\[1mm]
 & \qquad =
 \left\{  
   \begin{array}{cl}
    k_1!l_1!\big[1\wedge  |j-i|^{-(k_1D_X+l_1D_Y)}L_X^{k_1}(j-i) L_Y^{l_1}(j-i)\big] & \mbox{ if } k_1=k_2 \mbox{ and } l_1=l_2 \\[1mm]
    0 & \mbox{ otherwise. }
   \end{array}
 \right.
\end{align*} 
Thus, 
\begin{align*}
 & \E\Big[\Big( \sum_{j=1}^{n-h} \big\{ f_{s,t}(X_j,Y_{j+h}) +s\, t\, e^{-\frac{s^2+t^2}{2}} \, X_j\, Y_{j+h}  \big\}
   \Big)^2  \Big]\\
 &\qquad   = \sum_{q=3}^\infty \sum_{k,l\geq 1: k+l=q} \frac{J_{k,l}^2(s,t)}{k! l!} \sum_{i=1}^{n-h} \sum_{j=1}^{n-h}  
\Big[1\wedge  |j-i|^{-(k \, D_X+l \, D_Y)}L_X^{k }(j-i) L_Y^{l}(j-i) \Big].
\end{align*}
By Karamata's theorem, a slowly varying function $L(n)$ grows slower than any power of $n$, i.e. for any $\epsilon>0$ there exists a constant $C=C_\epsilon$ such that 
\[
  L(n)\leq C\, n^\epsilon.
\]
Now, choose $\epsilon>0$ so small that that $\epsilon<\min(D_X,D_Y)/3$, and let $C\geq 1$ be such that $L_X(n)\leq C\, n^\epsilon$ and $L_Y(n)\leq C\, n^\epsilon$. Moreover, we assume without loss of generality that $D_X\leq D_Y$. Then we obtain
\begin{align*}
 \sum_{i=1}^{n-h} \sum_{j=1}^{n-h} \big[ 1\wedge  |j-i|^{-(k \, D_X+l \, D_Y)}L_X^{k }(j-i) L_Y^{l}(j-i)\big] &  \leq n+ 
 C \sum_{1\leq i\neq j \leq n-h}^{n-h}   |j-i|^{-(k\,(D_X-\epsilon) +l\, (D_Y-\epsilon))} \\
 &\leq n+ 2\, C\, \sum_{m=1}^{n} (n-m) m^{-(k\,(D_X-\epsilon) +l\, (D_Y-\epsilon))} \\
 & \leq 2\, C\, n\, \big(1+\sum_{m=1}^{n} m^{-(2\,(D_X-\epsilon) + (D_Y-\epsilon))} \big) \\
 & \leq \tilde{C}  \, n^{1\vee (2- 2(D_X-\epsilon)-(D_Y-\epsilon)) } \log n \\
 & = \tilde{C}\,  n^{1\vee((2-D_X-D_Y) -(D_X-3\epsilon) ) } \log n \\
 & = \tilde{C}\,  n^{2-D_X-D_Y} n^{-[(1-(D_X+D_Y))\wedge (D_X-3\, \epsilon)]} \log n,
\end{align*}
where $\tilde{C}=\tilde{C}(\epsilon, D_X,D_Y)$ is a constant that is independent of $k,l$ for $k, l\geq 1$ satisfying $k+l\geq 3$. 
Furthermore, we note that 
\[
  \sum_{q=3}^\infty \sum_{k, l\geq 1:k+l=q} \frac{J_{k,l}^2(s,t)}{k!l!}\leq
   \sum_{q=0}^\infty \sum_{k, l\geq 1:k+l=q} \frac{J_{k,l}^2(s,t)}{k!l!} =\E f_{s,t}^2(X_j,Y_{j+h}).
\]
Thus, putting everything together, we get
\begin{align*}
 & \E\Big[\Big( \sum_{j=1}^{n-h} \big\{ f_{s,t}(X_j,Y_{j+h}) +s\, t\, e^{-\frac{s^2+t^2}{2}} \, X_j\, Y_{j+h}  \big\}
   \Big)^2  \Big]\\
 & \leq   \tilde{C}\,  \big[n^{2-D_X-D_Y} n^{-[(1-(D_X+D_Y))\wedge (D_X-3\, \epsilon)]} \log n \big] \E f_{s,t}^2(X_j,Y_{j+h})   ,
\end{align*}
and thus we obtain by Fubini's theorem
\begin{align*}
& \E\Big[  \big\|  \frac{1}{n} \sum_{j=1}^{n-h} (e^{isX_j} -\varphi_X(s))(e^{itY_{j+h}} -\varphi_Y(t)) 
  + s\, t\, e^{-\frac{s^2+t^2}{2}} \frac{1}{n}\, \sum_{i=1}^{n-h}  X_jY_{j+h}   \big\|_2^2\Big] \\
&\quad= 
\E \Big[ \iint_{\R^2} \Big(\frac{1}{n} \sum_{j=1}^{n-h} \Big\{ f_{s,t}(X_j,Y_{j+h}) +s\, t\, e^{-\frac{s^2+t^2}{2}} \, X_j\, Y_{j+h}  \Big\}
   \Big)^2  ds\, dt \Big] \\
   & \quad \leq \tilde{C} \,  \big[n^{-(D_X+D_Y)} n^{-[(1-(D_X+D_Y))\wedge (D_X-3\, \epsilon)]} \log n \big] \iint_{\R^2} \E f_{s,t}^2(X_j,Y_{j+h})\,  ds\, dt \\
   &\quad =o\Big(n^{-(D_X+D_Y)} L_X(n) L_Y(n) \Big).
\end{align*}
In the last step, we have used the fact that $D_X+D_Y<1$, that $3\, \epsilon<D_X$, and that the integral $\iint_{\R^2} \E f_{s,t}^2(X_j,Y_{j+h})\,  ds\, dt $ is finite by the proof of Lemma \ref{lemma:existence_1}.
\end{proof}

\begin{proof}[Proof of Proposition \ref{prop:conv_LRD}]
It  holds that
\begin{align*}
\sum\limits_{i=1}^n\sum\limits_{j=1}^nX_iY_j
=&\sum\limits_{j=1}^n\int_{\left[-\pi, \pi\right)}\left(\e^{ix}\right)^jZ_{G, X}(dx)\sum\limits_{k=1}^n\int\limits_{\left[-\pi, \pi\right)}\left(\e^{iy}\right)^kZ_{G, Y}(dy)\\
=& \int_{\left[-\pi, \pi\right)^2}\e^{i(x+y)}\left(\sum\limits_{j=0}^{n-1}\left(\e^{ix}\right)^j\right)\left(\sum\limits_{j=0}^{n-1}\left(\e^{iy}\right)^j\right)Z_{G, X}(dx)Z_{G, Y}(dy)\\
=&\int_{\left[-\pi, \pi\right)^2} \e^{i(x+y)}\left(\frac{\e^{ixn}-1}{\e^{ix}-1}\right)\left(\frac{\e^{iyn}-1}{\e^{iy}-1}\right)Z_{G, X}(dx)Z_{G, Y}(dy).
\end{align*}

By the change of variables formula, it follows that
\begin{align*}
&n^{\frac{D_X+D_Y}{2}-2}L_X^{-\frac{1}{2}}(n)L_Y^{-\frac{1}{2}}(n)\sum\limits_{i=1}^n\sum\limits_{j=1}^nX_iY_j\\=&n^{\frac{D_X+D_Y}{2}-2}L_X^{-\frac{1}{2}}(n)L_Y^{-\frac{1}{2}}(n)\int\limits_{\left[-\pi, \pi\right)^2}   \e^{i(x+y)}\left(\frac{\e^{ixn}-1}{\e^{ix}-1}\right)\left(\frac{\e^{iyn}-1}{\e^{iy}-1}\right)
Z_{G, X}(dx)Z_{G, Y}(dy)\\
=&\int\limits_{\left[-n\pi, n\pi\right)^2} \e^{i\frac{x+y}{n}}\frac{1}{n^2}\left(\frac{\e^{ix}-1}{\e^{i\frac{x}{n}}-1}\right)\left(\frac{\e^{iy}-1}{\e^{i\frac{y}{n}}-1}\right)
Z_{G, X}^{(n)}(dx)Z_{G, Y}^{(n)}(dy),
\end{align*}
where 
\begin{align*}
&Z_{G, X}^{(n)}(A)=\sqrt{n^{D_X}L_X^{-1}(n)}Z_{G, X}\left(\frac{A}{n}\right)=n^{\frac{D_X}{2}}L_X^{-\frac{1}{2}}(n)Z_{G, X}\left(\frac{A}{n}\right)
\intertext{and}
&Z_{G, Y}^{(n)}(A)=\sqrt{n^{D_Y}L_Y^{-1}(n)}Z_{G, Y}\left(\frac{A}{n}\right)=n^{\frac{D_Y}{2}}L_Y^{-\frac{1}{2}}(n)Z_{G, Y}\left(\frac{A}{n}\right).
\end{align*}

Moreover, it holds that
\begin{align*}
\sum\limits_{j=1}^nX_jY_j
&=\sum\limits_{j=1}^n\int_{\left[-\pi, \pi\right)}\e^{ixj}Z_{G, X}(dx)\int_{\left[-\pi, \pi\right)}\e^{iyj}Z_{G, Y}(dy)\\
&=\int_{\left[-\pi, \pi\right)^2}\sum\limits_{j=1}^n\e^{i(x+y)j}Z_{G, X}(dx)Z_{G, Y}(dy)\\
&=\int_{\left[-\pi, \pi\right)^2}\e^{i(x+y)}\sum\limits_{j=0}^{n-1}\e^{i(x+y)j}Z_{G, X}(dx)Z_{G, Y}(dy)\\
&=\int_{\left[-\pi, \pi\right)^2}\e^{i(x+y)}\frac{\e^{i(x+y)n}-1}{\e^{i(x+y)}-1}Z_{G, X}(dx)Z_{G, Y}(dy).
\end{align*}

Again, the change of variables formula yields
\begin{align*}
&n^{\frac{D_X+D_Y}{2}-1}L_X^{-\frac{1}{2}}L_Y^{-\frac{1}{2}}\sum\limits_{j=1}^nX_jY_j\\
=&n^{\frac{D_X+D_Y}{2}-1}L_X^{-\frac{1}{2}}L_Y^{-\frac{1}{2}}\int\limits_{\left[-\pi, \pi\right)^2} \e^{i(x+y)}\left(\frac{\e^{i(x+y)n}-1}{\e^{i(x+y)}-1}\right)
Z_{G, X}(dx)Z_{G, Y}(dy)\\
=&\int\limits_{\left[-n\pi, n\pi\right)^2} \e^{i\frac{x+y}{n}}\left(\frac{\e^{i(x+y)}-1}{n\left(\e^{i\frac{x+y}{n}}-1\right)}\right)
Z_{G, X}^{(n)}(dx)Z_{G, Y}^{(n)}(dy),
\end{align*}
where 
\begin{align*}
&Z_{G, X}^{(n)}(A)=\sqrt{n^{D_X}L_X^{-1}(n)}Z_{G, X}\left(\frac{A}{n}\right)=n^{\frac{D_X}{2}}L_X^{-\frac{1}{2}}(n)Z_{G, X}\left(\frac{A}{n}\right)
\intertext{and}
&Z_{G, Y}^{(n)}(A)=\sqrt{n^{D_Y}L_Y^{-1}(n)}Z_{G, Y}\left(\frac{A}{n}\right)=n^{\frac{D_Y}{2}}L_Y^{-\frac{1}{2}}(n)Z_{G, Y}\left(\frac{A}{n}\right).
\end{align*}
With the results of \cite{major:2020} it then follows that
\begin{align*}
n^{\frac{D_X+D_Y}{2}-2}L_X^{-\frac{1}{2}}L_Y^{-\frac{1}{2}}\sum\limits_{i=1}^n\sum\limits_{j=1}^nX_iY_j-n^{\frac{D_X+D_Y}{2}-1}L_X^{-\frac{1}{2}}L_Y^{-\frac{1}{2}}\sum\limits_{j=1}^nX_jY_j
\end{align*}
converges in distribution to 
\begin{align*}
\int_{\left[-\pi, \pi\right)^2}\left[\left(\frac{\e^{ix}-1}{ix}\right)\left(\frac{\e^{iy}-1}{iy}\right)-\frac{\e^{i(x+y)}-1}{i(x+y)}\right]
Z_{G, X, 0}(dx)Z_{G, Y, 0}(dy).
\end{align*}
\end{proof}

In order to prove Proposition \ref{prop:convergence_SRD}, 
we apply the following theorem that corresponds to a multivariate generalization of  Theorem 2 in \cite{cremers:kadelka:1986} for stochastic processes with paths in $L^p(S, \mu)$, where $\pr{S, \mathcal{S}, \mu}$ is a $\sigma$-finite measure space,  when choosing $p=2$. 

\begin{theorem}\label{thm:cremers_kadelka}
Let $\pr{S, \mathcal{S}, \mu}$ be a $\sigma$-finite measure space and  let $(\xi_n^1, \ldots, \xi_n^k)$, $n\geq 1$,  be a sequence of stochastic processes with paths in the product space $L^2(S, \mu)\otimes \cdots \otimes L^2(S, \mu)$. Then $$(\xi_n^1, \ldots, \xi_n^k)\overset{\mathcal{D}}{\longrightarrow} (\xi_0^1, \ldots, \xi_0^k), $$
 where $\overset{\mathcal{D}}{\longrightarrow}$ denotes convergence in $L^2(S, \mu)\otimes \cdots \otimes L^2(S, \mu)$, provided the finite dimensional distributions of $(\xi_n^1, \ldots, \xi_n^k)$ converge weakly to those of $(\xi_0^1, \ldots, \xi_0^k)$ almost everywhere and provided the following conditions hold: 
for some positive, $\mu$-integrable functions $f_i$, $i=1, \ldots, k$, it holds that
\begin{align*}
\E \left|\xi^i_n(s)\right|^{2}\leq f_i(s) \ \text{for all $s\in S$, $n\in \mathbb{N}$,}  \intertext{and}  \E \left|\xi^i_n(s)\right|^{2}\longrightarrow\E \left|\xi^i_0(s)\right|^{2} \ \text{for all $s\in S$.}
\end{align*}
\end{theorem}

Since  $L^2(S, \mu)\otimes \cdots \otimes L^2(S, \mu)\cong L^2(S\times \cdots \times S, \mu\times \cdots \times\mu)$, 
Theorem \ref{thm:cremers_kadelka}
is an immediate consequence of Theorem 2 in \cite{cremers:kadelka:1986}.

\begin{proof}[Proof of Proposition \ref{prop:convergence_SRD}]

In order to show convergence of the finite dimensional distributions  we have to prove that for fixed $k$ and $s_1, t_1, \ldots, s_k, t_k \in \mathbb{R}$
\begin{align*}
&Z_n\defeq\left(Z_{n, 1}, \ldots, Z_{n, k}\right)^\top, \ \text{ where } \
Z_{n, i}\defeq \left( Z_n(s_i, t_i; 0),  \ldots, Z_n(s_i, t_i; H) \right),
\end{align*}
converges in distribution  to the corresponding finite dimensional distribution of a complex Gaussian random variable $Z\defeq\left(Z_{1}, Z_{2}, \ldots, Z_{k}\right)^\top$, where $Z_j=(Z_{j, 0}, \ldots, Z_{j,H})$. 

Due to the  Cram\'{e}r-Wold theorem, for this we have to show that  for all $\lambda_{i, j}$, $\eta_{i, j} \in \mathbb{R}$, $i=0, \ldots, H$, $j=1, \ldots, k$
\begin{align*}
&\sum\limits_{i=0}^H\sum\limits_{j=1}^k\lambda_{i,j} \Ret(Z_{n}(s_j, t_j; i))+\sum\limits_{i=0}^H\sum\limits_{j=1}^k\eta_{i, j} \Imt(Z_{n}(s_j, t_j; i))\\
&\overset{\mathcal{D}}{\longrightarrow}\sum\limits_{i=0}^H\sum\limits_{j=1}^k\lambda_{i, j} \operatorname{Re} (Z_{j, i})+\sum\limits_{i=0}^H\sum\limits_{j=1}^k\eta_{i, j} \Imt(Z_{j, i}).
\end{align*}

To this end, note that 
\begin{align*}
&\sum\limits_{i=0}^H\sum\limits_{j=1}^k\lambda_{i,j} \Ret(Z_{n}(s_j, t_j; i))+\sum\limits_{i=0}^H\sum\limits_{j=1}^k\eta_{i, j} \Imt(Z_{n}(s_j, t_j; i))\\
&=\frac{1}{\sqrt{n}}\sum\limits_{l=1}^n K(X_l, Y_l, Y_{l+1}, \ldots, Y_{l+H}), 
\end{align*}
where
\begin{align*}
&K(X_l, Y_l, Y_{l+1}, \ldots, Y_{l+H})\\
&\defeq 
\sum\limits_{i=0}^H\sum\limits_{j=1}^k\lambda_{i,j} \Ret(f_{s_j, t_j}(X_l, Y_{l+i}))+\sum\limits_{i=0}^H\sum\limits_{j=1}^k\eta_{i, j}  \Imt(f_{s_j, t_j}(X_l, Y_{l+i})).
\end{align*}

In order to derive the asymptotic distribution of the above partial sum, we apply the following theorem that directly follows from Theorem 4 in \cite{arcones:1994}:

\begin{theorem}\label{thm:arcones}
Let $\mathbf{X}_k=(X_k^{(1)}, X_k^{(2)}, \ldots, X_k^{(d)})$, $k\geq 1$, be a stationary mean zero Gaussian sequence of $\mathbb{R}^d$-valued random vectors. Let $f$ be a function on $\mathbb{R}^d$ with Hermite rank $r$.
We define
\begin{align*}
\rho^{(p, q)}(k)\defeq\E\left(X_1^{(p)}X_{1+k}^{(q)}\right)
\end{align*}
for $k\geq 1$. Suppose that 
\begin{align*}
\sum\limits_{k=-\infty}^{\infty}\left|\rho^{(p, q)}(k)\right|^{r}<\infty
\end{align*}
for each $1\leq p, q\leq d$. Then, it holds that
\begin{align*}
\frac{1}{\sqrt{n}}\sum\limits_{j=1}^n\left(f(\mathbf{X}_j)-\E f(\mathbf{X}_j)\right)\overset{\mathcal{D}}{\longrightarrow}\mathcal{N}(0, \sigma^2),
\end{align*}
where
\begin{align*}
\sigma^2\defeq 
&\E\left[\left(f(\mathbf{X}_1)-\E f(\mathbf{X}_1)\right)^2\right]\\
&+2\sum\limits_{k=1}^{\infty}\E\left[\left(f(\mathbf{X}_1)-\E f(\mathbf{X}_1)\right)\left(f(\mathbf{X}_{1+k})-\E f(\mathbf{X}_{1+k})\right)\right].
\end{align*}
\end{theorem}

For an application of Theorem \ref{thm:arcones}, we have to compute the Hermite rank of $K$. For our purposes, however, it suffices to show that the Hermite rank is bigger than $1$.
For this,
we have to show that
\begin{align*}
\E\left(K(X_l, Y_l, Y_{l+1}, \ldots, Y_{l+H})H_{\mathbf{q}}^{*}(X_l, Y_l, Y_{l+1}, \ldots, Y_{l+H})\right)=0,
\end{align*}
whenever $\left|\mathbf{q}\right|=1$.

We distinguish two cases for $\mathbf{q}=(q_1, \ldots, q_{H+2})^{\top}$ with $\left|\mathbf{q}\right|=1$:
$q_1=1$ and $q_1=0$.

If $q_1=1$, $H_{\mathbf{q}}^{*}(X_l, Y_l, Y_{l+1}, \ldots, Y_{l+H})=X_l$
and $H_{\mathbf{q}}^{*}(X_l, Y_l, Y_{l+1}, \ldots, Y_{l+H})=Y_j$ for some $j\in \{l, \ldots, l+H\}$ if $q_1=0$.

If $q_1=1$, it follows that
\begin{align*}
&\E\left(K(X_l, Y_l, Y_{l+1}, \ldots, Y_{l+H})H_{\mathbf{q}}^{*}(X_l, Y_l, Y_{l+1}, \ldots, Y_{l+H})\right)\\
=&\E\left(K(X_l, Y_l, Y_{l+1}, \ldots, Y_{l+H})X_l\right)\\
=&
\sum\limits_{i=0}^H\sum\limits_{j=1}^k\lambda_{i,j}\E\left( \Ret(f_{s_j, t_j}(X_l, Y_{l+i}))X_l\right)+\sum\limits_{i=0}^H\sum\limits_{j=1}^k\eta_{i, j} \E\left( \Imt(f_{s_j, t_j}(X_l, Y_{l+i}))X_l\right).
\end{align*}

Moreover,   it holds that
\begin{align*}
\Ret\left(f_{s,t}(X_l, Y_{l+i})\right)
=& \cos(sX_l)\cos(tY_{l+i}) 
- \cos(sX_l)\E[\cos(tY_{l+i})]\\
 &     - \E[\cos(sX_l)]\cos(tY_{l+i}) 
 + \E[\cos(sX_l)]\E[\cos(tY_{l+i})] \\
& - \sin(sX_l)\sin(tY_{l+i}) 
 + \sin(sX_l)\E[\sin(tY_{l+i})]\\
&      + \E[\sin(sX_l)]\sin(tY_{l+i}) 
 - \E[\sin(sX_l)]\E[\sin(tY_{l+i})].
\end{align*}
and
\begin{align*}
\Imt\left(f_{s,t}(X_l, Y_{l+i})\right)
=& \cos(sX_l)\sin(tY_{l+i}) 
 - \cos(sX_l)\E[\sin(tY_{l+i})]\\
&      - \E[\cos(sX_l)]\sin(tY_{l+i}) 
 + \E[\cos(sX_l)]\E[\sin(tY_{l+i})] \\
& + \sin(sX_l)\cos(tY_{l+i}) 
 - \sin(sX_l)\E[\cos(tY_{l+i})]\\
&      - \E[\sin(sX_l)]\cos(tY_{l+i}) 
 + \E[\sin(sX_l)]\E[\cos(tY_{l+i})].
\end{align*} 
Since $X_l$ and $Y_{l+i}$ are independent, it follows that
  \begin{align*}
 \E\left( \Ret f_{s_j, t_j}\pr{X_l, Y_{l+i}}X_l\right)
=0, \
 \E\left( \Imt f_{s_j, t_j}\pr{X_l, Y_{l+i}}X_l\right)=0.
 \end{align*}
 If $q_1=0$, then there exists an $m\in\{0, \ldots, H\}$
 such that  $H_{\mathbf{q}}^{*}(X_l, Y_l, Y_{l+1}, \ldots, Y_{l+H})=Y_{l+m}$. Analogously to the previous considerations  it follows that
\begin{align*}
&\E\left(K(X_l, Y_l, Y_{l+1}, \ldots, Y_{l+H})H_{\mathbf{q}}^{*}(X_l, Y_l, Y_{l+1}, \ldots, Y_{l+H})\right)\\
=&\E\left(K(X_l, Y_l, Y_{l+1}, \ldots, Y_{l+H})Y_{l+m}\right)\\
=&
\sum\limits_{i=0}^H\sum\limits_{j=1}^k\lambda_{i,j}\E\left( \Ret(f_{s_j, t_j}(X_l, Y_{l+i}))Y_{l+m}\right)+\sum\limits_{i=0}^H\sum\limits_{j=1}^k\eta_{i, j} \E\left( \Imt(f_{s_j, t_j}(X_l, Y_{l+i}))Y_{l+m}\right)
\end{align*} 
and 
 \begin{align*}
 \E\left( \Ret f_{s_j, t_j}\pr{X_l, Y_{l+i}}Y_{l+m}\right)
=0, \
 \E\left( \Imt f_{s_j, t_j}\pr{X_l, Y_{l+i}}Y_{l+m}\right)=0.
 \end{align*}

 Therefore,  the Hermite rank $r$ of $K$ is bigger than $1$, 
 such that for $D_{\xi}, D_{\eta}\in~\left(\frac{1}{2}, 1\right)$
\begin{align*}
\sum\limits_{k=-\infty}^{\infty}\left|\rho_{\xi}(k)\right|^{r}\leq\sum\limits_{k=-\infty}^{\infty}\left|\rho(k)\right|^{2}<\infty,
\
\sum\limits_{k=-\infty}^{\infty}\left|\rho_{\eta}(k)\right|^{r}\leq\sum\limits_{k=-\infty}^{\infty}\left|\rho(k)\right|^{2}<\infty.
\end{align*}
As a result, Theorem \ref{thm:arcones} implies that
\begin{align*}
\frac{1}{\sqrt{n}}\sum\limits_{l=1}^n K(X_l, Y_l, Y_{l+1}, \ldots, Y_{l+H}) \overset{\mathcal{D}}{\longrightarrow}\mathcal{N}(0, \sigma^2),
\end{align*}
where
\begin{align*}
\sigma^2\defeq &
\E\left[K(X_1, Y_1, Y_{2}, \ldots, Y_{1+H})^2\right]\\
&+2\sum\limits_{k=1}^{\infty}\E\left[K(X_1, Y_1, Y_{2}, \ldots, Y_{1+H})K(X_{k+1}, Y_{k+1}, Y_{k+2}, \ldots, Y_{k+1+H})\right].
\end{align*}

According to Theorem \ref{thm:cremers_kadelka}, for a proof of Proposition \ref{prop:convergence_SRD}
it thus remains to show that 
for some positive, $w(s, t)dsdt$-integrable functions $f_h$
\begin{align*}
&\E \left|Z_n(s, t; h)\right|^{2}\leq f_h(s, t) \ \text{for all $\pr{s, t}\in \mathbb{R}^2$, $n\in \mathbb{N}$,} \ \intertext{and} &\lim\limits_{n\rightarrow\infty}\E \left|Z_n(s, t; h)\right|^{2}=\E \left|Z_h(s, t)\right|^{2} \ \text{for all $\pr{s, t}\in \mathbb{R}^2$.}
\end{align*}

For this, note that due to independence of $X_j$, $j\in \mathbb{N}$, and $Y_j$, $j\in \mathbb{N}$, it holds that
\begin{flalign*}
\E \left(\left|Z_n(s, t; h)\right|^2\right)
=\frac{1}{n}\sum\limits_{j=1}^{n-h}\sum\limits_{k=1}^{n-h}&\E \left(\exp(isX_j)-\varphi_X(s)\right)\left(\exp(-isX_k)-\varphi_X(-s)\right)\\
&\E \left(\exp(itY_{j+h})-\varphi_Y(t)\right)\left(\exp(-itY_{k+h})-\varphi_Y(-t)\right).
\end{flalign*}

An expansion in Hermite polynomials yields
\begin{align*}
\exp(isX_j)-\varphi_X(s)=&\sum\limits_{l=1}^{\infty}\frac{1}{l!}\E\pr{\cos(sG(\xi_j))H_l(\xi_j)}H_l(\xi_j)\\
&+i\sum\limits_{l=1}^{\infty}\frac{1}{l!}\E\pr{\sin(sG(\xi_j))H_l(\xi_j)}H_l(\xi_j)
\end{align*}
and
\begin{align*}
\exp(-isX_k)-\varphi_X(-s)
=&\sum\limits_{l=1}^{\infty}\frac{1}{l!}\E\pr{\cos(sG(\xi_k))H_l(\xi_k)}H_l(\xi_k)\\
&-i\sum\limits_{l=1}^{\infty}\frac{1}{l!}\E\pr{\sin(sG(\xi_k))H_l(\xi_k)}H_l(\xi_k).
\end{align*}
Since
\begin{align*}
\Cov(H_l(\xi_j)H_m(\xi_k))
=\begin{cases}
\rho^l_\xi(j-k)l! \ &\text{ if } l= m\\
0  \ &\text{ if } l\neq m
\end{cases},
\end{align*}
we thus have
\begin{align*}
&\E \left(\exp(isX_j)-\varphi_X(s)\right)\left(\exp(-isX_k)-\varphi_X(-s)\right)\\
=&\sum\limits_{l=1}^{\infty}\frac{\left(\E\pr{\cos(sG(\xi_1))H_l(\xi_1)}\right)^2+\left(\E\pr{\sin(sG(\xi_1))H_l(\xi_1)}\right)^2}{l!}\rho^l_\xi(j-k).
\end{align*}

With $J_{l, i}(s)\defeq \left(\E\pr{\cos(sG_i(\xi_j))H_l(\xi_j)}\right)^2+\left(\E\pr{\sin(sG_i(\xi_j))H_l(\xi_j)}\right)^2$, it holds that
\begin{align*}
\E \left(\left|Z_n(s, t; h)\right|^2\right)
=\frac{1}{n}\sum\limits_{j=1}^{n-h}\sum\limits_{k=1}^{n-h}
\sum\limits_{l=1}^{\infty}\sum\limits_{m=1}^{\infty}\frac{J_{l, 1}(s)J_{m, 2}(t)}{l!m!}\rho^l_{\xi}(j-k)\rho^m_{\eta}(j-k).
\end{align*}
For  $l, m\geq 1$ and $D_{\xi}, D_{\eta}\in \left(0, 1\right)$ with $D_{\xi}+D_{\eta}>1$, we have
\begin{align*}
\sum\limits_{j=1}^{n-h}\sum\limits_{k=1}^{n-h}\rho_{\xi}^{l}(j-k)\rho_{\eta}^{m}(j-k)\leq 
\sum\limits_{j=1}^{n-h}\sum\limits_{k=1}^{n-h}\rho_{\xi}(j-k)\rho_{\eta}(j-k)=\mathcal{O}\pr{n}.
\end{align*}
It follows that
\begin{align*}
\E \left(\left|Z_n(s, t; h)\right|^2\right)
= \mathcal{O}\pr{
\sum\limits_{l=1}^{\infty}\sum\limits_{m=1}^{\infty}\frac{J_{l, 1}(s)J_{m, 2}(t)}{l!m!}}.
\end{align*}

Since
\begin{align*}
\sum\limits_{l=1}^{\infty}\frac{J_{l, i}(s)}{l!}
=&\E \left(\exp(isX_1)-\varphi_X(s)\right)\left(\exp(-isX_1)-\varphi_X(-s)\right)\\
=&1-\varphi_X(s)\varphi_X(-s)\\
= &1
-\left[\E\pr{\cos(sG_i(\xi_1))}\right]^2-\left[\E\pr{\sin(sG_i(\xi_1))}\right]^2,
\end{align*}
 it follows by Lemma 1 in \cite{szekely:2007} that
\begin{align*}
\int_{\mathbb{R}}\sum\limits_{l=1}^{\infty}\frac{J_{l, i}(s)}{l!}(cs^2)^{-1}ds
=&\int_{\mathbb{R}}\left(1
-\left(\E\pr{\cos(sG_i(\xi_j))}\right)^2-\left(\E\pr{\sin(sG_i(\xi_j))}\right)^2\right)(cs^2)^{-1}ds\\
=&\int_{\mathbb{R}} \E\pr{1-\cos(s(X-X'))}(cs^2)^{-1}ds\\
=&\E\left(\int_{\mathbb{R}} \left(1-\cos(s(X-X'))\right)(cs^2)^{-1}ds\right)\\
\leq & C\E\left|X-X'\right|<\infty.
\end{align*}
As a result, we have
\begin{align*}
&\E \left|Z_n(s, t; h)\right|^{2}\leq f(s, t) \ \text{for all $\pr{s, t}\in \mathbb{R}^2$, $\; n\in \mathbb{N}$,}
\end{align*}
where, for some positive constant $C$,  $f(s, t)\defeq C\sum_{l=1}^{\infty}\frac{J_{l, 1}(s)}{l!}\sum_{m=1}^{\infty}\frac{J_{m, 2}(t)}{m!}$  is a positive, $w(s, t)dsdt$-integrable function .

Moreover, convergence of $\E \left|Z_n(s, t; h)\right|^{2}$ and $\E \left(Z_n(s, t; h)\right)^{2}$ follows by the dominated convergence theorem.

As limits we obtain
\begin{align*}
\E \left(\left|Z_n(s, t; h)\right|^2\right)
=&
\sum\limits_{k=-(n-h-1)}^{n-h-1}\left(1-\frac{|k|}{n}\right)\E\pr{f_{s, t}(X_1, Y_1)\overline{f_{s, t}(X_{k+1}, Y_{k+1})}}\\
&\longrightarrow\sum\limits_{k=-\infty}^{\infty}
\E \pr{f_{s, t}(X_1, Y_1)\overline{f_{s, t}(X_{k+1}, Y_{k+1})}},
\end{align*}
and
\begin{align*}
\E \left(\left(Z_n(s, t: h)\right)^2\right)
=&
\sum\limits_{k=-(n-h-1)}^{n-h-1}\left(1-\frac{|k|}{n}\right)\E\pr{f_{s, t}(X_1, Y_1)f_{s, t}(X_{k+1}, Y_{k+1})}\\
&\longrightarrow\sum\limits_{k=-\infty}^{\infty}
\E \pr{f_{s, t}(X_1, Y_1)f_{s, t}(X_{k+1}, Y_{k+1})}.
\end{align*}

Analogous computations show that
for $i, j\in \{0, \ldots, H\}$,  $i\leq j$, 
\begin{flalign*}
\E \left(Z_n(s, t; i)\overline{Z_n(s, t; j)}\right)
\ignore{=&\frac{1}{n}\sum\limits_{l=1}^{n-j}\sum\limits_{k=1}^{n-i}\E\pr{f_{s, t}(X_l, Y_{l+i})\overline{f_{s, t}(X_{k}, Y_{k+j})}}\\
=&\frac{1}{n}\sum\limits_{l=1}^{n-j}\sum\limits_{k=1}^{n-j}\E\pr{f_{s, t}(X_l, Y_{l+i})\overline{f_{s, t}(X_{k}, Y_{k+j})}}\\
&+\frac{1}{n}\sum\limits_{l=1}^{n-j}\sum\limits_{k=n-j+1}^{n-i}\E\pr{f_{s, t}(X_l, Y_{l+i})\overline{f_{s, t}(X_{k}, Y_{k+j})}}\\
&=\frac{n-j}{n}\sum\limits_{k=-(n-j-1)}^{n-j-1}\left(1-\frac{|k|}{n-j}\right)\E\pr{f_{s, t}(X_{1}, Y_{1+i})\overline{f_{s, t}(X_{k+1}, Y_{k+1+j})}}+o(1)\\
&}\longrightarrow\sum\limits_{k=-\infty}^{\infty}
\E \pr{f_{s, t}(X_1, Y_{1+i})\overline{f_{s, t}(X_{k+1}, Y_{k+1+j})}},
\end{flalign*}
while
\begin{flalign*}
\E \left(Z_n(s, t; i)Z_n(s, t; j)\right)
\ignore{=&\frac{1}{n}\sum\limits_{l=1}^{n-j}\sum\limits_{k=1}^{n-i}\E\pr{f_{s, t}(X_l, Y_{l+i})f_{s, t}(X_{k}, Y_{k+j})}\\
=&\frac{1}{n}\sum\limits_{l=1}^{n-j}\sum\limits_{k=1}^{n-j}\E\pr{f_{s, t}(X_l, Y_{l+i})f_{s, t}(X_{k}, Y_{k+j})}\\
&+\frac{1}{n}\sum\limits_{l=1}^{n-j}\sum\limits_{k=n-j+1}^{n-i}\E\pr{f_{s, t}(X_l, Y_{l+i})f_{s, t}(X_{k}, Y_{k+j})}\\
&=\frac{n-j}{n}\sum\limits_{k=-(n-j-1)}^{n-j-1}\left(1-\frac{|k|}{n-j}\right)\E\pr{f_{s, t}(X_{1}, Y_{1+i})f_{s, t}(X_{k+1}, Y_{k+1+j})}+o(1)\\
&}\longrightarrow\sum\limits_{k=-\infty}^{\infty}
\E \pr{f_{s, t}(X_1, Y_{1+i})f_{s, t}(X_{k+1}, Y_{k+1+j})}.
\end{flalign*}

\end{proof}

\begin{proof}[Proof of Theorem \ref{thm:sample_covariance}]
Without loss of generality we assume that $h=0$.
Note that 
\begin{align*}
\sum_{i=1}^n(X_i-\bar{X})(Y_i-\bar{Y})
=\sum\limits_{i=1}^n X_iY_i -\frac{1}{n}\sum\limits_{i=1}^n X_i\sum\limits_{i=1}^n Y_i.
\end{align*}

Thus, for $D_{\xi} + D_{\eta}<1$ and $G_1=G_2=\text{id}$, we have 
\begin{align*}
&n^{\frac{D_X+D_Y}{2}}L_X^{-\frac{1}{2}}(n)L_Y^{-\frac{1}{2}}(n)\frac{1}{n}\sum_{i=1}^n(X_i-\bar{X})(Y_i-\bar{Y})\\
=&n^{\frac{D_X+D_Y}{2}-1}L_X^{-\frac{1}{2}}(n)L_Y^{-\frac{1}{2}}(n)\sum\limits_{i=1}^n X_iY_i-n^{\frac{D_X+D_Y}{2}-2}L_X^{-\frac{1}{2}}(n)L_Y^{-\frac{1}{2}}(n)\sum\limits_{i=1}^n X_i\sum\limits_{j=1}^n Y_j.
\end{align*}

According to the proof of Proposition \ref{prop:conv_LRD} the above expression converges to 
\begin{align*}
\int_{\left[-\pi, \pi\right)^2}\left[\left(\frac{\e^{ix}-1}{ix}\right)\left(\frac{\e^{iy}-1}{iy}\right)-\frac{\e^{i(x+y)}-1}{i(x+y)}\right]
Z_{G, X, 0}(dx)Z_{G, Y, 0}(dy).
\end{align*}
\end{proof}

\begin{landscape}

\section{Additional simulation results}\label{app:sim}		
\begin{table}[htbp]
\caption{Rejection rates of the hypothesis tests resulting from the empirical distance covariance and the empirical covariance 
				obtained by  subsampling   based on \enquote{linearly} correlated 
				time series $X_j, \ j=1, \ldots, n$, $Y_j, \ j=1, \ldots, n$ according to
Section \ref{subsec:simulations} 
with block length $l_n$, $d=0.1n$, and Hurst parameters $H$.
				The level of significance equals 5\%.}
\begin{tabular}{cccccccccccccccccc}
 & & &   \multicolumn{7}{c}{distance covariance} &  & \multicolumn{7}{c}{covariance}\\
 				\cline{5-9}   \cline{13-17}\\
 \\
				& &  & & $H=0.6$ & &  & & $H=0.7$ & &  & & $H=0.6$ & &  & & $H=0.7$ & \\ 
				\cline{4-6}   \cline{8-10} \cline{12-14} \cline{16-18} \\
				& $n$ &  & $r=0$ & $r=0.25$ & $r=0.5$ &  & $r=0$ & $r=0.25$ & $r=0.5$ &  & $r=0$ & $r=0.25$ & $r=0.5$ &  &$r=0$ & $r=0.25$ & $r=0.5$ \\ 
					\hline\\
				\multirow{4}{*}{\rotatebox{90}{$l_n=n^{0.4}$}}  &
100 &  & 0.131 & 0.713 & 0.999 &  & 0.175 & 0.745 & 0.997 &  &   0.125 & 0.770 & 1,000 &  & 0.172 & 0.772 & 0.999 \\ 
&300 &  & 0.095 & 0.984 & 1.000 &  & 0.161 & 0.977 & 1.000 &  &   0.098 & 0.991 & 1,000 &  & 0.164 & 0.985 & 1.000 \\ 
& 500 &  & 0.089 & 1.000 & 1.000 &  & 0.148 & 0.998 & 1.000 &  &   0.083 & 1.000 & 1,000 &  & 0.160 & 0.999 & 1.000 \\ 
& 1000 &  & 0.081 & 1.000 & 1.000 &  & 0.131 & 1.000 & 1.000 &  &   0.085 & 1.000 & 1,000 &  & 0.139 & 1.000 & 1.000 \\
 &  &  &  &  &  &  &  &  &  &  &  &  &  &  &   &  \\ 
	\multirow{4}{*}{\rotatebox{90}{$l_n=n^{0.5}$}} &100 &  & 0.116 & 0.689 & 0.998 &  & 0.155 & 0.706 & 0.996 &  &   0.110 & 0.748 & 0.999 &  & 0.165 & 0.740 & 0.999 \\
&300 &  & 0.089 & 0.979 & 1.000 &  & 0.125 & 0.969 & 1.000 &  &   0.083 & 0.986 & 1,000 &  & 0.136 & 0.976 & 1.000 \\ 
& 500 &  & 0.084 & 0.999 & 1.000 &  & 0.126 & 0.998 & 1.000 &  &   0.078 &1.000 & 1,000 &  & 0.125 & 0.999 & 1.000 \\ 
& 1000 &  & 0.073 & 1.000 & 1.000 &  & 0.112 & 1.000 & 1.000 &  &   0.069 & 1.000 & 1,000 &  & 0.124 & 1.000 & 1.000\\ 
 &  &  &  &  &  &  &  &  &  &  &  &  & &  &   &  \\ 
	\multirow{4}{*}{\rotatebox{90}{$l_n=n^{0.6}$}} & 100 &  & 0.123 & 0.683 & 0.997 &  & 0.150 & 0.679 & 0.994 &  &   0.120 & 0.730 & 0.999 &  & 0.148 & 0.720 & 0.996 \\ 
& 300 &  & 0.091 & 0.973 & 1.000 &  & 0.118 & 0.958 & 1.000 &  &   0.089 & 0.983 & 1,000 &  & 0.115 & 0.971 & 1.000 \\ 
& 500 &  & 0.086 & 0.999 & 1.000 &  & 0.116 & 0.995 & 1.000 &  &   0.084 & 1.000 & 1,000 &  & 0.118 & 0.996 & 1.000 \\ 
& 1000 &  & 0.074 & 1.000 & 1.000 &  & 0.097 & 1.000 & 1.000 &  &   0.072 & 1.000 & 1,000 &  & 0.101 & 1.000& 1.000 \\ 
\end{tabular}
\label{table:linear}
\end{table}
\end{landscape}

\begin{landscape}
\begin{table}[htbp]
			\caption{Rejection rates of the hypothesis tests resulting from the empirical distance covariance  and the empirical covariance
		obtained by  subsampling   based on  \enquote{parabolically} correlated  time series $X_j, \ j=1, \ldots, n$, $Y_j, \ j=1, \ldots, n$ according to
Section \ref{subsec:simulations}
 with block length $l_n$,  $d=0.1n$, and Hurst parameter $H$. The level of significance equals 5\%.}
\begin{tabular}{ccccccccccccccccccc}
 & & &   \multicolumn{7}{c}{distance covariance} &  & \multicolumn{7}{c}{covariance}\\
 				\cline{5-9}   \cline{13-17}\\
 \\
				& &  & & $H=0.6$ & &  & & $H=0.7$ & &  & & $H=0.6$ & &  & & $H=0.7$ & \\ 
				\cline{4-6}   \cline{8-10} \cline{12-14} \cline{16-18} \\
& $n$ &  & $v=0.5$ & $v=0.75$ & $v=1$ &  & $v=0.5$ & $v=0.75$ & $v=1$ &  & $v=0.5$ & $v=0.75$ & $v=1$ &  & $v=0.5$ & $v=0.75$ & $v=1$ \\ 
\hline\\
& &  &  &  &  &  &  &  &  &  &  &  &  &  &  &  &  \\ 
	\multirow{4}{*}{\rotatebox{90}{$l_n=n^{0.4}$}}  & 100 &  & 0.308 & 0.637 & 0.923 &  & 0.364 & 0.697 & 0.940 &  & 0.130 & 0.146 & 0.196 &  & 0.186 & 0.256 & 0.324 \\ 
& 300 &  & 0.719 & 0.993 & 1.000 &  & 0.785 & 0.996 &1.000 &  & 0.107 & 0.136 & 0.176 &  & 0.190 & 0.284 & 0.387 \\ 
& 500 &  & 0.943 &1.000 &1.000 &  & 0.954 &1.000 &1.000 &  & 0.103 & 0.129 & 0.187 &  & 0.190 & 0.301 & 0.410 \\ 
& 1000 &  &1.000 &1.000 &1.000 &  &1.000 &1.000 &1.000 &  & 0.100 & 0.138 & 0.192 &  & 0.226 & 0.346 & 0.443 \\ 
 &  &  &  &  &  &  &  &  &  &  &  &  &  &  &  &  \\ 
	\multirow{4}{*}{\rotatebox{90}{$l_n=n^{0.5}$}}  & 100 &  & 0.296 & 0.605 & 0.900 &  & 0.340 & 0.662 & 0.919 &  & 0.124 & 0.139 & 0.194 &  & 0.174 & 0.245 & 0.307 \\ 
& 300 &  & 0.678 & 0.986 &1.000 &  & 0.735 & 0.991 &1.000 &  & 0.104 & 0.132 & 0.168 &  & 0.185 & 0.268 & 0.368 \\ 
& 500 &  & 0.923 & 1.000 &1.000 &  & 0.936 &1.000 &1.000 &  & 0.102 & 0.126 & 0.181 &  & 0.177 & 0.286 & 0.392 \\ 
& 1000 &  &1.000 &1.000 &1.000 &  & 1.000 &1.000 &1.000 &  & 0.099 & 0.135 & 0.187 &  & 0.216 & 0.332 & 0.429 \\ 
 &  &  &  &  &  &  &  &  &  &  &  &  &  &  &  &  \\ 
	\multirow{4}{*}{\rotatebox{90}{$l_n=n^{0.6}$}}  & 100 &  & 0.309 & 0.603 & 0.887 &  & 0.337 & 0.662 & 0.901 &  & 0.132 & 0.151 & 0.201 &  & 0.177 & 0.250 & 0.306 \\ 
& 300 &  & 0.667 & 0.978 & 0.999 &  & 0.708 & 0.981 & 1.000 &  & 0.111 & 0.138 & 0.173 &  & 0.184 & 0.264 & 0.365 \\ 
& 500 &  & 0.902 &1.000 &1.000 &  & 0.915 &1.000 &1.000 &  & 0.105 & 0.133 & 0.192 &  & 0.185 & 0.286 & 0.383 \\ 
& 1000 &  & 0.999 &1.000 &1.000 &  & 0.999 &1.000 &1.000 &  & 0.102 & 0.144 & 0.190 &  & 0.215 & 0.324 & 0.423 \\  
\end{tabular}
\label{table:squares}
\end{table}
\end{landscape}

\begin{landscape}
\begin{table}[htbp]
	\caption{Rejection rates of the hypothesis tests resulting from the empirical distance covariance  and the empirical covariance
				obtained by  subsampling   based on  \enquote{wavily} correlated  time series $X_j$, $j=1, \ldots, n$, $Y_j, \ j=1, \ldots, n$  according to
Section \ref{subsec:simulations} 
with block length $l_n$,  $d=0.1n$, and  Hurst parameter $H$.
				The level of significance equals 5\%.}
\begin{center}
\begin{tabular}{cccccccccccccccccc}
 & & &   \multicolumn{7}{c}{distance covariance} &  & \multicolumn{7}{c}{covariance}\\
 				\cline{5-9}   \cline{13-17}\\
				& &  & & $H=0.6$ & &  & & $H=0.7$ & &  & & $H=0.6$ & &  & & $H=0.7$ & \\ 
				\cline{4-6}   \cline{8-10} \cline{12-14} \cline{16-18} \\
	&			 $n$ &  & $v=1$ &$v=2$ & $v=3$ &  & $v=1$ &$v=2$ & $v=3$ &  &$v=1$ &$v=2$ & $v=3$ &  & $v=1$ &$v=2$ & $v=3$ \\
				\hline \\
 &  &  &  &  &  &  &  &  &  &  &  &  &  &  &  &  \\ 
\multirow{4}{*}{\rotatebox{90}{$l_n=n^{0.4}$}}  & 100 &  & 0.224 & 0.494 & 0.931 &  & 0.281 & 0.573 & 0.943 &  & 0.144 & 0.186 & 0.282 &  & 0.199 & 0.290 & 0.430 \\ 
& 300 &  & 0.442 & 0.964 & 1.000 &  & 0.538 & 0.970 & 1.000 &  & 0.118 & 0.188 & 0.283 &  & 0.213 & 0.328 & 0.467 \\ 
& 500 &  & 0.716 & 1.000 & 1.000 &  & 0.788 & 1.000 & 1.000 &  & 0.120 & 0.176 & 0.294 &  & 0.223 & 0.362 & 0.497 \\ 
& 1000 &  & 0.992 & 1.000 & 1.000 &  & 0.996 & 1.000 & 1.000 &  & 0.128 & 0.193 & 0.300 &  & 0.244 & 0.391 & 0.531 \\ 
 &  &  &  &  &  &  &  &  &  &  &  &  &  &  &  &  \\ 
\multirow{4}{*}{\rotatebox{90}{$l_n=n^{0.5}$}}  & 100 &  & 0.220 & 0.478 & 0.921 &  & 0.273 & 0.543 & 0.932 &  & 0.140 & 0.184 & 0.283 &  & 0.185 & 0.273 & 0.417 \\ 
& 300 &  & 0.415 & 0.941 & 1.000 &  & 0.500 & 0.953 & 1.000 &  & 0.118 & 0.181 & 0.280 &  & 0.196 & 0.309 & 0.453 \\ 
& 500 &  & 0.679 & 0.999 & 1.000 &  & 0.740 & 0.999 & 1.000 &  & 0.116 & 0.175 & 0.292 &  & 0.211 & 0.347 & 0.487 \\ 
& 1000 &  & 0.984 & 1.000 & 1.000 &  & 0.988 & 1.000 & 1.000 &  & 0.126 & 0.187 & 0.296 &  & 0.229 & 0.374 & 0.516 \\ 
 &  &  &  &  &  &  &  &  &  &  &  &  &  &  &  &  \\ 
\multirow{4}{*}{\rotatebox{90}{$l_n=n^{0.6}$}}  & 100 &  & 0.229 & 0.487 & 0.914 &  & 0.273 & 0.541 & 0.922 &  & 0.148 & 0.188 & 0.298 &  & 0.188 & 0.276 & 0.413 \\ 
& 300 &  & 0.431 & 0.919 & 1.000 &  & 0.491 & 0.931 & 1.000 &  & 0.126 & 0.192 & 0.283 &  & 0.195 & 0.312 & 0.452 \\ 
& 500 &  & 0.658 & 0.995 & 1.000 &  & 0.717 & 0.996 & 1.000 &  & 0.122 & 0.179 & 0.300 &  & 0.214 & 0.343 & 0.485 \\ 
& 1000 &  & 0.969 & 1.000 & 1.000 &  & 0.971 & 1.000 & 1.000 &  & 0.128 & 0.190 & 0.299 &  & 0.224 & 0.372 & 0.510 \\ 
\end{tabular}
\end{center}
\label{table:wave}
\end{table}
\end{landscape}

\begin{landscape}
\begin{table}[htbp]
	\caption{Rejection rates of the hypothesis tests resulting from the empirical distance covariance  and the empirical covariance
				obtained by  subsampling   based on  \enquote{rectangularly} correlated  time series $X_j$, $j=1, \ldots, n$, $Y_j, \ j=1, \ldots, n$ according  to
Section \ref{subsec:simulations} 
with block length $l_n$,  $d=0.1n$, and Hurst parameters $H$.
				The level of significance equals 5\%}
\begin{center}
\begin{tabular}{cccccccccccccccccc}
 & & &   \multicolumn{7}{c}{distance covariance} &  & \multicolumn{7}{c}{covariance}\\
 				\cline{5-9}   \cline{13-17}\\
 &  &  & & $H=0.6$ &  &  &  & $H=0.7$ &  &  &  & $H=0.6$ &  &  &  & $H=0.7$ &  \\ 
				\cline{4-6}   \cline{8-10} \cline{12-14} \cline{16-18} \\
& $n$ &  & $v=1$ & $v=2$ & $v=3$ &  & $v=1$ & $v=2$ & $v=3$ &  & $v=1$ & $v=2$ & $v=3$ &  & $v=1$ & $v=2$ & $v=3$ \\ 
\hline
 &  &  &  &  &  &  &  &  &  &  &  &  &  &  &  &  \\ 
\multirow{4}{*}{\rotatebox{90}{$l_n=n^{0.4}$}}  & 100 &  & 0.273 & 0.448 & 0.381 &  & 0.391 & 0.620 & 0.570 &  & 0.090 & 0.059 & 0.038 &  & 0.141 & 0.093 & 0.072 \\ 
& 300 &  & 0.733 & 0.990 & 0.985 &  & 0.819 & 0.995 & 0.994 &  & 0.078 & 0.037 & 0.023 &  & 0.114 & 0.070 & 0.050 \\ 
&500 &  & 0.966 & 1.000 & 1.000 &  & 0.980 & 1.000 & 1.000 &  & 0.065 & 0.034 & 0.017 &  & 0.104 & 0.061 & 0.036 \\ 
& 1000 &  & 1.000 & 1.000 & 1.000 &  & 1.000 & 1.000 & 1.000 &  & 0.058 & 0.025 & 0.014 &  & 0.091 & 0.053 & 0.036 \\ 
 &  &  &  &  &  &  &  &  &  &  &  &  &  &  &  &  \\ 
\multirow{4}{*}{\rotatebox{90}{$l_n=n^{0.5}$}}  & 100 &  & 0.254 & 0.425 & 0.366 &  & 0.343 & 0.551 & 0.506 &  & 0.092 & 0.063 & 0.042 &  & 0.131 & 0.091 & 0.065 \\ 
& 300 &  & 0.663 & 0.972 & 0.962 &  & 0.720 & 0.977 & 0.978 &  & 0.076 & 0.038 & 0.025 &  & 0.104 & 0.059 & 0.045 \\ 
& 500 &  & 0.932 & 0.999 & 1.000 &  & 0.948 & 1.000 & 1.000 &  & 0.066 & 0.034 & 0.019 &  & 0.093 & 0.053 & 0.030 \\ 
&1000 &  & 1.000 & 1.000 & 1.000 &  & 1.000 & 1.000 & 1.000 &  & 0.055 & 0.025 & 0.015 &  & 0.077 & 0.042 & 0.028 \\ 
 &  &  &  &  &  &  &  &  &  &  &  &  &  &  &  &  \\ 
\multirow{4}{*}{\rotatebox{90}{$l_n=n^{0.6}$}}  &  100 &  & 0.266 & 0.437 & 0.399 &  & 0.334 & 0.530 & 0.498 &  & 0.113 & 0.079 & 0.058 &  & 0.138 & 0.102 & 0.078 \\ 
& 300 &  & 0.646 & 0.946 & 0.937 &  & 0.663 & 0.942 & 0.950 &  & 0.090 & 0.052 & 0.038 &  & 0.106 & 0.071 & 0.052 \\ 
& 500 &  & 0.893 & 0.998 & 0.997 &  & 0.902 & 0.997 & 0.997 &  & 0.075 & 0.042 & 0.028 &  & 0.091 & 0.055 & 0.034 \\ 
& 1000 &  & 1.000 & 1.000 & 1.000 &  & 0.998 & 1.000 & 1.000 &  & 0.060 & 0.034 & 0.022 &  & 0.075 & 0.046 & 0.034 \\ 
\end{tabular}
\end{center}
\label{table:rectangular}
\end{table}
\end{landscape}
\end{appendix}

\end{document}